\documentclass[11pt,british]{article}
\usepackage[T1]{fontenc}
\usepackage[latin9]{inputenc}
\usepackage{babel}
\usepackage{enumitem}
\usepackage{amsmath}
\usepackage{amsthm}
\usepackage{amssymb}
\usepackage[numbers]{natbib}
\usepackage[unicode=true,pdfusetitle,
 bookmarks=true,bookmarksnumbered=false,bookmarksopen=false,
 breaklinks=false,pdfborder={0 0 1},backref=false,colorlinks=false]
 {hyperref}

\makeatletter
\theoremstyle{plain}
\newtheorem{thm}{\protect\theoremname}
\theoremstyle{definition}
\newtheorem{example}[thm]{\protect\examplename}
\theoremstyle{definition}

\theoremstyle{remark}
\newtheorem{rem}[thm]{\protect\remarkname}
\theoremstyle{plain}
\newtheorem{prop}[thm]{\protect\propositionname}
\theoremstyle{plain}
\newtheorem{cor}[thm]{\protect\corollaryname}
\theoremstyle{plain}
\newtheorem{lem}[thm]{\protect\lemmaname}

\usepackage{geometry}
\newcommand{\rrvert}{\vert}
\newcommand{\llvert}{\vert}

\theoremstyle{definition}
\newtheorem{myassumption}{Assumption}

\makeatother

\addto\captionsbritish{\renewcommand{\conditionname}{Condition}}
\addto\captionsbritish{\renewcommand{\corollaryname}{Corollary}}
\addto\captionsbritish{\renewcommand{\examplename}{Example}}
\addto\captionsbritish{\renewcommand{\lemmaname}{Lemma}}
\addto\captionsbritish{\renewcommand{\propositionname}{Proposition}}
\addto\captionsbritish{\renewcommand{\remarkname}{Remark}}
\addto\captionsbritish{\renewcommand{\theoremname}{Theorem}}
\addto\captionsenglish{\renewcommand{\conditionname}{Condition}}
\addto\captionsenglish{\renewcommand{\corollaryname}{Corollary}}
\addto\captionsenglish{\renewcommand{\examplename}{Example}}
\addto\captionsenglish{\renewcommand{\lemmaname}{Lemma}}
\addto\captionsenglish{\renewcommand{\propositionname}{Proposition}}
\addto\captionsenglish{\renewcommand{\remarkname}{Remark}}
\addto\captionsenglish{\renewcommand{\theoremname}{Theorem}}
\providecommand{\conditionname}{Condition}
\providecommand{\corollaryname}{Corollary}
\providecommand{\examplename}{Example}
\providecommand{\lemmaname}{Lemma}
\providecommand{\propositionname}{Proposition}
\providecommand{\remarkname}{Remark}
\providecommand{\theoremname}{Theorem}

\numberwithin{equation}{section}
\numberwithin{thm}{section}

\begin{document}
\global\long\def\epsilon{\varepsilon}%

\global\long\def\E{\mathbb{E}}%

\global\long\def\I{\mathbf{1}}%

\global\long\def\N{\mathbb{N}}%

\global\long\def\R{\mathbb{R}}%

\global\long\def\C{\mathbb{C}}%

\global\long\def\Q{\mathbb{Q}}%

\global\long\def\P{\mathbb{P}}%

\global\long\def\D{\Delta_{n}}%

\global\long\def\dom{\operatorname{dom}}%

\global\long\def\b#1{\mathbb{#1}}%

\global\long\def\c#1{\mathcal{#1}}%

\global\long\def\s#1{{\scriptstyle #1}}%

\global\long\def\u#1#2{\underset{#2}{\underbrace{#1}}}%

\global\long\def\r#1{\xrightarrow{#1}}%

\global\long\def\mr#1{\mathrel{\raisebox{-2pt}{\ensuremath{\xrightarrow{#1}}}}}%

\global\long\def\t#1{\left.#1\right|}%

\global\long\def\l#1{\left.#1\right|}%

\global\long\def\f#1{\lfloor#1\rfloor}%

\global\long\def\sc#1#2{\langle#1,#2\rangle}%

\global\long\def\abs#1{\lvert#1\rvert}%

\global\long\def\bnorm#1{\Bigl\lVert#1\Bigr\rVert}%

\global\long\def\wraum{(\Omega,\c F,\P)}%

\global\long\def\fwraum{(\Omega,\c F,\P,(\c F_{t}))}%

\global\long\def\norm#1{\lVert#1\rVert}%

\date{}

\author{Randolf Altmeyer \\ \textit{ Imperial College London \footnote{Department of Mathematics, London, UK. Email: r.altmeyer@imperial.ac.uk, \newline The author is grateful for helpful comments from the Associate Editor and two anonymous referees. The author thanks Richard Nickl for many helpful discussions and comments. }}}

\title{Polynomial time guarantees for sampling based posterior inference in high-dimensional generalised linear models}

\maketitle

\begin{abstract}
   The problem of computing posterior functionals in general high-dimensional
statistical models with possibly non-log-concave likelihood functions is
considered. Based on the proof strategy of Nickl and Wang (2022), 
but using
only local likelihood conditions and without relying on M-estimation theory,
nonasymptotic statistical and computational guarantees are provided for
a gradient based MCMC algorithm. Given a suitable initialiser, these guarantees
scale polynomially in key algorithmic quantities. The abstract results
are applied to several concrete statistical models, including density estimation,
nonparametric regression with generalised linear models and a canonical
statistical non-linear inverse problem from PDEs.
\end{abstract}

\medskip



\section{Introduction}
\label{sec1}

Posterior inference for high-dimensional statistical models is increasingly
important in contemporary applications, particularly in the physical sciences
and in engineering \citep{reich2015,stuart2010,ghanem2017}. Computing relevant
functionals such as the posterior mean, mode or quantiles often relies
on iterative sampling algorithms. Without additional structural assumptions,
however, the mixing times of these algorithms can scale exponentially in
the model dimension $p$ or the sample size $n$
\citep{eberle2016,tzen2018}. In this case, valid inference on an underlying
ground truth requiring $p\asymp n^{\rho}$, $\rho >0$, is intractable. Overcoming
such computational hardness barriers is crucial to allow for efficient
sampling based Bayesian procedures.

A canonical sampling approach uses Markov chain Monte Carlo (MCMC) algorithms
\citep{robert1999}. These generate a specifically designed Markov chain
$(\vartheta _{k})_{k=1}^{\infty}$, whose marginal laws
$\mathcal{L}(\vartheta _{k})$ approximate up to a target precision level
the posterior distribution with probability density

\begin{equation}
\pi \bigl(\theta |Z^{(n)}\bigr)\propto e^{\ell _{n}(\theta )}\pi (\theta ),
\quad \theta \in \mathbb{R}^{p}. \label{eq:posteriorDensity}
\end{equation}

Here, the data $Z^{(n)}$ are observations in a statistical model depending
on a ground truth $\theta _{0}\in \ell ^{2}(\mathbb{N})$ with log-likelihood
function $\ell _{n}$ and prior density $\pi $. Implementations of concrete
posterior distributions using MCMC sampling algorithms have been discussed
in various works \citep{lenk1988,lenk1991,tokdar2007,mariucci2020}, but
computational guarantees are only rarely addressed. The important results
of \citep{hairer2014,belloni2009} show that sampling at polynomial cost
is possible, in principle, but their assumptions are rather restrictive
and not explicit in their quantitative dependence on $n$ and $p$, see
\citep{nickl2020a} for a discussion and additional references. To break
the curse of dimensionality, and often also the
\emph{curse of non-linearity}, other sampling approaches replace the complex
posterior measure by a more simple object
\citep{ray2021,vanleeuwen2019,rue2009}, yielding empirically efficient
procedures, but with unclear relation to the true posterior measure.

Suppose that the computational complexity of some MCMC algorithm arises
mainly from the number of iterations. In this case, existing guarantees
for practically feasible mixing times, growing at most polynomially in
$n$, $p$, are essentially limited to strongly concave
$\log \pi (\cdot |Z^{(n)})$ with Lipschitz-gradients
\citep{dalalyan2017,dwivedi2018,lovasz2007,wu2022minimax}. These two properties
are not satisfied for many relevant posteriors, even with Gaussian priors.
We discuss this issue for generalised linear models (GLMs)
\citep{nelder1972} (see Examples \ref{exa:GLM} and \ref{exa:negative} below),
but it applies more broadly to mixture models, density estimation and statistical
non-linear inverse problems
\citep{monard2021a,monard2019,giordano2020b,abraham2020,nickl2020b}, among
others. The lack of conjugacy in these models seems to explain why sampling
is difficult, but also conjugate models can suffer from multimodality and
slow mixing \citep{hastie2015sampling}. For related discussions on MCMC
in different statistical settings and on closely related optimisation algorithms
see
\citep{arous2020,rebeschini2015,talwar2019,kunisky2022,ma2019,raginsky2017,belloni2009,chewi2022}.

In a recent contribution Nickl and Wang \citep{nickl2020a} obtain polynomial
time sampling guarantees in a specific example involving a partial differential
equation (PDE) and a Gaussian process prior. To go beyond the non-concave
setting, the key idea, which was later extended by \citep{bohr2021a} to
other PDEs, is to rely on the Fisher information for providing a natural
statistical notion of curvature for the log-likelihood function near
$\theta _{0}$. By combining empirical process techniques with tools from
Bayesian nonparametrics \citep{ghosal2017}, there exists a high-dimensional
region $\mathcal{B}\subset \mathbb{R}^{p}$ of parameters near
$\theta _{0}$, where the posterior measure concentrates most of its mass
and where $\ell _{n}$ is \emph{locally} strongly concave with high probability.
Convexifying~$-\ell _{n}$ yields a surrogate posterior measure, whose log-density
$\log \tilde{\pi}(\cdot |Z^{(n)})$ is \emph{globally} strongly concave with
Lipschitz-gradients, and which is close to the true posterior measure in
Wasserstein distance with high probability. Given a problem-specific initialiser
$\theta _{\operatorname{init}}\in \mathbb{R}^{p}$ to identify the region
$\mathcal{B}$ in a data-driven fashion,
$\tilde{\pi}(\cdot |Z^{(n)})$ can be leveraged to generate approximate
samples of the posterior by a gradient-based Langevin MCMC-scheme. This
yields nonasymptotic sampling guarantees for posterior functionals with
polynomial dependence of key algorithm-specific parameters
\emph{simultaneously} on $n$, $p$ and the target precision level, leveraging
recent results by Durmus and Moulines \citep{durmus2019}.

In this work, we extend the proof strategy of \citep{nickl2020a} beyond
the PDE setting to general high-dimensional statistical models assuming
only local likelihood conditions, thus demonstrating the feasibility of
efficient sampling based Bayesian inference. The local likelihood conditions
are verified for density estimation and GLMs, comprising important statistical
regression models such as Gaussian, Poisson and logistic regression. To
showcase our approach in a canonical example from the inverse problem literature,
we discuss in detail sampling for an elliptic PDE, sometimes called
\emph{Darcy's problem} (see \citep{nickl2022} and Section~3.7 of
\citep{stuart2010}),\emph{ }with a non-Lipschitz forward map. Moreover,
our approach allows for extending the results of
\citep{nickl2020a,bohr2021a} to non-Gaussian measurement errors.

We further prove that the Langevin Markov chain based on the surrogate
density takes exponentially in $n$ many steps to leave the region of local
curvature, where it coincides with $\pi (\cdot |Z^{(n)})$. Our results
therefore imply that upon initialising into a region of sufficient local
curvature even a standard vanilla Langevin MCMC algorithm is able to compute
posterior aspects at polynomial cost. This is consistent with related results
for gradient based optimisation algorithms showing that local curvature
near the global optimum can improve the rate of convergence
\citep{bach2014}. Sampling algorithms, on the other hand, necessarily have
to explore the full parameter space and therefore depend more heavily on
global properties of the underlying target distribution. Note that, even
if an initialiser near $\theta _{0}$ is available, it is not clear if the
computation of posterior functionals, which depend on the whole posterior
measure, is feasible. The existence of a suitable initialiser is postulated
here, and finding one in polynomial time may be in itself a non-trivial
task. As an illustration we prove the existence of a suitable initialiser
for GLMs.

Our proofs use specific results for unadjusted Langevin MCMC-schemes, but
extensions to other sampling algorithms seem possible as long as they satisfy
sampling guarantees for log-concave target distributions. The log-concave
approximation of the posterior measure in
\citep{nickl2020a,bohr2021a} relies on expanding
$\log \pi (\cdot |Z^{(n)})$ near the maximum-a-posteriori (MAP) estimator.
The analysis of the latter is based on M-estimation theory
\cite{vandegeer2007} and seems to preclude, for example, regression models
with unbounded and possibly non-Lipschitz regression functions, which are,
however, present in most of the applications studied below. Instead, our
proof for the log-concave approximation is fully Bayesian and takes a novel
route by establishing first contraction of the surrogate posterior and
then expanding $\log \tilde\pi (\cdot |Z^{(n)})$ near $\theta _{0}$.

Conceptually, the approximation by a log-concave measure is different from
the more traditional Gaussian Laplace-type approximations
\citep{schillings2020,helin2022,rue2009,belloni2009}, where the posterior
is replaced by a quadratic with constant covariance matrix relative to
a well-chosen centring point (often the MAP estimator). In contrast, we
leave the log-likelihood function locally unchanged. This added flexibility
seems to be crucial for obtaining the fast convergence towards the posterior
measure in our results. Related to this are Bernstein-von Mises theorems,
which generally do not hold in high-dimensional settings
\citep{castillo2014a,freedman1999}. In particular, \citep{nickl2022a} have
recently shown that no such theorem exists for Darcy's problem, while we
show that efficient MCMC-sampling is possible (see also
\citep[Remark~5.4.2]{nickl2022}).

The rest of the paper is organised as follows. Section~\ref{sec:mainResults} develops the approximation by the surrogate posterior
measure in a general context and presents convergence guarantees for the
surrogate and vanilla Langevin sampler. In Section~\ref{sec:Applications}, applications to several statistical models are
discussed in detail, including density estimation, nonparametric regression
with GLMs and Darcy's problem.

We write $a\lesssim b$ if $a\leq Cb$ for a universal constant $C$, and
$a\asymp b$ if $a\lesssim b$ and $b\lesssim a$. For a measurable space
$(\mathcal{X},\mathcal{A})$, equipped with a measure
$\nu _{\mathcal{X}}$, let $L^{p}(\mathcal{X})$,
$1\leq p\leq \infty $, be the spaces of $p$-integrable $\mathcal{A}$-measurable
functions with respect to $\nu _{\mathcal{X}}$ normed by
$\lVert \cdot \rVert _{L^{p}}$. Denote by
$\ell ^{\infty}(\mathbb{N})$ the usual sequence spaces with norm
$\lVert \cdot \rVert _{\ell ^{p}}$. Set
$\lVert \cdot \rVert =\lVert \cdot \rVert _{\ell ^{2}}$. For a matrix
$M\in \mathbb{R}^{p\times p}$ let
$\lVert M\rVert _{\operatorname{op}}$ be the operator norm. The minimal
and maximal eigenvalues of a positive symmetric matrix $\Sigma $ are denoted
by $\lambda _{\operatorname{min}}(\Sigma )$,
$\lambda _{\operatorname{max}}(\Sigma )$. For two Borel probability measures
$\mu _{1}$, $\mu _{2}$ on $\mathbb{R}^{p}$ with finite second moments the
(squared) Wasserstein distance is defined as
$W_{2}^{2}(\mu _{1},\mu _{2})=\operatorname{\inf }\int _{
\mathbb{R}^{p}\times \mathbb{R}^{p}}\lVert \theta -\theta '\rVert ^{2}d
\mu (\theta ,\theta ')$, where the infimum is computed over all couplings
$\mu $ of $\mu _{1}$, $\mu _{2}$. If $\mathcal{X}$ is a smooth domain in
$\mathbb{R}^{d}$, $d\in \mathbb{N}$, then denote by
$C^{k}(\mathcal{X})$, $0\leq k\leq \infty $, the spaces of $k$-times differentiable
real-valued functions. For a real-valued function
$(\theta ,x)\mapsto h(\theta ,x)$ defined on a subset of
$\mathbb{R}^{p}\times \mathcal{X}$ and considered as a map on
$\mathbb{R}^{p}$ for fixed $x$ its gradient and Hessian if they exist are
denoted by $\nabla h=\nabla _{\theta}h$,
$\nabla ^{2}h=\nabla ^{2}_{\theta}h$, respectively. A~function
$h:\mathbb{R}^{p}\rightarrow \mathbb{R}$ is Lipschitz if the norm $
\lVert h\rVert _{\mathrm{Lip}}=\sup_{x,y\in \mathbb{R}^{p},x
\neq y}\frac{ \llvert h(x)-h(y) \rrvert }{\lVert x-y\rVert }$ is finite.

\section{General sampling guarantees}\label{sec:mainResults}

In this section our goal is to obtain polynomial time sampling guarantees
for high dimensional posteriors under general high level assumptions on
the statistical model. Let
$\Theta \subset \ell ^{\infty}(\mathbb{N})$ be a parameter space containing
$\mathbb{R}^{p}$ and let $Z^{(n)}=(Z_{i})_{i=1}^{n}$ be $n$ independent
observations with values in $\mathcal{Z}^{n}$, where
$(\mathcal{Z},\mathcal{O})$ is a measurable space, drawn from a product
measure
$\mathbb{P}_{\theta _{0}}^{n}=\otimes _{i=1}^{n}\mathbb{P}_{\theta}$ for
an unknown parameter $\theta _{0}\in \Theta $. Denote by
$p_{\theta}$ the probability density of $\mathbb{P}_{\theta}$ for
$\theta \in \Theta $ with respect to a dominating measure $\nu $. The\emph{
}log-likelihood function is
\begin{equation*}
\ell _{n}(\theta )\equiv \ell _{n}\bigl(\theta
,Z^{(n)}\bigr)=\sum_{i=1}^{n} \log
p_{\theta}(Z_{i})=\sum_{i=1}^{n}
\ell (\theta ,Z_{i}).
\end{equation*}
The Bayesian approach assumes $\theta \sim \Pi $ for a prior probability
measure $\Pi \equiv \Pi _{n}$ on $\Theta $ which may depend on $n$. The
posterior distribution $\Pi (\cdot |Z^{(n)})$ arises from Bayes' formula.
It is common to consider as priors infinite random series
on~$\Theta $ with decaying coefficients such as Gaussian process or Laplace
priors \cite{vandervaart2008,agapiou2021a}. They allow for modelling parameters
$\theta _{0}$ with Sobolev-type or Besov-type regularities (as in
Theorem~\ref{thm:GaussianContraction} below). In order to sample from the posterior
we need to use a finite dimensional parameter space, for example relative
to the first $p$ components of $\theta _{0}$. We therefore suppose that
$\Pi $ is supported on $\mathbb{R}^{p}$ and that it has a Lebesgue density
$\pi \equiv \pi _{n}$, so the posterior $\Pi (\cdot |Z^{(n)})$ has the
probability density
\begin{align}
\pi \bigl(\theta |Z^{(n)}\bigr) & = \frac{\prod_{i=1}^{n}p_{\theta}(Z_{i})\pi (\theta )}
{\int _{\Theta}\prod_{i=1}^{n}p_{\theta}(Z_{i})\pi (\theta )\,d\theta} \propto
e^{\ell _{n}(\theta )+\log \pi (\theta )},\quad \theta \in \mathbb{R}^{p}.
\label{eq:posteriorDensity_normalised}
\end{align}

\subsection{Strong concavity and generalised linear models}
\label{sec:setup}

A function $f:\mathbb{R}^{p}\rightarrow \mathbb{R}$ is strongly convex,
or equivalently $-f$ is strongly concave, and has Lipschitz gradients if
for curvature and Lipschitz constants $m_{f},\Lambda _{f}>0$

\begin{align}
\begin{cases} f(\theta ) \geq f\bigl(\theta '\bigr)+
\bigl(\theta -\theta '\bigr)^{\top}\nabla f\bigl(\theta
'\bigr)+ \frac{m_{f}}{2}\bigl\lVert \theta -\theta '
\bigr\rVert ^{2},
\\
\bigl\lVert \nabla f(\theta )-\nabla f\bigl(\theta '\bigr)\bigr
\rVert \leq \Lambda _{f} \bigl\lVert \theta -\theta '
\bigr\rVert , \end{cases}\quad \text{and all }\theta ,\theta '\in
\mathbb{R}^{p}. \label{eq:stronglyConvex} 
\end{align}

For twice continuously differentiable $f$ this is equivalent to

\begin{align*}
m_{f}=\inf_{\theta \in \mathbb{R}^{p}}\lambda _{\operatorname{min}}\bigl(
\nabla ^{2} f(\theta )\bigr)>0,\qquad \Lambda _{f} = \sup
_{\theta \in
\mathbb{R}^{p}}\lambda _{\operatorname{min}}\bigl(\nabla ^{2} f(
\theta )\bigr)>0.
\end{align*}
As a motivating example let us discuss for the generalised linear models
introduced by \cite{nelder1972} whether
$f=-\log \pi (\cdot |Z^{(n)})$ satisfies \eqref{eq:stronglyConvex}. When
$\ell _{n}$ is itself concave, then strong concavity can be induced by
the prior. This idea works for certain GLMs with Gaussian priors, see also
the numerical examples in \citep{dalalyan2017,wu2022minimax}.
\begin{example}[Generalised linear models]%
\label{exa:GLM}
Set $\nu =\xi \otimes \nu _{\mathcal{X}}$, where $\xi $ is a Borel measure
on $\mathbb{R}$. Given an orthonormal basis $(e_{k})_{k\geq 1}$ of
$L^{2}(\mathcal{X})$, a parameter
$\theta \in \Theta \subset \ell ^{2}(\mathbb{N})$ is identified with a
function in $L^{2}(\mathcal{X})$ by the isometry
$\Phi (\theta )=\sum_{k=1}^{\infty}\theta _{k}e_{k}$. Let
$g:\mathcal{I}\rightarrow \mathbb{R}$ be an invertible and continuous
\emph{link} function, defined on some interval
$\mathcal{I}\subset \mathbb{R}$, and suppose that
$Z_{i}=(Y_{i},X_{i})$ takes values in $\mathbb{R}\times \mathcal{X}$ such
that the law of the response variables $Y_{i}$ follows conditional on
$X_{i}=x$ a one-parameter exponential family with
$g(\mathbb{E}_{\theta}[Y_{i}|X_{i}])=\Phi (\theta )(X_{i})$. If
$p_{X}$ denotes the $\nu _{\mathcal{X}}$-density of the covariates
$X_{i}$, then this means that the coordinate $\nu $-densities
$p_{\theta}$ are of the form
\begin{equation*}
p_{\theta}(y,x)=\exp \bigl(yb_{\theta}(x)-A \bigl(b_{\theta}(x)
\bigr) \bigr)p_{X}(x),\quad y\in \mathbb{R},x\in \mathcal{X},
\end{equation*}
with a convex function $A:\mathbb{R}\mapsto [0,\infty )$ and
$b_{\theta}(x)=(A')^{-1}  (g^{-1}  (\Phi (\theta )(x)  )
  )$. So for $\theta \in \mathbb{R}^{p}$
\begin{align}
\nabla ^{2}\ell _{n}(\theta ) &= \sum
_{i=1}^{n} \bigl(\bigl(Y_{i}-A'
\bigl(b_{
\theta}(X_{i})\bigr)\bigr)\nabla ^{2}
b_{\theta}(X_{i})- A''
\bigl(b_{\theta}(X_{i})\bigr) \nabla b_{\theta}(X_{i})
\nabla b_{\theta}(X_{i})^{\top} \bigr). \label{eq:GLM_hessian}
\end{align}
With the canonical link $g=(A')^{-1}$, the map
$\theta \mapsto b_{\theta}=\Phi (\theta )$ is linear, so
$\nabla b_{\theta}$ does not depend on $\theta $,
$\nabla ^{2} b_{\theta}=\nabla ^{2}\Phi (\theta )=0$ and
$-\ell _{n}(\theta )$ is convex because $A$ is. For a Gaussian prior we
have
$\log \pi (\theta )=-\lVert \Sigma ^{1/2}_{\pi }\theta \rVert ^{2}/2$ with
a positive definite $\Sigma _{\pi}\in \mathbb{R}^{p\times p}$, so
$m_{f}=\lambda _{\operatorname{min}}(\Sigma _{\pi})$ and
$\Lambda _{f}=\lambda _{\operatorname{max}}(\Sigma _{X})\sup_{x\in
\mathbb{R}^{p}}A''(x)+\lambda _{\operatorname{max}}(\Sigma _{\pi})$,
$\Sigma _{X}=\sum_{i=1}^{n}\nabla b_{\theta}(X_{i})\nabla b_{\theta}(X_{i})^{
\top}$. Note that $A''$ is bounded by 1 for Gaussian and logistic regression
with $A(x)=x^{2}/2$ and $A(x)=\log (1+e^{x})$, respectively.
\end{example}

For non-concave $\ell _{n}$ adding a strongly concave prior is not enough
to ensure that the log-posterior density is strongly concave. Apart from the basic cases in the first example its gradients are usually
not Lipschitz.

\begin{example}%
\label{exa:negative}
When $g$ is not the canonical link in the example above such as in probit
regression, then $\theta \mapsto b_{\theta}$ is non-linear and
$-\ell _{n}$ is usually not convex. For Poisson regression with
$A(x)=e^{x}$, even with the canonical link $g(x)=\log x$ the Lipschitz
constant of $\theta \mapsto \nabla \ell _{n}(\theta )$ grows exponentially
as $\lVert \theta \rVert \rightarrow \infty $. Other examples for non-convex
$-\ell _{n}$ appear in non-linear inverse problems where $\Phi $ is replaced
with a forward map corresponding to the solution of a partial differential
equation (see Section~\ref{subsec:Darcy's-problem-1} for concrete cases).
\end{example}

In order to motivate how strong concavity can be relaxed note in
\eqref{eq:GLM_hessian} that
\begin{align*}
-\mathbb{E}^{n}_{\theta _{0}}\nabla ^{2}\ell
_{n}(\theta _{0}) &= n \mathbb{E} \bigl[A''
\bigl(b_{\theta _{0}}(X_{1})\bigr)\nabla b_{\theta _{0}}(X_{1})
\nabla b_{\theta _{0}}(X_{1})^{\top} \bigr].
\end{align*}
Assuming the design density $p_{X}$ is lower bounded, this suggests heuristically
that the eigenvalues of $-\nabla ^{2}\ell _{n}(\theta )$ are contained
in an interval $[c_{1}n,c_{2}n]$ for constants $c_{1},c_{2}>0$ for all
$\theta $ sufficiently close to the truth $\theta _{0}$, while also assuming
that $-\nabla ^{2}\ell _{n}(\theta )$ is concentrated with high
$\mathbb{P}^{n}_{\theta _{0}}$-probability around its expectation uniformly
for all such $\theta $. It is therefore reasonable to expect that
$\log \pi (\cdot |Z^{(n)})$ is locally concave and has Lipschitz gradients
near $\theta _{0}$, but generally the curvature and Lipschitz constants
will degenerate as the region around $\theta _{0}$ grows.

From a statistical point of view this heuristic is justified, because the
expectation of $-\nabla ^{2}\ell _{n}(\theta _{0})$ is closely related
to the Fisher information of the underlying statistical model. A~non-degenerate
Fisher information is the basis of the Fisher-scoring method for computing
the maximum likelihood estimator by gradient descent in finite dimensional
models \cite{mccullagh2019}. While the Fisher information may be degenerate
in infinite dimensional models \cite{nickl2022a}, we verify quantitative
curvature bounds on
$-\mathbb{E}^{n}_{\theta _{0}}\nabla ^{2}\ell _{n}(\theta )$, depending
on $n$ and $p$, for several important cases in Section~\ref{sec:Applications}.

\subsection{Main assumptions}
\label{subsec:Main-assumptions}

Since the posterior density \eqref{eq:posteriorDensity} is generally not
log-concave, we proceed by constructing a log-concave approximation and
a suitable sampling algorithm. As explained above, this approximation is
guided by the statistical model itself. We begin by stating the necessary
assumptions for the log-concave approximation. For a radius
$\eta >0$ define the high-dimensional region
\begin{equation}
\mathcal{B}= \bigl\{ \theta \in \mathbb{R}^{p}:\lVert \theta -\theta
_{*,p} \rVert \leq \eta \bigr\} , \label{eq:GLM_B} 
\end{equation}
where $\theta _{*,p}\in \mathbb{R}^{p}$ is an approximation of
$\theta _{0}$, for example its $\mathbb{R}^{p}$-projection
$(\theta _{0,1},\dots ,\theta _{0,p})$. We require quantitative control
on the regularity and curvature of the log-likelihood on
$\mathcal{B}$, relative to the sample size $n$, the model dimension
$p$ and a `statistical accuracy' $\delta _{n}>0$, which will be specified
below.
\renewcommand{\themyassumption}{A\arabic{myassumption}}
\begin{myassumption}[regularity]%
\label{assu:a1}%
The data $Z^{(n)}$ arise from the law $\mathbb{P}_{\theta _{0}}^{n}$ for
a fixed $\theta _{0}\in \Theta $. There exists an event
$\mathcal{E}$ and constants $c_{1},c_{2}>0$ with
$\mathbb{P}_{\theta _{0}}^{n}(\mathcal{E})\geq 1-c_{2}e^{-c_{1}n
\delta _{n}^{2}}$ such that for some $c_{\max}\geq c_{\min}>0$,
$\kappa _{1},\kappa _{2},\kappa _{3}\geq 0$ the following holds on
$\mathcal{E}$:
\begin{enumerate}[label=(\roman*)]
\item[{(i)}] (local regularity)
$\theta \mapsto \ell _{n}(\theta )\in C^{2}(\mathcal{B})$.
\item[{(ii)}] (local boundedness)
$\lVert \nabla \ell _{n}(\theta _{*,p})\rVert \leq c_{\max}n\delta _{n}p^{
\kappa _{1}}$,
$\sup_{\theta \in \mathcal{B}}\lVert \nabla ^{2}\ell _{n}(\theta )
\rVert _{\operatorname{op}}\leq c_{\max}np^{\kappa _{2}}$.
\item[{(iii)}] (local curvature)
$\inf_{\theta \in \mathcal{B}}\lambda _{\min}  (-\nabla ^{2}
\ell _{n}(\theta )  )\geq c_{\min}np^{-\kappa _{3}}$.
\end{enumerate}
Furthermore,
$\eta \geq (4c_{\operatorname{max}}/c_{\operatorname{min}})
\tilde{\delta}_{n,p}\log (n+e^{8})$, with the
\emph{curvature adjusted rate}
$\tilde{\delta}_{n,p}=\max (\delta _{n}^{1/\beta},\delta _{n}p^{
\kappa _{1}+\kappa _{3}})$, $\beta \geq 1$.
\end{myassumption}

In order to access the region $\mathcal{B}$ in a data-driven fashion we
expect to have access to some
$\theta _{\operatorname{init}}\in \mathcal{B}$, which will also serve as
the initialiser of the MCMC scheme later on.

\begin{myassumption}[initialiser]%
\label{assu:a2}
$\lVert \theta _{\operatorname{init}}-\theta _{*,p}\rVert \leq \eta /8$.
\end{myassumption}

The behaviour of $\log \pi (\cdot |Z^{(n)})$ outside the region
$\mathcal{B}$ will be restricted by requiring that the posterior distribution
contracts around $\theta _{*,p}$ at rate $\delta _{n}$. For
$c,L,\beta >0$ define the events
\begin{align}
\mathcal{D}_{n}(c,L) &= \bigl\{\Pi \bigl(\theta \in
\mathbb{R}^{p}:\lVert \theta -\theta _{*,p}\rVert
^{\beta }>L\delta _{n} \big|Z^{(n)} \bigr)> e^{-cn\delta _{n}^{2}}
\bigr\}, \label{eq:DncL} 
\\
\bar{\mathcal{D}}_{n}(c) &= \biggl\{
\int _{\lVert \theta -\theta _{*,p}
\rVert \leq \delta _{n}}e^{\ell _{n}(\theta )-\ell _{n}(\theta _{0})} \pi (\theta )\,d\theta \leq
e^{-cn\delta _{n}^{2}} \biggr\}. \label{eq:barDnc}
\end{align}

\begin{myassumption}[posterior contraction]%
\label{assu:a3}%
Suppose that $\delta _{n}\rightarrow 0$ with
$n\delta _{n}^{2}\rightarrow \infty $ as $n\rightarrow \infty $. Let
$\beta \geq 1$ and suppose that
$\lVert \theta _{*,p}\rVert \leq c_{0}$ for $c_{0}>0$. There exist first
for any $c>0$ and any sufficiently large $L$ constants
$C_{1},C'_{1}>0$, and second constants $C_{2},C_{2}',C_{3}>0$ such that
\begin{align*}
\mathbb{P}_{\theta _{0}}^{n} \bigl(\mathcal{D}_{n}(c,L)
\bigr) & \leq C'_{1}e^{-C_{1}n\delta _{n}^{2}},\qquad
\mathbb{P}_{\theta _{0}}^{n} \bigl(\bar{\mathcal{D}}_{n}(C_{3})
\bigr) \leq C'_{2}e^{-C_{2}n
\delta _{n}^{2}}.
\end{align*}
\end{myassumption}

Both conditions in Assumption~\ref{assu:a3} can be verified by standard
tools from Bayesian nonparametrics, following the seminal work
\citep{ghosal2000a}, with the event $\bar{\mathcal{D}}_{n}(c)$ being a
small ball condition related to the normalising factors in
\eqref{eq:posteriorDensity_normalised}. A~standard approach is to show
first posterior contraction around $\theta _{0}$ in Hellinger distance
(cf. the proof of Theorem~\ref{thm:GaussianContraction}), so Assumption~\ref{assu:a3}
implicitly encodes the bias condition that
$\theta _{*,p}$ is $\delta _{n}$-close to $\theta _{0}$. Informally speaking,
posterior contraction ensures that $\theta \in \mathbb{R}^{p}$ with
$\lVert \theta -\theta _{*,p}\rVert ^{\beta}>L\delta _{n}$ are negligible
for posterior inference and therefore can be safely ignored by the sampling
algorithm. Note that $\eta /\delta _{n}\rightarrow \infty $ as
$n\rightarrow \infty $ by the lower bound on $\eta $ in Assumption~\ref{assu:a1},
so the posterior concentrates most of its mass on the set
of good curvature $\mathcal{B}$. Finally, the prior is as follows.

\begin{myassumption}[prior]%
\label{assu:a4}
The negative prior log-density $-\log \pi $ satisfies
\eqref{eq:stronglyConvex} with curvature and Lipschitz constants
$m_{\pi}$, $\Lambda _{\pi}$. The unique maximiser
$\theta _{\pi ,\operatorname{max}}\in \mathbb{R}^{p}$ of $\pi $ satisfies
$\lVert \theta _{\pi ,\operatorname{max}}\rVert \leq c_{0}$. Moreover,
$\sup_{n,p}\int _{\mathbb{R}^{p}}\lVert \theta \rVert ^{4}\pi (
\theta )\,d\theta <\infty $.

\end{myassumption}

Concrete examples satisfying this assumption are Gaussian priors and finite-dimensional
approximations of `$p$-exponential' priors
\citep{agapiou2021a,agapiou2024laplace}.

\subsection{Log-concave approximation
of the posterior}
\label{subsec:Log-concave-approximation-of}

If the prior is supported on all of $\mathbb{R}^{p}$ (as it is in the Gaussian
case), then so is the posterior, and it is necessary for a sampling algorithm
to explore the whole support. Given that the negative log-likelihood is
convex on the region $\mathcal{B}$, we want to define a strongly concave
surrogate log-likelihood function which agrees with $\ell _{n}$ on a subset
$\tilde{\mathcal{B}}\subset \mathcal{B}$.

Our construction is similar to Definition~3.5 of \citep{nickl2020a}. For
a smooth function $\varphi :\mathbb{R}\rightarrow [0,\infty )$ with support
in $[-1,1]$, satisfying $\varphi (-x)=\varphi (x)$,
$\int _{\mathbb{R}}\varphi (x)\,dx=1$, define the mollifier
$\varphi _{t}(x)=t^{-1}\varphi (x/t)$, $t>0$. With this define two smooth
auxiliary functions: a `cut-off function'\emph{ }$v:\mathbb{R}^{p}
\rightarrow [0,1]$ and a globally convex function
$v_{\eta}:\mathbb{R}^{p}\rightarrow [0,\infty )$,
$v_{\eta}(t)=(\varphi _{\eta /8}*\gamma _{\eta})(t)$, where
$*$ is the convolution product, and where
\begin{equation*}
v(t)= %
\begin{cases} 1, & t\leq 3/4,
\\
0, & t>7/8, \end{cases} %
\qquad \gamma _{\eta}(t)= %
\begin{cases} 0, & t<5\eta /8,
\\
(t-5\eta /8)^{2}, & t\geq 5\eta /8. \end{cases} %
\end{equation*}
With $c_{\operatorname{max}}$ and $\kappa _{2}$ from Assumption~\ref{assu:a1} set
$K=60c_{\operatorname{max}}\lVert v\rVert _{C^{2}}n  (1+p^{
\kappa _{2}}  )$, and define the function
$\tilde{\ell}_{n}:\mathbb{R}^{p}\rightarrow \mathbb{R}$ by
\begin{align}
\tilde{\ell}_{n}(\theta ) & =v \bigl(\lVert \theta -\theta
_{
\operatorname{init}}\rVert /\eta \bigr) \bigl(\ell _{n}(\theta )-\ell
_{n}( \theta _{\operatorname{init}})\bigr)+\ell _{n}(\theta
_{
\operatorname{init}})-Kv_{\eta} \bigl(\lVert \theta -\theta _{
\operatorname{init}}
\rVert \bigr). \label{eq:surrogate}
\end{align}
We define analogously to (\ref{eq:posteriorDensity}) the `surrogate' posterior
measure $\tilde{\Pi}(\cdot |Z^{(n)})$ with density
\begin{equation}
\tilde{\pi}\bigl(\theta |Z^{(n)}\bigr)= \frac{e^{\tilde{\ell}_{n}(\theta )}\pi (\theta )}{\int _{\Theta}e^{\tilde{\ell}_{n}(\theta )}\pi (\theta )\,d\theta} \propto
e^{\tilde{\ell}_{n}(\theta )}\pi (\theta ). \label{eq:surrogatePosterior} 
\end{equation}

We show now that $\tilde{\ell}_{n}$ indeed coincides with the true log-likelihood
for points near $\theta _{\operatorname{init}}$, and we quantify how the
cut-off and convexification affect the curvature and Lipschitz constants
of $\log \tilde\pi (\cdot |Z^{(n)})$.
\begin{thm}
\label{thm:surrogate}%
Under Assumptions \ref{assu:a1} and \ref{assu:a2} the following holds:
\begin{enumerate}[label=(\roman*)]
\item[{(i)}] $\ell _{n}(\theta )=\tilde{\ell}_{n}(\theta )$ for all
$\theta \in \tilde{\mathcal{B}}=  \{ \theta \in \mathbb{R}^{p}:
\lVert \theta -\theta _{*,p}\rVert \leq 3\eta /8  \}$.
\item[{(ii)}] On the event $\mathcal{E}$, $-\tilde{\ell}_{n}$ satisfies
\eqref{eq:stronglyConvex} with curvature and Lipschitz constants
\begin{equation*}
\tilde{m}=c_{\min}np^{-\kappa _{3}},\qquad \tilde{\Lambda}=420c_{
\operatorname{max}}
\lVert v\rVert _{C^{2}}n \bigl(1+p^{\kappa _{2}} \bigr).
\end{equation*}
\end{enumerate}
In particular, under Assumptions \ref{assu:a1}, \ref{assu:a2} and
\ref{assu:a4}, $-\log \tilde{\pi}(\cdot |Z^{(n)})$ satisfies
\eqref{eq:stronglyConvex} with curvature and Lipschitz constants
\begin{equation*}
m=\tilde{m}+m_{\pi},\qquad \Lambda =\tilde{\Lambda}+\Lambda
_{\pi}.
\end{equation*}
\end{thm}

\begin{proof}
Part (i) is true by the construction of $\tilde{\ell}_{n}$ in (\ref{eq:surrogate})
and the condition on the initialiser in Assumption~\ref{assu:a2}. The supplement
follows immediately from part (ii). For the proof of (ii) let us restrict
to the event $\mathcal{E}$ and write
$\bar{v}=v(\lVert \cdot -\theta _{\operatorname{init}}\rVert /\eta )$,
$\tilde{v}=v_{\eta}(\lVert \cdot -\theta _{\operatorname{init}}
\rVert )$. We consider first the set
$V=\{\theta \in \mathbb{R}^{p}:\lVert \theta -\theta _{
\operatorname{init}}\rVert \leq 3\eta /4\}\subset \mathcal{B}$. On
$V$, $\tilde{v}$ vanishes and $\bar{v}=1$. Hence, by the local curvature
bound from Assumption~\ref{assu:a1}(iii) we have
\begin{equation}
\inf_{\theta \in V}\lambda _{\min} \bigl(-\nabla
^{2}\tilde{\ell}_{n}( \theta ) \bigr)\geq \inf
_{\theta \in \mathcal{B}}\lambda _{\min} \bigl(-\nabla ^{2}\ell
_{n}(\theta ) \bigr)\geq c_{\min}np^{-
\kappa _{3}}.
\label{eq:surrogate_lower_1} 
\end{equation}
Next, Lemma B.5 and the proof of Lemma B.6 in \citep{nickl2020a} (with
$\lambda _{\operatorname{max}}(I)=1$) imply that
$\lVert \nabla \bar{v}(\theta )\rVert \leq \lVert v\rVert _{C^{1}}
\eta ^{-1}$,
$\lVert \nabla ^{2}\bar{v}(\theta )\rVert _{\operatorname{op}}\leq 4
\lVert v\rVert _{C^{2}}\eta ^{-2}$ for all
$\theta \in \mathbb{R}^{p}$, as well as
$\lambda _{\min}(\nabla ^{2}\tilde{v}(\theta ))\geq 1/3$,
$\lVert \nabla ^{2}\tilde{v}(\theta )\rVert _{\operatorname{op}}\leq 6$.
Assumption~\ref{assu:a1}(ii) gives
\begin{align*}
\sup_{\theta \in \mathcal{B}} \bigl\llvert \ell _{n}(\theta )-\ell
_{n}(\theta _{*,p}) \bigr\rrvert 
    & \leq \bigl\lVert \nabla
\ell _{n}(\theta _{*,p})\bigr\rVert \eta +\sup
_{
\theta \in \mathcal{B}}\bigl\lVert \nabla ^{2}\ell _{n}(
\theta )\bigr\rVert _{
\operatorname{op}}\bigl(\eta ^{2}/2\bigr) \\
    &\leq
c_{\max}n \bigl(\delta _{n}p^{
\kappa _{1}}\eta
+p^{\kappa _{2}}\eta ^{2}/2 \bigr),
\\
\sup_{\theta \in \mathcal{B}}\bigl\lVert \nabla \ell _{n}(\theta )
\bigr\rVert & \leq \bigl\lVert \nabla \ell _{n}(\theta
_{*,p})\bigr\rVert +\sup_{\theta \in
\mathcal{B}}\bigl\lVert \nabla
^{2}\ell _{n}(\theta )\bigr\rVert _{
\operatorname{op}}\eta \leq
c_{\max}n \bigl(\delta _{n}p^{\kappa _{1}}+ \eta
p^{\kappa _{2}} \bigr).
\end{align*}
By the triangle inequality
$\sup_{\theta \in \mathcal{B}}|\ell _{n}(\theta )-\ell _{n}(\theta _{
\operatorname{init}})|\leq 2c_{\max}n(\delta _{n}p^{\kappa _{1}}\eta +p^{
\kappa _{2}}\eta ^{2}/2)$. Combining the last two displays and noting that
$\bar{v}$ vanishes outside $\mathcal{B}$ yields by the chain rule for
$\theta \in \mathbb{R}^{p}$
\begin{align*}
\bigl\lVert \nabla ^{2}(\bar{v}\bigl(\ell _{n}-\ell
_{n}(\theta _{
\operatorname{init}})\bigr) (\theta )\bigr\rVert
_{\operatorname{op}} & \leq \sup_{
\theta \in \mathcal{B}} \bigl(\bigl\lVert \nabla
^{2}\bar{v}(\theta )\bigr\rVert _{
\operatorname{op}} \bigl\llvert \ell
_{n}(\theta )-\ell _{n}(\theta _{
\operatorname{init}}) \bigr
\rrvert +2\bigl\lVert \nabla \bar{v}(\theta )\bigr\rVert \bigl\lVert \nabla \ell
_{n}(\theta )\bigr\rVert
\\
&\quad{} + \bigl\llvert \bar{v}(\theta ) \bigr\rrvert \bigl\lVert \nabla
^{2}\ell _{n}(\theta ) \bigr\rVert _{\operatorname{op}} \bigr)
\leq 10c_{\operatorname{max}}\lVert v \rVert _{C^{2}}n \bigl(\eta
^{-1}\delta _{n}p^{\kappa _{1}}+p^{
\kappa _{2}} \bigr)\\
&=:
K/6,
\end{align*}
using $\eta \geq \delta _{n}p^{\kappa _{1}+\kappa _{3}}$ by the lower bound
on $\eta $ in Assumption~\ref{assu:a1}. Consequently,
\begin{align*}
\inf_{\theta \in V^{c}}\lambda _{\min} \bigl(-\nabla
^{2} \tilde{\ell}_{n}(\theta ) \bigr) & \geq -\sup
_{\theta \in
\mathbb{R}^{p}}\bigl\lVert \nabla ^{2}(\bar{v}\ell
_{n}) (\theta )\bigr\rVert _{
\operatorname{op}}+K/3\geq K/6.
\end{align*}
Together with (\ref{eq:surrogate_lower_1}),
$\lVert v\rVert _{C^{2}}\geq 1$ and $K/6\geq c_{\min}n$ we obtain
the wanted curvature bound:
\begin{equation*}
\inf_{\theta \in \mathbb{R}^{p}}\lambda _{\min} \bigl(-\nabla
^{2} \tilde{\ell}_{n}(\theta ) \bigr)\geq \min
\bigl(c_{\min}np^{-
\kappa _{3}},K/6 \bigr)=c_{\min}np^{-\kappa _{3}}=:m.
\end{equation*}
At last, the gradient-Lipschitz bound follows for
$\theta \neq \theta '\in \mathbb{R}^{p}$ from
\begin{equation*}
\sup_{\theta \in \mathbb{R}^{p}}\bigl\lVert \nabla ^{2}\tilde{\ell
}_{n}( \theta )\bigr\rVert _{\operatorname{op}}\leq \sup
_{\theta \in \mathbb{R}^{p}} \bigl\lVert \nabla ^{2}(\bar{v}\ell
_{n}) (\theta )\bigr\rVert _{
\operatorname{op}}+K\sup_{\theta \in \mathbb{R}^{p}}
\bigl\lVert \nabla ^{2} \tilde{v}(\theta )\bigr\rVert _{\operatorname{op}}
\leq K/6+6K\leq 7K=: \Lambda .
\end{equation*}\qedhere
\end{proof}

Our first main result in Theorem~\ref{thm:surrogateWasserstein} shows that
the surrogate posterior is very close to the true posterior in Wasserstein
distance with high $\mathbb{P}^{n}_{\theta _{0}}$-probability. A~key step
in the proof is to argue first by the log-concavity of
$\tilde{\Pi}(\cdot |Z^{(n)})$ instead of the standard approach in Bayesian
nonparametrics using Hellinger tests (see e.g. the proof of
Theorem~\ref{thm:GaussianContraction} below) that the surrogate posterior contracts
around $\theta _{*,p}$ at the rate $\tilde{\delta}_{n,p}$, and second,
that it contracts \emph{at the same rate $\delta _{n}$} as the true posterior.
Analogously to \eqref{eq:DncL} define with respect to the surrogate posterior
the event
\begin{align*}
\tilde{\mathcal{D}}_{n}(c,L) &= \bigl\{\tilde{\Pi} \bigl(
\theta \in \mathbb{R}^{p}:\lVert \theta -\theta _{*,p}\rVert
^{\beta }>L \delta _{n}\big|Z^{(n)} \bigr)> e^{-cn\delta _{n}^{2}}
\bigr\}.
\end{align*}

\begin{prop}
\label{prop:surrogateContraction}%
Under Assumptions \ref{assu:a1}-\ref{assu:a3}, for any $c>0$ and any sufficiently
large $L>0$ there exist $\tilde{c}_{1},\tilde{c}_{2}>0$ with
\begin{equation*}
\mathbb{P}_{\theta _{0}}^{n} \bigl(\tilde{\mathcal{D}}_{n}(c,L)
\bigr)\leq \tilde{c}_{2}e^{-\tilde{c}_{1}n\delta _{n}^{2}}.
\end{equation*}
\end{prop}

The proof is preceded by two lemmas.

\begin{lem}
\label{lem:surrogate_stability}%
Grant Assumptions \ref{assu:a1} and \ref{assu:a2}, and suppose that
$\theta \in \mathbb{R}^{p}$ satisfies
$\lVert \theta -\theta _{*,p}\rVert >(4Cc_{\max}/c_{\min})
\tilde{\delta}_{n,p}$ for $C\geq 1$. Then, on the event
$\mathcal{E}$, we have
$$\tilde{\ell}_{n}(\theta )-\tilde{\ell}_{n}(\theta _{*,p})<-(4C^{2}c_{
\max}^{2}/c_{\min})n\delta _{n}^{2}.$$
\end{lem}

\begin{proof}
The first part of Theorem~\ref{thm:surrogate}(i) gives
$\nabla \tilde{\ell}_{n}(\theta _{*,p})=\nabla \ell _{n}(\theta _{*,p})$,
while the strong concavity in its second part and the Cauchy--Schwarz inequality
as well as the upper bound on the gradient in Assumption~\ref{assu:a1}(ii) imply
\begin{align*}
\tilde{\ell}_{n}(\theta )-\tilde{\ell}_{n}(\theta
_{*,p}) & \leq c_{
\max}np^{\kappa _{1}}\delta _{n}
\lVert \theta -\theta _{*,p}\rVert -(m/2) \lVert \theta -\theta
_{*,p}\rVert ^{2}.
\end{align*}
If
$\lVert \theta -\theta _{*,p}\rVert >(4Cc_{\max}/c_{\min})
\tilde{\delta}_{n,p}$ and $C\geq 1$, then
$(4c_{\max}/c_{\min})\delta _{n}p^{\kappa _{1}+\kappa _{3}}\lVert
\theta -\theta _{*,p}\rVert <\lVert \theta -\theta _{*,p}\rVert ^{2}$ and
therefore
\begin{equation*}
\tilde{\ell}_{n}(\theta )-\ell _{n}(\theta
_{*,p}) <(m/4)\lVert \theta -\theta _{*,p}\rVert
^{2}-(m/2)\lVert \theta -\theta _{*,p} \rVert
^{2}<-\bigl(4C^{2}c_{\max}^{2}/c_{\min}
\bigr)n\delta _{n}^{2}.
\end{equation*}\qedhere
\end{proof}

\begin{lem}%
\label{lem:randomNormFactors}
Let $3\eta /8\geq L^{1/\beta}\tilde{\delta}_{n,p}$. On the event
$\mathcal{D}^{c}_{n}(c,L)$, the random normalising factors
\begin{align*}
p_{n} = \frac{\int _{\Theta}e^{\ell _{n}(\theta )}\pi (\theta )\,d\theta}{\int _{\Theta}e^{\tilde{\ell}_{n}(\theta )}\pi (\theta )\,d\theta}
\end{align*}
satisfy $p_{n}<(1-e^{-cn\delta _{n}^{2}})$,
$-e^{-cn\delta _{n}^{2}}(1-e^{-cn\delta _{n}^{2}})^{-1}<1-p_{n}$, while
on the event $\tilde{\mathcal{D}}^{c}_{n}(c,L)$ they satisfy
$1-p_{n}< e^{-cn\delta _{n}^{2}}$.
\end{lem}

\begin{proof}
Set
$U=\{\theta \in \mathbb{R}^{p}:\lVert \theta -\theta _{*,p}\rVert ^{
\beta}\leq L\delta _{n}\}$. For
$3\eta /8\geq L^{1/\beta}\tilde{\delta}_{n,p}$, we have
$U\subset \tilde{\mathcal{B}}$. So, by Theorem~\ref{thm:surrogate}(i) on
the event $\mathcal{D}^{c}_{n}(c,L)$,
\begin{align*}
p_{n}^{-1}\geq p_{n}^{-1}\tilde{\Pi}
\bigl(U|Z^{(n)}\bigr) & =\Pi \bigl(U|Z^{(n)}\bigr)>
1-e^{-cn
\delta _{n}^{2}}.
\end{align*}
This is equivalent to
$-e^{-cn\delta _{n}^{2}}(1-e^{-cn\delta _{n}^{2}})^{-1}< 1-p_{n}$, while
on the event $\tilde{\mathcal{D}}^{c}_{n}(c,L)$
\begin{equation*}
p_{n}\geq p_{n}\Pi \bigl(U|Z^{(n)}\bigr) =\tilde{
\Pi}\bigl(U|Z^{(n)}\bigr)> 1-e^{-cn
\delta _{n}^{2}}.
\end{equation*}\qedhere
\end{proof}

\begin{proof}[Proof of Proposition~\ref{prop:surrogateContraction}]
Fix $C_{3}$ from Assumption~\ref{assu:a3}, and for $c,L>0$ set
$C=L^{1/\beta}c_{\operatorname{min}}/(4c_{\operatorname{max}})$ as well
as $\tilde{C}=4C^{2}c^{2}_{\max}/c_{\min}-C_{3}-c$. Suppose that
$L$ is sufficiently large such that $C\geq 1$, $\tilde{C}>0$. By the lower
bound on $\eta $ in Assumption~\ref{assu:a1} and increasing the constant
$\tilde{c}_{2}$ we can assume that $n$ is large enough to ensure
$3\eta /8\geq L^{1/\beta}\tilde{\delta}_{n,p}$.

On the event $\mathcal{E}\cap \bar{\mathcal{D}}^{c}_{n}(C_{3})$, for
$\lVert \theta -\theta _{*,p}\rVert > L^{1/\beta}\tilde{\delta}_{n,p}=(4Cc_{
\max}/c_{\min})\tilde{\delta}_{n,p}$ Lemma~\ref{lem:surrogate_stability} shows
\begin{align}
\begin{aligned}
\tilde{\pi}\bigl(\theta |Z^{(n)}\bigr)&=
\frac{e^{\tilde{\ell}_{n}(\theta )-\ell _{n}(\theta _{0})}\pi (\theta )}{\int _{\Theta}e^{\tilde{\ell}_{n}(\theta )-\ell _{n}(\theta _{0})}\pi (\theta )\,d\theta} \leq
\frac{e^{\tilde{\ell}_{n}(\theta )-\ell _{n}(\theta _{0})}\pi (\theta )}{\int _{\lVert \theta -\theta _{*,p}\rVert \leq \delta _{n}}e^{\ell _{n}(\theta )-\ell _{n}(\theta _{0})}\pi (\theta )\,d\theta}
\\
&\leq  e^{C_{3}n\delta _{n}^{2}}e^{\tilde{\ell}_{n}(\theta )-\ell _{n}(
\theta _{0})}\pi (\theta )\leq e^{-(\tilde{C}+c)n\delta _{n}^{2}}e^{
\tilde{\ell}_{n}(\theta _{*,p})-\ell _{n}(\theta _{0})}
\pi (\theta ).
\end{aligned}
\label{eq:surrogateConcent_upperBound}
\end{align}
Since $\theta _{*,p}\in \tilde{\mathcal{B}}$, we have
$\tilde{\ell}_{n}(\theta _{*,p})=\ell _{n}(\theta _{*,p})$, and so
$$\mathbb{E}_{\theta _{0}}^{n}[e^{\tilde{\ell}_{n}(\theta _{*,p})-
\ell _{n}(\theta _{0})}]=\mathbb{E}^{n}_{\theta _{*,p}}\mathbf{1}(e^{-
\ell _{n}(\theta _{0})}\neq 0)\leq 1.$$ The previous display and the Markov
inequality therefore show
\begin{align}
\begin{aligned}
& \mathbb{P}_{\theta _{0}}^{n} \bigl(\tilde{\Pi}\bigl(\theta \in
\mathbb{R}^{p}:\lVert \theta -\theta _{*,p}\rVert
^{\beta}> L \tilde{\delta}^{\beta}_{n,p}|Z^{(n)}
\bigr)>e^{-cn\delta _{n}^{2}}/2, \mathcal{E}\cap \bar{\mathcal{D}}_{n}(C_{3})
\bigr)
\\
& \quad \leq \mathbb{P}_{\theta _{0}}^{n} \biggl(e^{-\tilde{C}n\delta _{n}^{2}}e^{
\tilde{\ell}_{n}(\theta _{*,p})-\ell _{n}(\theta _{0})}
\int _{\lVert
\theta -\theta _{*,p}\rVert ^{\beta}> L\tilde{\delta}^{\beta}_{n,p}} \pi (\theta )\,d\theta >1/2 \biggr)
\\
& \quad \leq 2e^{-\tilde{C}n\delta _{n}^{2}}\mathbb{E}_{\theta _{0}}^{n}
\bigl[e^{
\tilde{\ell}_{n}(\theta _{*,p})-\ell _{n}(\theta _{0})}\bigr]\leq 2e^{-
\tilde{C}n\delta _{n}^{2}}.
\end{aligned}
\label{eq:normalisationUpperBound}
\end{align}
By Assumptions \ref{assu:a1}, \ref{assu:a3}, $\mathcal{E}$ and
$\bar{\mathcal{D}}_{n}(C_{3})$ are high
$\mathbb{P}^{n}_{\theta _{0}}$-probability events, so
\begin{align*}
\mathbb{P}_{\theta _{0}}^{n} \bigl(\tilde{\mathcal{D}}_{n}(c,L)
\bigr) &\leq \mathbb{P}_{\theta _{0}}^{n} \bigl(\tilde{\Pi}\bigl(
\theta \in \mathbb{R}^{p}:L\delta _{n}<\lVert \theta -\theta
_{*,p}\rVert ^{
\beta}\leq L\tilde{\delta}^{\beta}_{n,p}|Z^{(n)}
\bigr)>e^{-cn\delta _{n}^{2}}/2 \bigr)
\\
&\quad{} + \mathbb{P}_{\theta _{0}}^{n} \bigl(\tilde{\Pi}\bigl(\theta
\in \mathbb{R}^{p}:\lVert \theta -\theta _{*,p}\rVert
^{\beta}> L \tilde{\delta}^{\beta}_{n,p}|Z^{(n)}
\bigr)>e^{-cn\delta _{n}^{2}}/2 \bigr)
\\
& \leq \mathbb{P}_{\theta _{0}}^{n} \bigl(\tilde{\Pi}\bigl(
\theta \in \mathbb{R}^{p}:L\delta _{n}<\lVert \theta -\theta
_{*,p}\rVert ^{
\beta}\leq L\tilde{\delta}^{\beta}_{n,p}|Z^{(n)}
\bigr)>e^{-cn\delta _{n}^{2}}/2 \bigr) \\
& \quad + c_{2}e^{-c_{1}n\delta _{n}^{2}}+C_{2}'e^{-C_{2}n\delta _{n}^{2}}.
\end{align*}
To conclude the proof, on the event $\mathcal{D}^{c}_{n}(c,L)$ we can apply
Lemma~\ref{lem:randomNormFactors}. Replacing $\tilde{\Pi}$ with
$\Pi $ up to the random normalising factors $p_{n}$, and upper bounding
these and the event in question we find
\begin{align*}
&\tilde{\Pi}\bigl(\theta \in \mathbb{R}^{p}:L\delta _{n}<
\lVert \theta - \theta _{*,p}\rVert ^{\beta}\leq L\tilde{
\delta}^{\beta}_{n,p}|Z^{(n)}\bigr)\\
&\quad  \leq
\bigl(1-e^{-cn\delta _{n}^{2}}\bigr)^{-1} \Pi \bigl(\theta \in
\mathbb{R}^{p}: \lVert \theta -\theta _{*,p}\rVert
^{\beta }> L\delta _{n}|Z^{(n)}\bigr).
\end{align*}
By increasing $\tilde{c}_{2}$ one more time, we ensure that for all large
enough $n$ also $1\geq 3e^{-cn\delta _{n}^{2}}$, or equivalently
$(1-e^{-cn\delta _{n}^{2}})e^{-cn\delta _{n}^{2}}/2\geq e^{-2cn
\delta _{n}^{2}}$. So,
\begin{align*}
\mathbb{P}_{\theta _{0}}^{n} \bigl(\tilde{\mathcal{D}}_{n}(c,L)
\bigr) \leq \mathbb{P}^{n}_{\theta _{0}}\bigl(\mathcal{D}_{n}(2c,L)
\bigr)+ \mathbb{P}^{n}_{\theta _{0}}\bigl(\mathcal{D}_{n}(c,L)
\bigr) + c_{2}e^{-c_{1}n
\delta _{n}^{2}}+C_{2}'e^{-C_{2}n\delta _{n}^{2}}.
\end{align*}
Possibly increasing $L$ to apply Assumption~\ref{assu:a3} with $c$ and
$2c$, the claim follows.
\end{proof}

With this preparation we proceed to the proof of the Wasserstein approximation.

\begin{thm}
\label{thm:surrogateWasserstein}%
Under Assumptions \ref{assu:a1}-\ref{assu:a4} there exist an event
$\tilde{\mathcal{E}}$ and constants $c_{1},c_{2}>0$ with
$\mathbb{P}_{\theta _{0}}^{n}(\tilde{\mathcal{E}})\geq 1-c_{2}e^{-c_{1}n
\delta _{n}^{2}}$ such that
$W_{2}^{2}  (\tilde{\Pi}(\cdot |Z^{(n)}),\Pi (\cdot |Z^{(n)})
  )\leq 3e^{-n\delta _{n}^{2}}$ on $\tilde{\mathcal{E}}$.
\end{thm}

\begin{proof}
Fix $C_{3}$ from Assumption~\ref{assu:a3}. For $c,L>0$ to be determined
later set
$\tilde{\mathcal{E}}=\mathcal{E}\cap \mathcal{D}^{c}_{n}(c,L)
\cap \tilde{\mathcal{D}}^{c}_{n}(c,L)\cap \bar{\mathcal{D}}^{c}_{n}(C_{3})$.
We begin by applying Theorem~6.15 of Villani (2009) to upper bound the
squared Wasserstein distance between the posterior and the surrogate posterior
as
\begin{align*}
W_{2}^{2}\bigl(\tilde{\Pi}\bigl(\cdot |Z^{(n)}
\bigr),\Pi \bigl(\cdot |Z^{(n)}\bigr)\bigr) & \leq 2
\int _{\mathbb{R}^{p}}\lVert \theta -\theta _{*,p}\rVert
^{2} \bigl\llvert \tilde{\pi}\bigl(\theta |Z^{(n)}\bigr)-\pi
\bigl(\theta |Z^{(n)}\bigr) \bigr\rrvert \,d\theta .
\end{align*}
Decompose the integral as
$\mathcal{I}_{1}+\mathcal{I}_{2}+\mathcal{I}_{3}$ with
\begin{align*}
\mathcal{I}_{1} & =
\int _{\lVert \theta -\theta _{*,p}\rVert ^{\beta}
\leq L\tilde{\delta}^{\beta}_{n,p}}\lVert \theta -\theta _{*,p} \rVert
^{2} \bigl\llvert \tilde{\pi}\bigl(\theta |Z^{(n)}\bigr)-\pi
\bigl(\theta |Z^{(n)}\bigr) \bigr\rrvert \,d\theta ,
\\
\mathcal{I}_{2} & =
\int _{\lVert \theta -\theta _{*,p}\rVert ^{\beta}>
L\tilde{\delta}^{\beta}_{n,p}}\lVert \theta -\theta _{*,p}\rVert
^{2} \tilde{\pi}\bigl(\theta |Z^{(n)}\bigr)\,d\theta ,
\\
\mathcal{I}_{3} & =
\int _{\lVert \theta -\theta _{*,p}\rVert ^{\beta}>
L\tilde{\delta}^{\beta}_{n,p}}\lVert \theta -\theta _{*,p}\rVert
^{2} \pi \bigl(\theta |Z^{(n)}\bigr)\,d\theta .
\end{align*}
Taking $L$ sufficiently large, by Assumptions \ref{assu:a1},
\ref{assu:a3} and Proposition~\ref{prop:surrogateContraction} the event
$\tilde{\mathcal{E}}$ is of high $\mathbb{P}_{\theta _{0}}^{n}$-probability.
Using a union bound, it is therefore enough to show that each of terms
$\mathcal{I}_{1}$ to $\mathcal{I}_{3}$ exceeds
$e^{-n\delta _{n}^{2}}$ on $\tilde{\mathcal{E}}$ only with exponentially
small $\mathbb{P}_{\theta _{0}}^{n}$-probability.

By increasing the constant $c_{2}$ and recalling the lower bound on
$\eta $ from Assumption~\ref{assu:a1} we can assume that $n$ is large enough
to ensure $3\eta /8\geq L^{1/\beta}\tilde{\delta}_{n,p}$. So, if
$\lVert \theta -\theta _{*,p}\rVert ^{\beta}\leq L\tilde{\delta}^{
\beta}_{n,p}$, then
$\lVert \theta -\theta _{*,p}\rVert ^{2}\leq L^{2/\beta}
\tilde{\delta}^{2}_{n,p}\leq 1$ and
$p_{n}\pi (\theta |Z^{(n)})=\tilde{\pi}(\theta |Z^{(n)})$ with the random
normalising factors $p_{n}$, using that the log-likelihood and the surrogate
log-likelihood coincide for $\theta \in \tilde{\mathcal{B}}$ by Theorem~\ref{thm:surrogate}(i). This means on $\tilde{\mathcal{E}}$
\begin{align*}
\mathcal{I}_{1} & \leq L^{2/\beta}\tilde{\delta}_{n,p}^{2}
\int _{
\lVert \theta -\theta _{*,p}\rVert ^{\beta}\leq L\tilde{\delta}^{
\beta}_{n,p}} \bigl\llvert \tilde{\pi}\bigl(\theta
|Z^{(n)}\bigr)-\pi \bigl(\theta |Z^{(n)}\bigr) \bigr\rrvert d
\theta
\\
& = L^{2/\beta}\tilde{\delta}_{n,p}^{2} \llvert
1-p_{n} \rrvert \Pi \bigl(\lVert \theta - \theta _{*,p}
\rVert ^{\beta }\leq L\tilde{\delta}^{\beta}_{n,p}|Z^{(n)}
\bigr) \leq \llvert 1-p_{n} \rrvert .
\end{align*}
Suppose for the rest of the proof that
$c\geq \max (2,2\log 2/(n\delta ^{2}))$. Then
$e^{-(c/2)n\delta _{n}^{2}}\leq 1/2$ and thus
$e^{-cn\delta _{n}^{2}}(1-e^{-cn\delta _{n}^{2}})^{-1}\leq e^{-(c/2)n
\delta _{n}^{2}}$. Consequently,
$\mathcal{I}_{1}\leq e^{-(c/2)n\delta _{n}^{2}}$ by Lemma~\ref{lem:randomNormFactors}. On the other hand, we find from the Cauchy--Schwarz
inequality and (\ref{eq:surrogateConcent_upperBound}), with
$\tilde{C}=\tilde{C}(L)$ increasing in $L$ and $L$ large enough, that
\begin{align*}
\mathcal{I}_{2}^{2} & \leq \tilde{\Pi}\bigl(\lVert \theta -
\theta _{*,p} \rVert ^{\beta}> L\tilde{\delta}^{\beta}_{n,p}|Z^{(n)}
\bigr)
\int _{\lVert
\theta -\theta _{*,p}\rVert ^{\beta}> L\tilde{\delta}^{\beta}_{n,p}} \lVert \theta -\theta _{*,p}\rVert
^{4}\tilde{\pi}\bigl(\theta |Z^{(n)}\bigr)d \theta
\\
& \leq e^{-(\tilde{C}+c)n\delta _{n}^{2}}e^{\tilde{\ell}_{n}(\theta _{*,p})-
\ell _{n}(\theta _{0})}
\int _{\Theta}\lVert \theta -\theta _{*,p} \rVert
^{4}\pi (\theta )\,d\theta .
\end{align*}
By Assumption~\ref{assu:a1} we have
$\lVert \theta _{*,p}\rVert \leq c_{0}$, and by Assumption~\ref{assu:a4} the prior has uniformly bounded fourth moments, so the integral
above is bounded by a constant $C>0$ uniformly in $n$. We infer from the
Markov inequality, Fubini's theorem as well as
\eqref{eq:normalisationUpperBound} that
\begin{align*}
& \mathbb{P}_{\theta _{0}}^{n} \bigl(\mathcal{I}_{2}>e^{-n\delta _{n}^{2}},
\tilde{\mathcal{E}} \bigr)\leq Ce^{-(\tilde{C}+c-2)n\delta _{n}^{2}} \mathbb{E}_{\theta _{0}}^{n}
\bigl[e^{\tilde{\ell}_{n}(\theta _{*,p})-\ell _{n}(
\theta _{0})}\bigr]\leq Ce^{-\tilde{C}n\delta _{n}^{2}}.
\end{align*}
Since
$\pi (\theta |Z^{(n)})=p_{n}^{-1}\tilde{\pi}(\theta |Z^{(n)})\leq (1-e^{-cn
\delta _{n}^{2}})^{-1}\tilde{\pi}(\theta |Z^{(n)})\leq 2\tilde{\pi}(
\theta |Z^{(n)})$, this argument gives the same upper bound relative to
$\mathcal{I}_{3}$ with the larger constant $2C$. This finishes the proof.
\end{proof}

\begin{rem}[comparison with \citep{nickl2020a}]%
\label{rem:whyDifferent}
The proof that the surrogate posterior approximates the true posterior
well in Wasserstein distance generalises the specific argument of Theorem~4.14
in \citep{nickl2020a}, which relies on a version of Lemma~\ref{lem:surrogate_stability}
with $\theta _{*,p}$ replaced by the MAP
estimator. The analysis of the MAP estimator is based on
\cite{vandegeer2001}. The latter seems not to generalise to the statistical
models of Section~\ref{sec:Applications}.
\end{rem}

\subsection{Sampling guarantees relative to the surrogate posterior}
\label{subsec:Polynomial-time-sampling}

A standard MCMC approach for sampling from the Gibbs-type measure
\eqref{eq:surrogatePosterior} is the
\emph{unadjusted Langevin algorithm }\citep{roberts1996}. With the initialiser
$\vartheta _{0}=\theta _{\operatorname{init}}\in \mathbb{R}^{p}$ from Assumption~\ref{assu:a2} as input, a step size $\gamma >0$ and independent $p$-dimensional
Gaussian innovations $\xi _{k}\sim N(0,I_{p\times p})$, it produces a Markov
chain with iterates $\vartheta _{k}\in \mathbb{R}^{p}$, where
\begin{align}
\tilde \vartheta_{k+1} & = \tilde\vartheta _{k}+\gamma
\nabla \log \tilde{\pi}\bigl(\tilde \vartheta_{k} |Z^{(n)}\bigr)+\sqrt{2\gamma}\xi _{k+1}\nonumber\\
& = \tilde\vartheta _{k}+\gamma
\bigl(\nabla \tilde\ell _{n}(\tilde\vartheta _{k})+\nabla \log \pi (
\tilde\vartheta _{k}) \bigr)+\sqrt{2\gamma}\xi _{k+1}.
\label{eq:surrogateLangevinChain} 
\end{align}

Suppose that the gradient $\nabla \tilde{\ell}_{n}$ can be evaluated at
a cost depending at most polynomially on $n$, $p$. Combining the Wasserstein
approximation in Theorem~\ref{thm:surrogateWasserstein} with standard nonasymptotic
sampling bounds for strongly log-concave potentials with Lipschitz-gradients,
we quantify the distance of the law
$\mathcal{L}(\tilde{\vartheta}_{k})$ to the true posterior measure in terms
of the number of steps $k$, the sample size $n$ and a sufficiently small
step size $\gamma $.
\begin{thm}
\label{thm:MCMC_surrogate_Wasserstein}%
Let $(\tilde{\vartheta}_{k})_{k\geq 1}$ be the Markov chain with iterates
\eqref{eq:surrogateLangevinChain}. Suppose
$\gamma \leq 2/(m+\Lambda )$ and set
\begin{align*}
A=4\bigl(9\eta /8 + (2\eta +4c_{0})\Lambda _{\pi}/m
\bigr)^{2}+4p/m,\qquad B( \gamma )=36\gamma p\Lambda ^{2}/m^{2}+12
\gamma ^{2}p\Lambda ^{4}/m^{3}.
\end{align*}
Under Assumptions \ref{assu:a1}-\ref{assu:a4} we have on the event
$\tilde{\mathcal{E}}$
\begin{equation*}
W_{2}^{2} \bigl(\mathcal{L}(\tilde{\vartheta}_{k}),
\Pi \bigl(\cdot |Z^{(n)}\bigr) \bigr)\leq 6e^{-n\delta _{n}^{2}}+A(1-m\gamma
/2)^{k}+B(\gamma ), \quad k\geq 1.
\end{equation*}
\end{thm}

As the Langevin algorithm \eqref{eq:surrogateLangevinChain} is unadjusted,
the law $\mathcal{L}(\tilde\vartheta _{k})$ is asymptotically biased as
$k\rightarrow \infty $ and convergence to $\Pi (\cdot |Z^{(n)})$ is achieved
only if $\gamma \rightarrow 0$ and $n\rightarrow \infty $. All constants
in the theorem are fully explicit. Importantly, the result implies that
the mixing time $k_{\operatorname{mix}}$ to reach a target accuracy
$\varepsilon >0$ such that
\begin{align*}
W_{2} \bigl(\mathcal{L}(\tilde{\vartheta}_{k}),\Pi \bigl(
\cdot |Z^{(n)}\bigr) \bigr)\leq \varepsilon ,\quad k\geq
k_{\operatorname{mix}},
\end{align*}
depends at most polynomially on $n$, $p$, $\gamma $ and
$\varepsilon $.

In order to prove the theorem observe first the following crude upper bound
on the Euclidean distance between the initialiser and the mode of the surrogate
posterior.
\begin{lem}
\label{lem:thetaMax}%
Grant Assumptions \ref{assu:a1}--\ref{assu:a4}. Let
$\theta _{\operatorname{max}}\in \mathbb{R}^{p}$ be the unique maximiser
of $\tilde{\pi}(\cdot |Z^{(n)})$. Then on the event $\mathcal{E}$
\begin{equation*}
\lVert \theta _{\operatorname{init}}-\theta _{\operatorname{max}} \rVert \leq 9\eta /8 + (2
\eta +4c_{0})\Lambda _{\pi}/m.
\end{equation*}
\end{lem}

\begin{proof}
Let $\tilde{\theta}_{\operatorname{max}}$ be the unique maximiser of the
strongly concave map $\tilde{\ell}_{n}$. By Assumption~\ref{assu:a2}
\begin{equation}
\lVert \theta _{\operatorname{init}}-\theta _{\operatorname{max}} \rVert \leq \eta /8+
\lVert \theta _{*,p}-\tilde{\theta }_{
\operatorname{max}}\rVert +\lVert
\tilde{\theta }_{\operatorname{max}}- \theta _{\operatorname{max}}\rVert . \label{eq:theta_max_1}
\end{equation}
If
$\lVert \tilde{\theta }_{\operatorname{max}}-\theta _{*,p}\rVert >(4c_{
\operatorname{max}}/c_{\operatorname{min}})\tilde{\delta}_{n,p}$, then
$\tilde{\ell}_{n}(\tilde{\theta}_{\operatorname{max}})-\tilde{\ell}_{n}(
\theta _{*,p})<0$ by Lemma~\ref{lem:surrogate_stability} on the event
$\mathcal{E}$. Since $\tilde{\theta}_{\operatorname{max}}$ is the unique
maximiser of $\tilde{\ell}_{n}$, this means necessarily
$\lVert \tilde{\theta }_{\operatorname{max}}-\theta _{*,p}\rVert
\leq (4c_{\operatorname{max}}/c_{\operatorname{min}})\tilde{\delta}_{n,p}
\leq \eta $. This already yields the claim by \eqref{eq:theta_max_1} if
$\tilde{\theta}_{\operatorname{max}}=\theta _{\operatorname{max}}$. Suppose
now
$\tilde{\theta}_{\operatorname{max}}\neq \theta _{\operatorname{max}}$,
and let $\theta _{\pi ,\max}$ be the unique maximiser of the prior density.
Since $\log \pi $ is concave and has $\Lambda _{\pi}$-Lipschitz gradients,
we get
\begin{align*}
\log \pi (\theta _{\operatorname{max}})-\log \pi (\tilde{\theta}_{
\operatorname{max}})
&\leq \nabla \log \pi (\tilde{\theta}_{
\operatorname{max}})^{\top}(\theta
_{\operatorname{max}}- \tilde{\theta}_{\operatorname{max}})
\\
& = \bigl(\nabla \log \pi (\tilde{\theta}_{\operatorname{max}})- \nabla \log
\pi (\theta _{\pi ,\operatorname{max}}) \bigr)^{\top}( \theta _{\operatorname{max}}-
\tilde{\theta}_{\operatorname{max}})
\\
& \leq \Lambda _{\pi}\lVert \theta _{\operatorname{max}}- \tilde{
\theta }_{\operatorname{max}}\rVert \lVert \tilde{\theta }_{
\operatorname{max}}-\theta
_{\pi ,\operatorname{max}}\rVert ,
\end{align*}
and therefore
\begin{align*}
0 & \leq \tilde{\ell}_{n}(\tilde{\theta}_{\operatorname{max}})- \tilde{
\ell}_{n}(\theta _{\operatorname{max}}) =\log \tilde{\pi}\bigl( \tilde{
\theta}_{\operatorname{max}}|Z^{(n)}\bigr)-\log \tilde{\pi}\bigl(\theta
_{
\operatorname{max}}|Z^{(n)}\bigr)+\log \pi (\theta _{\operatorname{max}})- \log
\pi (\tilde{\theta}_{\operatorname{max}})
\\
& \leq -\frac{m}{2}\lVert \tilde{\theta }_{\operatorname{max}}- \theta
_{\operatorname{max}}\rVert ^{2}+\Lambda _{\pi}\lVert \theta
_{
\operatorname{max}}-\tilde{\theta }_{\operatorname{max}}\rVert \lVert \tilde{\theta
}_{\operatorname{max}}-\theta _{\pi ,
\operatorname{max}}\rVert .
\end{align*}
Hence,
$\lVert \tilde{\theta }_{\operatorname{max}}-\theta _{
\operatorname{max}}\rVert \leq 2\lVert \tilde{\theta }_{
\operatorname{max}}-\theta _{\pi ,\operatorname{max}}\rVert \Lambda _{
\pi}/m$, and we conclude by \eqref{eq:theta_max_1} and
$\lVert \theta _{*,p}\rVert ,\lVert \theta _{\pi ,\operatorname{max}}
\rVert \leq c_{0}$, noting that
\begin{equation*}
\lVert \tilde{\theta }_{\operatorname{max}}-\theta _{\pi ,
\operatorname{max}}\rVert \leq
(4c_{\operatorname{max}}/c_{
\operatorname{min}})\tilde{\delta}_{n,p}+\lVert
\theta _{*,p}-\theta _{
\pi ,\operatorname{max}}\rVert \leq \eta +
2c_{0}.
\end{equation*}\qedhere
\end{proof}

\begin{proof}[Proof of Theorem~\ref{thm:MCMC_surrogate_Wasserstein}]
Apply Theorem~5 of \citep{durmus2019} (in the form stated as Proposition
A.4 in \citep{nickl2020a}) to the strongly log-concave measure
$\mu =\tilde{\Pi}(\cdot |Z^{(n)})$ from Theorem~\ref{thm:surrogate} with
unique maximiser $\theta _{\operatorname{max}}$ such that
\begin{equation*}
W_{2}^{2}\bigl(\mathcal{L}(\tilde{\vartheta}_{k}),
\tilde{\Pi}\bigl(\cdot |Z^{(n)}\bigr)\bigr) \leq 2(1-m\gamma
/2)^{k} \bigl(\lVert \theta _{\operatorname{init}}- \theta _{\operatorname{max}}
\rVert ^{2}+p/m \bigr)+B(\gamma )/2, \quad k\geq 0.
\end{equation*}
The claim follows from the triangle inequality for the Wasserstein distance,
Theorem~\ref{thm:surrogateWasserstein} and Lemma~\ref{lem:thetaMax}.
\end{proof}

Let the law of the Markov chain be denoted by $\mathbf{P}$. Then the mixing
time in Wasserstein distance immediately yields convergence guarantees
with high $\mathbb{P}_{\theta _{0}}^{n}\times \mathbf{P}$-probability for
the approximation of posterior functionals by ergodic averages of the Markov
chain.

\begin{thm}
\label{thm:MCMC_surrogate_functionals}%
Suppose $\gamma \leq 2/(m+\Lambda )$,
$\varepsilon \geq \sqrt{48e^{-n\delta _{n}^{2}}+8B(\gamma )}$ and consider
as burn-in time
\begin{equation*}
J_{\operatorname{in}} \geq \frac{2}{m\gamma} \biggl\llvert \log \biggl(
\frac{\varepsilon ^{2}}{8A} \biggr) \biggr\rrvert .
\end{equation*}
Under Assumptions \ref{assu:a1}--\ref{assu:a4} there exists a constant
$c_{F}>0$ such that for all Lipschitz functions
$h:\mathbb{R}^{p}\rightarrow \mathbb{R}$ with
$\lVert h\rVert _{\mathrm{Lip}}=1$ and all $J\geq 1$ on the event
$\tilde{\mathcal{E}}$
\begin{align*}
\mathbf{P} \Biggl( \Biggl\llvert \frac{1}{J}\sum
_{k=1+J_{\operatorname{in}}}^{J+J_{
\operatorname{in}}}h(\tilde{\vartheta}_{k})-
\int _{\Theta}h(\theta ) \pi \bigl(\theta |Z^{(n)}\bigr)d
\theta \Biggr\rrvert >\varepsilon \Biggr) & \leq 2 \exp \biggl(-c_{F}
\frac{\varepsilon ^{2}m^{2}J\gamma}{1+1/(mJ\gamma )} \biggr).
\end{align*}
\end{thm}

\begin{proof}
For $k\geq J_{\operatorname{in}}$,
$A(1-m\gamma /2)^{k}\leq \varepsilon ^{2}/8$, hence
$W_{2}^{2}(\mathcal{L}(\tilde{\vartheta}_{k}),\Pi (\cdot |Z^{(n)}))
\leq \varepsilon ^{2}/4$ by Theorem~\ref{thm:MCMC_surrogate_Wasserstein}.
The claim follows now from the proof
of \citep[Theorem~3.8]{nickl2020a}.
\end{proof}

\subsection{Sampling guarantees for vanilla Langevin MCMC}
\label{subsec:Polynomial-time-sampling-true-posterior}

We study next the question if the guarantees in Theorem~\ref{thm:MCMC_surrogate_functionals}
hold also with respect to the vanilla
Langevin Markov chain $(\vartheta _{k})_{k\geq 0}$ with iterates
\begin{align}
\vartheta _{k+1} & =\vartheta _{k}+\gamma \bigl(\nabla \ell
_{n}( \vartheta _{k})+\nabla \log \pi (\vartheta _{k})
\bigr)+\sqrt{2\gamma}\xi _{k+1}, \quad \vartheta _{0}\equiv
\theta _{\operatorname{init}}. \label{eq:langevin} 
\end{align}
Since both $(\vartheta _{k})_{k\geq 0}$ and
$(\tilde\vartheta _{k})_{k\geq 0}$ start from
$\theta _{\operatorname{init}}\in \tilde{\mathcal{B}}$, we have
$\tilde{\vartheta}_{k}=\vartheta _{k}$ according to Theorem~\ref{thm:surrogate}(i)
as long as $\vartheta _{k}$ has not exited yet from
the region of local curvature $\tilde{\mathcal{B}}$. Note that
$\tilde{\vartheta}_{k}$ will leave from $\tilde{\mathcal{B}}$
\emph{eventually} due to the Gaussian innovations. We will prove, however,
that this takes in average exponentially in $n$ many steps
$J_{\operatorname{out}}\gg J_{\operatorname{in}}$. The main idea is to
relate the discrete time Markov chain
$(\tilde{\vartheta}_{k})_{k\geq 0}$ to a continuous time Langevin diffusion
process with gradient potential
$\nabla \log \tilde{\pi}(\cdot |Z^{(n)})$. This reduces the problem of
computing the exit time of $(\tilde{\vartheta}_{k})_{k\geq 0}$ from
$\tilde{\mathcal{B}}$ to the corresponding exit time of the diffusion process.
The latter is achieved analytically by comparing to a suitable Ornstein--Uhlenbeck
process.
\begin{thm}
\label{thm:exitTime}%
Let
$\tau =\inf \{k\geq 1:\tilde{\vartheta}_{k}\notin \tilde{\mathcal{B}}
\}$ denote the first time the Markov chain
$(\tilde{\vartheta}_{k})_{k\geq 0}$ leaves from the region
$\tilde{\mathcal{B}}$. Grant Assumptions \ref{assu:a1}--\ref{assu:a4} and
suppose that $\gamma \leq m/(\sqrt{54}\Lambda ^{2})$,
$2c_{0}\Lambda _{\pi}\leq \eta m/32$. Then there exist
$c_{1},c_{2},c_{3},c_{4}>0$ such that on the event $\mathcal{E}$ for all
$J\geq 1$
\begin{equation*}
\mathbf{P} (\tau \leq J )\leq \frac{128Jp^{5/2}}{\eta \sqrt{m}}\exp \biggl(- \frac{\eta ^{2}m}{1024p}
\biggr).
\end{equation*}
\end{thm}

\begin{proof}
By extending the probability space carrying the Markov chain we can further
assume without loss of generality that it also supports a $p$-dimensional
Brownian motion $(W_{t})_{t\geq 0}$ with respect to a filtration
$(\mathcal{F}_{t})_{t\geq 0}$ satisfying the usual conditions (see
\citep[Section~5.2.A]{karatzas1998}). Set
$f(\theta )=-\log \tilde{\pi}(\cdot |Z^{(n)})$ and consider the two
$p$-dimensional stochastic differential equations
\begin{align}
dL_{t} =-\nabla f(L_{t})\,dt+\sqrt{2}\,dW_{t},\qquad
d\bar{L}_{t} =-\nabla f( \bar{L}_{\lfloor t/\gamma \rfloor \gamma})\,dt+
\sqrt{2}\,dW_{t}, \label{eq:LangevinSDE} 
\end{align}
for $t\geq 0$, both starting at
$L_{0}=\bar{L}_{0}=\theta _{\operatorname{init}}$. By Theorem~\ref{thm:surrogate} the function $f$ is strongly convex and has Lipschitz
gradients on $\mathcal{E}$, so classical results for stochastic differential
equations (e.g., \citep[Theorem~5.2.9]{karatzas1998}) verify that the first
SDE has a unique strong solution $(L_{t})_{t\geq 0}$ with respect to the
filtration $(\mathcal{F}_{t})_{t\geq 0}$. The process
$(\bar{L}_{t})_{t\geq 0}$ is the continuous time interpolation of
$(\tilde{\vartheta}_{k})_{k\geq 0}$ in the sense that
$ \mathcal{L}(\bar{L}_{\gamma},\dots ,\bar{L}_{J\gamma})=\mathcal{L}(
\vartheta _{1},\dots ,\vartheta _{J})$. Hence,
\begin{align*}
& \mathbf{P} (\tau \leq J )=\mathbf{P} \Bigl(\sup_{k=1,
\dots ,J}\lVert
\bar{L}_{k\gamma }-\theta _{*,p}\rVert >3\eta /8 \Bigr).
\end{align*}
We study this probability by comparing $(\bar L_{t})_{t\geq 0}$ to
$(L_{t})_{t\geq 0}$. We do this in three steps.

\emph{Step~1.} Recall that $p$-dimensional Brownian motion for
$p\geq 2$ does not hit points $\mathbf{P}$-almost surely
\citep[Theorem~18.6]{kallenberg2002}. By Girsanov's theorem this also holds
for the diffusion process $L$ on any finite time interval. We can therefore
apply It\^{o}'s formula to the function
$\theta \mapsto \lVert \theta -\theta _{*,p}\rVert $ (which is only non-smooth
at the point $\theta _{*,p}$) such that
\begin{equation*}
\lVert L_{t}-\theta _{*,p}\rVert =
\int _{0}^{t} \biggl\{ - \frac{L_{s}-\theta _{*,p}}{\lVert L_{s}-\theta _{*,p}\rVert }\cdot
\nabla f(L_{s})+\frac{1}{2} \frac{p-1}{\lVert L_{s}-\theta _{*,p}\rVert } \biggr\} \,ds+
\sqrt{2} \tilde{W}_{t},
\end{equation*}
where
$\tilde{W}_{t}=\int _{0}^{t}(L_{s}-\theta _{*,p})\lVert L_{s}-\theta _{*,p}
\rVert ^{-1}\,dW_{s}$ is a scalar Brownian motion by L\'{e}vy's characterisation
of Brownian motion. The strong convexity of $f$ implies
\begin{align*}
(\theta -\theta _{*,p})\cdot \nabla f(\theta ) & \geq (\theta - \theta
_{*,p})\cdot \nabla f(\theta _{*,p})+(m/2)\lVert \theta -
\theta _{*,p}\rVert ^{2}
\\
& =(m/2) (\theta -\theta _{*,p})\cdot (\theta -\bar{\theta}),\qquad
\bar{\theta}:=\theta _{*,p}+(2/m)\nabla f(\theta _{*,p}).
\end{align*}
Let $(V_{t})_{t\geq 0}$ be a $p$-dimensional Ornstein--Uhlenbeck process
satisfying
\begin{align*}
dV_{t} & =-(m/2) (V_{t}-\bar{\theta})\,dt+\sqrt{2}d
\tilde{W}_{t},\qquad V_{0}= \theta _{\operatorname{init}}.
\end{align*}
By a comparison argument for scalar It\^{o} processes
\citep{ikeda1977} we get
%
\begin{equation}
\lVert L_{t}-\theta _{*,p}\rVert \leq \lVert
V_{t}-\theta _{*,p} \rVert \quad \mathbf{P}\text{-almost
surely for all }t\geq 0. \label{eq:comparison} 
\end{equation}
The process $(V_{t})_{t\geq 0}$ has the explicit solution
\begin{align*}
V_{t} & =\theta _{\operatorname{init}}e^{-(m/2)t}+\bar{\theta}
\bigl(1-e^{-(m/2)t}\bigr)+ \sqrt{2/m}\tilde{W}_{1-e^{-mt}}
\\
& =\theta _{*,p}+(\theta _{\operatorname{init}}-\theta _{*,p})e^{-(m/2)t}+(2/m)
\nabla f(\theta _{*,p}) \bigl(1-e^{-(m/2)t}\bigr)+\sqrt{2/m}
\tilde{W}_{1-e^{-mt}}.
\end{align*}
Since
$\lVert \theta _{*,p}\rVert ,\lVert \theta _{\pi ,\operatorname{max}}
\rVert \leq c_{0}$, we have
$\lVert \nabla \log \pi (\theta _{*,p})\rVert =\lVert \nabla \log
\pi (\theta _{*,p})\rVert -\lVert \nabla \log \pi (\theta _{\pi ,
\operatorname{max}})\rVert \leq 2c_{0}\Lambda _{\pi}$. Assumption~\ref{assu:a1} as well as
$m\geq c_{\operatorname{min}}np^{-\kappa _{3}}$ imply
$\lVert \nabla \ell _{n}(\theta _{*,p})\rVert \leq c_{\max}n\delta _{n}p^{
\kappa _{1}} \leq \eta m/(4\log (n+e^{8}))\leq \eta m/32$. Hence, using
the condition on $\Lambda _{\pi}$ we find
$\lVert \nabla f(\theta _{*,p})\rVert \leq \lVert \nabla \ell _{n}(
\theta _{*,p})\rVert +\lVert \nabla \log \pi (\theta _{*,p})\rVert
\leq \eta m/16$. Since also
$\lVert \theta _{\operatorname{init}}-\theta _{*,p}\rVert \leq \eta /8$
by Assumption~\ref{assu:a2} and because $ac+a(1-c)=a$ for
$a,c\in \mathbb{R}$ we have
\begin{equation*}
\lVert \theta _{\operatorname{init}}-\theta _{*,p}\rVert
e^{-(m/2)t}+(2/m) \bigl\lVert \nabla f(\theta _{*,p})\bigr\rVert
\bigl(1-e^{-(m/2)t}\bigr)\leq \eta /8.
\end{equation*}
With this conclude from (\ref{eq:comparison})
\begin{align*}
 \mathbf{P} \Bigl(\sup_{0\leq t\leq J\gamma}\lVert L_{t}-\theta
_{*,p} \rVert >3\eta /16 \Bigr)
&\leq \mathbf{P} \Bigl(\sup
_{0\leq t\leq J
\gamma}\lVert \sqrt{2/m}\tilde{W}_{1-e^{-mt}}\rVert >\eta
/16 \Bigr)\\
& \leq p\mathbf{P} \Bigl(\sup_{0\leq s\leq t} \llvert
\tilde{W}_{1,s} \rrvert >b \Bigr)
\end{align*}
where $t=1-e^{-mJ\gamma}$, $b=\sqrt{m/(2p)}\eta /16$, and where
$\tilde{W}_{1,s}$ is a scalar Brownian motion. We can now apply a well-known
result on the exit time of a scalar Brownian motion from an interval, cf.
\citep[Remark~2.8.3]{karatzas1998}, which gives
\begin{align*}
\mathbf{P} \Bigl(\sup_{0\leq t\leq J\gamma}\lVert L_{t}-\theta
_{*,p} \rVert >3\eta /16 \Bigr) & \leq p\bigl(\sqrt{2t}/(b\sqrt{\pi})
\bigr)e^{-b^{2}/(2t)}\\
& \leq \frac{32p^{3/2}}{\eta \sqrt{m}}\exp \biggl(- \frac{\eta ^{2}m}{1024p(1-e^{-mJ\gamma})}
\biggr).
\end{align*}

\emph{Step~2.} We apply Lemma~22 of \citep{durmus2019} (their equation (51))
to the strongly convex function $f$ and
$\kappa =(2m\Lambda )/(m+\Lambda )$, $\varepsilon =\kappa /4$ such that
for all $k\geq 1$
\begin{align*}
\lVert L_{k\gamma }-\bar{L}_{k\gamma }\rVert ^{2} & \leq (1-
\gamma \kappa /2)\lVert L_{(k-1)\gamma }-\bar{L}_{(k-1)\gamma }\rVert
^{2}
\\
&\quad{}  +(\gamma +2/\kappa )
\int _{(k-1)\gamma}^{k\gamma}\bigl\lVert \nabla
f(L_{s})- \nabla f(L_{(k-1)\gamma })\bigr\rVert ^{2}\,ds.
\end{align*}
Since $L_{0}=\bar{L}_{0}$, this yields inductively
\begin{align*}
\lVert L_{k\gamma }-\bar{L}_{k\gamma }\rVert ^{2} & \leq (
\gamma +2/ \kappa )\sum_{i=1}^{k}(1-\gamma
\kappa /2)^{k-i}
\int _{(i-1)\gamma}^{i
\gamma}\bigl\lVert \nabla
f(L_{s})-\nabla f(L_{(i-1)\gamma })\bigr\rVert ^{2}\,ds.
\end{align*}
Now, $\nabla f$ is $\Lambda $-Lipschitz, and so using
$\gamma \leq m/(\sqrt{54}\Lambda ^{2})\leq \Lambda ^{-1}$,
$\kappa \geq m$ and recalling that $L$ solves the first SDE in
\eqref{eq:LangevinSDE}, we have with
$V=\sup_{0\leq t\leq J\gamma}\lVert W_{t}-W_{\lfloor t/\gamma
\rfloor \gamma }\rVert $ that
\begin{align*}
\lVert L_{k\gamma }-\bar{L}_{k\gamma }\rVert &\leq \sqrt{2\bigl(
\gamma / \kappa +2/\kappa ^{2}\bigr)}\Lambda \sup_{0\leq t\leq k\gamma}
\lVert L_{t}-L_{
\lfloor t/\gamma \rfloor \gamma }\rVert
\\
&  \leq (\sqrt{6}\Lambda /m)\sup_{0\leq t\leq k\gamma} \biggl(
\int _{\lfloor t/\gamma \rfloor \gamma}^{t}\bigl\lVert \nabla
f(L_{s}) \bigr\rVert \,ds+\sqrt{2}\lVert W_{t}-W_{\lfloor t/\gamma \rfloor \gamma }
\rVert \biggr)
\\
&  \leq (\sqrt{6}\Lambda /m)\sup_{0\leq t\leq k\gamma} \biggl(
\int _{\lfloor t/\gamma \rfloor \gamma}^{t}\bigl\lVert \nabla
f(L_{s})- \nabla f(\theta _{*,p})\bigr\rVert \,ds+\gamma
\bigl\lVert \nabla f(\theta _{*,p}) \bigr\rVert \biggr)\\
&\quad +(\sqrt{12}
\Lambda /m)V
\\
&  \leq \bigl(\sqrt{6}\Lambda ^{2}\gamma /m\bigr)\sup
_{0\leq t\leq J
\gamma}\lVert L_{t}-\theta _{*,p}\rVert +
\sqrt{6}\Lambda \gamma \eta /8+\bigl( \sqrt{12}\Lambda \gamma ^{1/2}/m
\bigr)\gamma ^{-1/2}V
\\
& \leq (1/3)\sup_{0\leq t\leq J\gamma}\lVert L_{t}-\theta
_{*,p} \rVert +\eta /16+(2/m)^{1/2}\gamma ^{-1/2}V,
\end{align*}
because $\lVert \nabla f(\theta _{*,p})\rVert \leq \eta m/8$ by Step~1
and using $\gamma \leq m/(\sqrt{54}\Lambda ^{2})$ in the last line. Observe
that
\begin{align*}
V & \overset{d} {=}\gamma ^{1/2}\sup_{0\leq t\leq J}\lVert
W_{t}-W_{
\lfloor t\rfloor }\rVert \leq \gamma ^{1/2}p^{1/2}
\max_{1\leq i\leq p,1
\leq k\leq J}\sup_{k-1\leq t\leq k} \llvert
W_{i,t}-W_{i,\lfloor t
\rfloor} \rrvert ,
\end{align*}
as well as
$(W_{i,t}-W_{i,\lfloor t\rfloor})_{k-1\leq t\leq k}\overset{d}{=}(
\tilde{W}_{1,t})_{0\leq t\leq 1}$. Consequently,
\begin{align*}
&\mathbf{P} \Bigl(\sup_{k=1,\dots ,J}\lVert L_{k\gamma }
- \bar{L}_{k\gamma }\rVert >3\eta /16 \Bigr)\\
& \quad \leq \mathbf{P} \Bigl((1/3)
\sup_{0\leq t\leq J\gamma}\lVert L_{t}-\theta _{*,p}\rVert
+(2/m)^{-1/2} \gamma ^{-1/2}V>\eta /8 \Bigr)
\\
& \quad \leq \mathbf{P} \Bigl(\sup_{0\leq t\leq J\gamma}\lVert L_{t}-
\theta _{*,p}\rVert >3\eta /16 \Bigr)+Jp\mathbf{P} \Bigl(\sup
_{0
\leq s\leq 1} \llvert \tilde{W}_{1,s} \rrvert >b \Bigr),
\end{align*}
with $b$ as in Step~1.\vadjust{\goodbreak}

\emph{Step~3.} Conclude by
$\lVert \bar{L}_{k\gamma }-\theta _{*,p}\rVert \leq \lVert L_{k
\gamma }-\theta _{*,p}\rVert +\lVert L_{k\gamma }-\bar{L}_{k\gamma }
\rVert $, the first two steps and $1-e^{-mJ\gamma}\leq 1$.
\end{proof}

Combining the last two theorems gives the following important result.

\begin{cor}
\label{cor:surrogateTruePosterior}%
Let $(\vartheta _{k})_{k\geq 1}$ be the Markov chain with iterates
\eqref{eq:langevin}. Let $\varepsilon $, $J_{\operatorname{in}}$ be as
in Theorem~\ref{thm:MCMC_surrogate_functionals} and suppose that
$\gamma \leq m/(\sqrt{54}\Lambda ^{2})$,
$2c_{0}\Lambda _{\pi}\leq \eta m/32$,
$1\leq n\delta _{n}^{2}\leq \eta ^{2}m/p$. Under Assumptions
\ref{assu:a1}-\ref{assu:a4} there exist $c'_{F},c_{1},c_{2}>0$ and an event
$\bar{\mathcal{E}}$ with
$\mathbb{P}_{\theta _{0}}^{n}(\bar{\mathcal{E}})\geq 1-c_{2}e^{-c_{1}n
\delta _{n}^{2}}$ such that for all Lipschitz functions
$h:\mathbb{R}^{p}\rightarrow \mathbb{R}$ with
$\lVert h\rVert _{\mathrm{Lip}}=1$ and all
$J\geq (2m\gamma )^{-1}$
\begin{align*}
& \mathbf{P} \Biggl( \Biggl\llvert \frac{1}{J}\sum
_{k=1+J_{\operatorname{in}}}^{J+J_{
\operatorname{in}}}h(\vartheta _{k})-
\int _{\Theta}h(\theta )\pi \bigl( \theta |Z^{(n)}\bigr)d
\theta \Biggr\rrvert >\varepsilon \Biggr)\\
&\quad \leq 2\exp \bigl(-c'_{F}
\varepsilon ^{2}m^{2}J\gamma \bigr)+128(J+J_{
\operatorname{in}})p^{2}
\exp \biggl(-\frac{n\delta _{n}^{2}}{1024} \biggr).
\end{align*}
\end{cor}

\begin{proof}
If $A$ is the event whose probability we want to upper bound, then
$\mathbf{P}(A)\leq \mathbf{P}(A,\tau >J+J_{\operatorname{in}})+
\mathbf{P}(\tau \leq J+J_{\operatorname{in}})$. Since
$\tilde{\vartheta}_{k}=\vartheta _{k}$ for all
$k\leq J+J_{\operatorname{in}}<\tau $, the result follows immediately from
Theorems \ref{thm:MCMC_surrogate_functionals} and \ref{thm:exitTime}, noting
$mJ\gamma \geq 1/2$.\break\mbox{}
\end{proof}

As long as we make at most
$(J+J_{\operatorname{in}})p^{2}\leq \exp (n\delta _{n}^{2}/1025)$ many
steps along the Markov chain $(\vartheta _{k})_{k\geq 1}$ and if
$\varepsilon ^{2}m^{2}J\gamma \geq n\delta _{n}^{2}$, then this result
shows that we can compute posterior functionals up to a precision
$\varepsilon $ with high probability \emph{jointly} under
$\mathbb{P}_{\theta _{0}}^{n}\times \mathbf{P}$ using the vanilla Langevin
algorithm. Assuming that in practice we only have a computational budget
of iterations $J\asymp n^{\rho}$, $\rho >0$, depending polynomially on
$n$, we see that the convexification by the surrogate log-likelihood is
not necessary, at least when the step size is small enough and given the
model configuration as in the Corollary.

In the final result of this section we recover the ground truth
$\theta _{0}$ through an ergodic average of
$(\vartheta _{k})_{k\geq 0}$ with high
$\mathbb{P}_{\theta _{0}}^{n}\times \mathbf{P}$-probability by taking
$f(\theta )=\theta $ in the previous Corollary and assuming that
$\theta _{*,p}$ is close to $\theta _{0}$ up to a bias
$\lVert \theta _{0}-\theta _{*,p}\rVert ^{\beta}\leq C\delta _{n}$. Due
to the presence of this bias it is reasonable to restrict to the sampling
to target precision levels $\varepsilon \geq c\delta _{n}$.

\begin{cor}
\label{cor:PosteriorMean}%
Grant the assumptions of Corollary~\ref{cor:surrogateTruePosterior} assume
in addition
$\lVert \theta _{0}-\theta _{*,p}\rVert ^{\beta}\leq C\delta _{n}$ for
$C>0$. Let
$\varepsilon \geq \max (c\delta _{n}^{1/\beta},p^{1/2}\sqrt{48e^{-n
\delta _{n}^{2}}+8B(\gamma )})$ for some large enough $c>0$. Then there
exist $c,c_{1},c_{2},c_{3},c_{4}>0$ such that on an event
$\bar{\mathcal{E}}$ with
$\mathbb{P}_{\theta _{0}}^{n}(\bar{\mathcal{E}})\geq 1-c_{1}e^{-c_{2}n
\delta _{n}^{2}}$ for all $J\geq (2m\gamma )^{-1}$
\begin{align*}
& \mathbf{P} \Biggl(\Biggl\lVert \frac{1}{J}\sum
_{k=1+J_{\operatorname{in}}}^{J+J_{
\operatorname{in}}}\vartheta _{k}-\theta
_{0}\Biggr\rVert >2\varepsilon \Biggr)\leq 2\exp
\bigl(-c'_{F}\varepsilon ^{2}m^{2}J
\gamma /p \bigr)+128(J+J_{\operatorname{in}})p^{3}\exp \biggl(-
\frac{n\delta _{n}^{2}}{1024} \biggr).
\end{align*}
\end{cor}

\begin{proof}
We first demonstrate that the posterior mean is close to the ground truth
with high $\mathbb{P}^{n}_{\theta _{0}}$-probability. The proof adapts
arguments from \citep{monard2021a} to our setting. The Jensen inequality
shows for $L>0$
\begin{align*}
\biggl\lVert
\int _{\Theta }\theta \pi \bigl(\theta |Z^{(n)}\bigr)d
\theta -\theta _{0} \biggr\rVert &\leq
\int _{\Theta}\lVert \theta -\theta _{0}\rVert \pi
\bigl( \theta |Z^{(n)}\bigr)\,d\theta
\\
& \leq
\int _{\lVert \theta -\theta _{*,p}\rVert ^{\beta}> L\delta _{n}} \lVert \theta -\theta _{*,p}\rVert \pi
\bigl(\theta |Z^{(n)}\bigr)\,d\theta +L^{1/
\beta}\delta
_{n}^{1/\beta}+\lVert \theta _{*,p}-\theta
_{0}\rVert .
\end{align*}
Arguing as in the proof of Theorem~\ref{thm:surrogateWasserstein} we find
for large enough $L$ that the first term exceeds
$e^{-n\delta _{n}^{2}}$ only with exponentially small
$\mathbb{P}^{n}_{\theta _{0}}$-probability. Together with the bias condition
we thus find
$\lVert \int _{\Theta }\theta \,d\pi (\theta |Z^{(n)})\,d\theta -\theta _{0}
\rVert \leq c'\delta _{n}^{1/\beta}$ with high
$\mathbb{P}^{n}_{\theta _{0}}$-probability for some $c'>0$. Hence, we get
for large enough $c$ with the coordinate functions $h_{i}(x)=x_{i}$
\begin{align*}
&\mathbf{P} \Biggl(\Biggl\lVert \frac{1}{J}\sum
_{k=1+J_{\operatorname{in}}}^{J+J_{
\operatorname{in}}}\vartheta _{k}-\theta
_{0}\Biggr\rVert >2\varepsilon \Biggr) 
\leq \mathbf{P} \Biggl(\Biggl
\lVert \frac{1}{J}\sum_{k=1+J_{
\operatorname{in}}}^{J+J_{\operatorname{in}}}
\vartheta _{k}-
\int _{
\Theta }\theta \pi \bigl(\theta |Z^{(n)}\bigr)d
\theta \Biggr\rVert >\varepsilon \Biggr)
\\
 &\quad  \leq p\max_{i=1,\dots ,p}\mathbf{P} \Biggl(\frac{1}{J}
\sum_{k=1+J_{
\operatorname{in}}}^{J+J_{\operatorname{in}}}h_{i}(\vartheta
_{k})-
\int _{\Theta}h_{i}(\theta )\pi \bigl(\theta
|Z^{(n)}\bigr)\,d\theta > \varepsilon /p^{1/2} \Biggr),
\end{align*}
Conclude by Corollary~\ref{cor:surrogateTruePosterior}.
\end{proof}

\section{Applications}
\label{sec:Applications}

We next apply the general sampling results to several important statistical
models. For clarity, we will make a number of simplifying assumptions throughout,
and restrict to the vanilla Langevin algorithm \eqref{eq:langevin}. Our
focus lies on exhibiting sampling guarantees scaling polynomially in
$n$ and $p$, and we therefore avoid optimising constants.

Let us first fix some notation. For $\alpha \geq 0$ introduce the
$\ell ^{2}(\mathbb{N})$-Sobolev spaces
\begin{equation*}
h^{\alpha}(\mathbb{N})= \Biggl\{ \theta \in \ell ^{2}(
\mathbb{N}): \lVert \theta \rVert _{\alpha}^{2}=\sum
_{k=1}^{\infty}k^{2\alpha} \theta
_{k}^{2}<\infty \Biggr\}.
\end{equation*}
In the following, let $\mathcal{X}$ be an open and bounded subset of
$\mathbb{R}^{d}$ and $\nu _{X}$ is a finite measure. We always identify
$\theta \in \ell ^{2}(\mathbb{N})$ with the function
$\Phi (\theta )=\sum_{k\geq 1}\theta _{k} e_{k}$ with respect to an orthonormal
basis $(e_{k})_{k\geq 1}$ of $L^{2}(\mathcal{X})$. We choose for
$(e_{k})_{k\geq 1}$ specifically the eigenbasis of the Dirichlet Laplacian
(cf. \cite[Section~5A]{taylormichael2010} or
\cite[Chapter~6.5]{evans2010}). With this choice, the spaces
$h^{\alpha}(\mathbb{N})$ are naturally isomorphic to
$L^{2}(\mathcal{X})$-Sobolev spaces (see Section~\ref{subsec:Sobolev} for
details), thus simplifying some proofs. For example, when
$\mathcal{X}=(0,1)$, then we obtain the standard Fourier basis
$e_{k}(x)=\sin (k\pi x)$
\citep[Section~7.7.1]{pinchover2005introduction}.

Suppose that the observations $Z_{i}$ are i.i.d. and that they are generated
according to $\mathbb{P}_{\theta _{0}}^{n}$ for a parameter
$\theta _{0}\in h^{\alpha}(\mathbb{N})$ and some $\alpha $. We always assume
\begin{align*}
\delta _{n}=n^{-\alpha /(2\alpha +1)},\qquad \theta _{*,p}=(\theta
_{0,1}, \dots ,\theta _{0,p}),
\end{align*}
so $\delta _{n}$ is the \emph{usual} statistical minimax rate of convergence
for $\alpha $-regular signals \cite{tsybakov2008}. Note that
$n\delta _{n}^{2}=n^{1/(2\alpha +1)}$. A~popular `sieve' prior distribution
puts independent scalar Gaussian priors with decreasing variances on the
first $p$ coefficients $\theta \in \mathbb{R}^{p}$,
\begin{equation}
\theta \sim \Pi \equiv \Pi _{n}=N\bigl(0,n^{-1/(2\alpha +1)}\Sigma
_{
\alpha}^{-1}\bigr),\qquad \Sigma _{\alpha}=
\operatorname{diag}\bigl(1,2^{2\alpha}, \dots ,p^{2\alpha}\bigr).
\label{eq:SievePrior} 
\end{equation}
The results below extend naturally to more general series-type prior distributions
(such as in \cite[Section~11.4.5]{ghosal2017} or \cite{agapiou2021a}) and
basis functions $(e_{k})_{k\geq 1}$. The data dependent rescaling of the
prior ensures in the proof of Theorem~\ref{thm:GaussianContraction} that
posterior draws concentrate with high probability on balls
$\{\theta \in \mathbb{R}^{p}:\lVert \theta \rVert _{\alpha}\leq r\}$ for
some fixed radius. This technique has been essential in many recent results
for nonlinear Bayesian inverse to prove posterior contraction based on
stability bounds for the forward operator such as in
Assumption~\ref{assu:GLM_local_reg}(iv) below. In the next sections we apply this
technique for the first time to density estimation and GLMs. Differently
from common assumptions in the literature (e.g.,
\cite[(2.6)]{ghosal1997}), this allows for a-priori unbounded parameter
spaces.

We first state two results that will allow us to verify the high-probability
growth conditions on the log-likelihood function and the posterior contraction
in Assumptions \ref{assu:a1} and \ref{assu:a3} from moment and stability
bounds. The proof of the first theorem can be found in Section~\ref{subsec:Sufficient-moment-conditions} and is based on the classical
Bernstein inequality (see, e.g., Proposition~3.1.8 in
\citep{gine2016}) and a chaining argument for empirical processes with
mixed tails, cf. Theorem~3.5 of \citep{dirksen2015a}. Here, we denote by
$\ell (\theta )=\log p_{\theta}(Z_{i})$ the log-likelihood function for
a single observation.

\begin{thm}
\label{thm:localisation}%
Assume there exist $C>0$, $C_{1}\geq C_{2}>0$,
$\kappa _{1},\kappa _{2},\kappa _{3},\kappa _{4},n_{0}\geq 0$ such that
the following conditions hold for all $n\geq n_{0}$ and all
$v\in \mathbb{R}^{p}$ with $\lVert v\rVert =1$:
\begin{enumerate}[label=(\roman*)]
\item[{(i)}] (local regularity)
$\theta \mapsto \ell (\theta )\in C^{2}(\mathcal{B})$,
$\mathbb{P}_{\theta _{0}}$-almost surely.
\item[{(ii)}] (local mean boundedness) $|\mathbb{E}_{\theta _{0}}v^{\top}
\nabla \ell (\theta _{*,p})|\leq C_{1}\delta _{n}p^{\kappa _{1}}$ and for
all $q\geq 2$
\begin{align*}
\mathbb{E}_{\theta _{0}} \bigl\llvert v^{\top}\nabla \ell (\theta
_{*,p}) \bigr\rrvert ^{q} &\leq (q!/2)C_{1}^{q}p^{2\kappa _{1}+(q-2)\kappa _{2}},
\\
\sup_{\theta \in \mathcal{B}}\mathbb{E}_{\theta _{0}} \bigl\llvert
v^{\top} \nabla ^{2}\ell (\theta )v \bigr\rrvert
^{q}&\leq (q!/2)C_{1}^{q}p^{2\kappa _{2}+(q-2)
\kappa _{4}},
\\
\sup_{\theta ,\theta '\in \mathcal{B}}\mathbb{E}_{\theta _{0}} \bigl\llvert
v^{\top}\nabla ^{2}\bigl(\ell (\theta )-\ell \bigl(\theta
'\bigr)\bigr)v \bigr\rrvert ^{q}&\leq (q!/2)
\bigl(C_{1}p^{\kappa _{4}}\bigl\lVert \theta -\theta '
\bigr\rVert \bigr)^{q}.
\end{align*}
\item[{(iii)}] (local mean curvature)
$\inf_{\theta \in \mathcal{B}}\lambda _{\min}  (\mathbb{E}_{
\theta _{0}}  [-\nabla ^{2}\ell (\theta )  ]  )\geq C_{2}p^{-
\kappa _{3}}$.
\item[{(iv)}] (growth conditions) $p\leq Cn\delta _{n}^{2}$,
$\max (\delta _{n}p^{\kappa _{2}},\eta \delta _{n}p^{\kappa _{4}},
\delta _{n}^{2}p^{\kappa _{4}})\log n\leq p^{-\kappa _{3}}$.
\end{enumerate}
Then Assumption~\ref{assu:a1} holds with
$c_{\operatorname{min}}=C_{2}/2$,
$c_{\operatorname{max}}=\max (2C_{1}\sqrt{8C},C_{2}/2)+C_{1}$ for some
event $\mathcal{E}$, provided the lower bound on $\eta $ stated there is
satisfied.
\end{thm}

The second theorem (with a proof in Section~\ref{sec:proofGaussContraction}) generalises
\citep[Theorem~7.3.1]{gine2016} and \citep[Theorem~13]{giordano2020b},
using Bernstein-type moment conditions, entropy bounds and a stability
condition for the Hellinger distance. Define for
$\theta ,\theta '\in \Theta $ the squared Hellinger distance
\begin{align*}
h^{2}\bigl(\theta ,\theta '\bigr) & =
\int _{\mathcal{Z}} \bigl(\sqrt{p_{\theta}(z)}-
\sqrt{p_{\theta '}(z)} \bigr)^{2}\,d\nu (z).
\end{align*}

\begin{thm}
\label{thm:GaussianContraction}%
Let $\theta _{0}\in h^{\alpha}(\mathbb{N})$ with
$\lVert \theta _{0}\rVert _{\alpha}\leq c_{0}$ for $\alpha >1/2$,
$c_{0}>0$. Let $\Pi (\cdot |Z^{(n)})$ be the posterior distribution with
the rescaled Gaussian prior $\Pi $ from \eqref{eq:SievePrior}. For
$C>0$ suppose $p\leq Cn^{1/(2\alpha +1)}$,
$\lVert \theta _{0}-\theta _{*,p}\rVert \leq C\delta _{n}$ and set
\begin{equation*}
\mathcal{B}_{n,r}= \bigl\{ \theta \in \mathbb{R}^{p}:\lVert
\theta - \theta _{*,p}\rVert \leq \delta _{n},\lVert \theta
\rVert _{\alpha} \leq r \bigr\} ,\quad r>0.
\end{equation*}
Suppose further that there exists $\beta \geq 1$ such that for any
$r>0$ and some $c_{r}>0$ the following holds:
\begin{enumerate}[label=(\roman*)]
\item[{(i)}] (moment conditions) For all $\theta \in \mathcal{B}_{n,r}$ and all
$q\ge 2$
\begin{equation}
\mathbb{E}_{\theta _{0}}\bigl(\ell (\theta _{0})-\ell (\theta )
\bigr)\leq c_{r} \delta _{n}^{2},\qquad
\mathbb{E}_{\theta _{0}} \bigl\llvert \ell (\theta _{0})- \ell (
\theta ) \bigr\rrvert ^{q}\leq (q!/2)\delta _{n}^{2}c_{r}^{q}.
\label{eq:ApproxCondition} 
\end{equation}
\item[{(ii)}] (stability condition) For all
$\theta ,\theta '\in h^{\alpha}(\mathbb{N})$ with
$\lVert \theta \rVert _{\alpha},\lVert \theta '\rVert _{\alpha}\leq r$,
%
\begin{equation}
c_{r}^{-1}\bigl\lVert \theta -\theta '\bigr
\rVert ^{\beta}\leq h\bigl(\theta , \theta '\bigr)\leq
c_{r}\bigl\lVert \theta -\theta '\bigr\rVert .
\label{eq:hellingerCondition} 
\end{equation}
\end{enumerate}
Then the posterior distribution contracts around $\theta _{*,p}$ at the
rate $\delta _{n}$ and satisfies Assumption~\ref{assu:a3}.
\end{thm}

\begin{rem}[on $p$, $\alpha $ and the bias condition]%
\label{rem:bias}
Let $\theta _{0}\in h^{\alpha '}(\mathbb{N})$ for
$\alpha '\in \mathbb{N}$, so
$\lVert \theta _{0}-\theta _{*,p}\rVert \leq p^{-\alpha '}
\lVert \theta _{0}\rVert _{\alpha '}$. Under the model restriction
$p\leq Cn^{1/(2\alpha +1)}$ this means for a generic
$\theta _{0}\in h^{\alpha}(\mathbb{N})$ that the optimal rate of convergence
$\delta _{n}$ is achieved in Theorem~\ref{thm:GaussianContraction} only
when $p=Cn^{1/(2\alpha +1)}$. If
$\limsup_{n\rightarrow \infty}p/n^{1/(2\alpha +1)}=0$, then the bias condition
$\lVert \theta _{0}-\theta _{*,p}\rVert \leq C\delta _{n}$ is satisfied
for specific $\theta _{0}$ or when
$n^{(\alpha /\alpha ')/(2\alpha +1)}\lesssim p$.
\end{rem}

\subsection{Nonparametric regression}
\label{subsec:Nonparametric-regression}

In this section we translate the abstract Assumptions in Theorems
\ref{thm:localisation} and \ref{thm:GaussianContraction} to nonparametric
regression models with general assumptions on the regression function.
Posterior sampling guarantees will be formulated for concrete regression
models in Sections~\ref{sec:GLM} and \ref{subsec:Darcy's-problem-1}.

As in Example~\ref{exa:GLM} we observe i.i.d. random vectors
$Z_{i}=(Y_{i},X_{i})$, $i=1,\dots ,n$, such that $Y_{i}$ follows conditional
on $X_{i}=x$ are one-parameter exponential family, but this time assume
more generally
\begin{equation*}
g\bigl(\mathbb{E}_{\theta}[Y_{i}|X_{i}]\bigr)=
\mathcal{G}(\theta ) (X_{i})
\end{equation*}
for a link function $g$ and a known \emph{forward operator}
$\mathcal{G}:\Theta \mapsto L^{2}(\mathcal{X})$. Set
$b_{\theta}(x)=(A')^{-1}  (g^{-1}  (\mathcal{G}(\theta )(x)
  )  )$. With $p_{X}$ the $\nu _{\mathcal{X}}$-density of the
$X_{i}$ the log-likelihood function is
\begin{equation}
\ell _{n}(\theta )=\sum_{i=1}^{n}
\bigl(Y_{i}b_{\theta}(X_{i})-A
\bigl(b_{\theta}(X_{i}) \bigr) \bigr)+\sum
_{i=1}^{n}\log p_{X}(X_{i}).
\label{eq:GLM_logLik}
\end{equation}
The observations follow a nonparametric regression model with errors
$\varepsilon _{i}=Y_{i}-g^{-1}(\mathcal{G}(\theta )(X_{i}))$ such
that
\begin{align}
Y_{i}=g^{-1}\bigl(\mathcal{G}(\theta ) (X_{i})
\bigr)+\varepsilon _{i},\qquad \mathbb{E}_{\theta}[\varepsilon
_{i}|X_{i}]=0. \label{eq:nonGaussRegression}
\end{align}
Recall the high-dimensional region
\begin{equation*}
\mathcal{B}= \bigl\{ \theta \in \mathbb{R}^{p}:\lVert \theta -\theta
_{*,p} \rVert \leq \eta \bigr\}.
\end{equation*}
The moment bounds in Theorem~\ref{thm:localisation} correspond to
$L^{\infty}(\mathcal{X})$- and $L^{2}(\mathcal{X})$-bounds for
$\mathcal{G}$ on $\mathcal{B}$, and the curvature lower bound to a lower
bound on the gradient of $\mathcal{G}$, while the stability condition in
Theorem~\ref{thm:GaussianContraction} relies on a local stability bound
and a local Lipschitz condition.
\setcounter{myassumption}{6}
\renewcommand{\themyassumption}{\Alph{myassumption}}
\begin{myassumption}[local regularity and curvature]%
\label{assu:GLM_local_reg}
$\lVert \theta _{0}\rVert \leq c_{0}$ for $c_{0}>0$ and
$\alpha >1/2$.
\begin{enumerate}[label=(\roman*)]
\item[{(i)}] For all $x\in \mathcal{X}$,
$\theta \mapsto \mathcal{G}(\theta )(x)\in C^{2}(\mathcal{B})$.
\item[{(ii)}] There exist $\bar{c}_{\max}\geq 1$, $k_{1},\dots ,k_{4}\geq 0$ such
that for all $\theta \neq \theta '\in \mathcal{B}$:
\begin{align*}
& \bigl\lVert \mathcal{G}(\theta )-\mathcal{G}(\theta _{*,p})\bigr
\rVert _{L^{2}} \leq \bar{c}_{\max}\eta ,
\\
& \bigl\lVert \mathcal{G}(\theta )\bigr\rVert _{L^{\infty}}\leq
\bar{c}_{\max}, \qquad \bigl\lVert \nabla \mathcal{G}(\theta )\bigr\rVert
_{L^{\infty}(
\mathcal{X},\mathbb{R}^{p})}\leq \bar{c}_{\max}p^{k_{1}},
\\
& \bigl\lVert \nabla ^{2}\mathcal{G}(\theta )\bigr\rVert
_{L^{\infty}(
\mathcal{X},\mathbb{R}^{p\times p})}, \qquad
\frac{\lVert \nabla ^{2}\mathcal{G}(\theta )-\nabla ^{2}\mathcal{G}(\theta ')
\rVert _{L^{\infty}(\mathcal{X},\mathbb{R}^{p\times p})}}{\lVert \theta -\theta '\rVert }
\leq \bar{c}_{\max}p^{k_{2}},
\\
& \bigl\lVert \nabla \mathcal{G}(\theta )\bigr\rVert _{L^{2}(\mathcal{X},
\mathbb{R}^{p})}\leq
\bar{c}_{\max}p^{k_{3}},\qquad \bigl\lVert \nabla ^{2}
\mathcal{G}(\theta )\bigr\rVert _{L^{2}(\mathcal{X},\mathbb{R}^{p\times p})}
\leq \bar{c}_{\max}p^{k_{4}}.
\end{align*}
\item[{(iii)}] There exist $\bar c_{\min}>0$, $k_{5}\geq 0$ such that
$\lVert \nabla \mathcal{G}(\theta )^\top v\rVert _{L^{2}}\geq \bar{\mathrm{c}}_{\min}p^{-k_{5}}$
for all $\theta \in \mathcal{B}$, $v\in\mathbb{R}^p$ with $\lVert v\rVert=1$.
\item[{(iv)}] There exists $\beta \geq 1$ such that if $r>0$ and
$\lVert \theta \rVert _{\alpha},\lVert \theta '\rVert _{\alpha}\leq r$,
then for some $\bar c_{r}>0$
\begin{equation*}
\bar c_{r}^{-1}\bigl\lVert \theta -\theta '
\bigr\rVert ^{\beta}\leq \bigl\lVert \mathcal{G}(\theta )-\mathcal{G}\bigl(
\theta '\bigr)\bigr\rVert _{L^{2}}\leq \bar c_{r}
\bigl\lVert \theta -\theta '\bigr\rVert ,\qquad \bigl\lVert
\mathcal{G}(\theta ) \bigr\rVert _{L^{\infty}}\leq \bar c_{r}.
\end{equation*}
\end{enumerate}
\end{myassumption}

\begin{prop}
\label{prop:ExpoSurrogate}%
Suppose the observations arise from a nonparametric regression model with
log-likelihood function \eqref{eq:GLM_logLik}. Let
$g\in C^{3}(\mathcal{I})$, $c_{X}^{-1}\leq p_{X}(x)\leq c_{X}$ for all
$x\in \mathcal{X}$ and some $c_{X}>0$. For $C>0$ suppose
$p\leq Cn^{1/(2\alpha +1)}$,
$\lVert \theta _{0}-\theta _{*,p}\rVert \leq C\delta _{n}$. Moreover, suppose
that the forward operator $\mathcal{G}$ satisfies Assumption~\ref{assu:GLM_local_reg}. Then, under the growth conditions
\begin{align*}
\max \bigl(\delta _{n}p^{\max (k_{1},2k_{3},k_{4})},\delta _{n}\max (
\eta , \delta _{n})p^{\max (3k_{1},k_{1}+k_{2})}\bigr)\log n &\leq p^{-2k_{5}},
\\
\max \bigl(\delta _{n}^{1/\beta},\delta _{n}p^{k_{3}+2k_{5}}
\bigr) \bigl(\log \bigl(n+e^{8}\bigr)\bigr)^{2} &\leq \eta ,
\end{align*}
for $n$ large enough Assumptions \ref{assu:a1} and \ref{assu:a3} are met
with $\kappa _{1}=k_{3}$, $\kappa _{2}=\max (k_{1},2k_{3},k_{4})$,
$\kappa _{3}=2k_{5}$ and some $C_{1},C_{2}>0$.
\end{prop}

A proof can be found in Section~\ref{sec:proofSpecificModels}.

\subsection{Generalised linear models}
\label{sec:GLM}

We return to the situation of Example~\ref{exa:GLM}, so the log-likelihood
is as in \eqref{eq:GLM_logLik} with
$\mathcal{G}(\theta )=\Phi (\theta )=\sum_{k=1}^{\infty}\theta _{k}e_{k}$.
Verifying Assumption~\ref{assu:GLM_local_reg} essentially reduces to finding
a radius $\eta $ such that $\mathcal{G}$ is uniformly bounded on
$\mathcal{B}$. Sufficient conditions are stated in the next lemma.

\begin{lem}
\label{lem:GLM_conditions}%
Suppose the data $Z^{(n)}$ arise in a GLM with smooth link function
$g$, $c_{X}^{-1}\leq p_{X}(x)\leq c_{X}$ for all
$x\in \mathcal{X}$ for some $c_{X}>0$, as well as
$\lVert \theta _{0}\rVert _{\alpha}\leq c_{0}$ for $\alpha >1$,
$c_{0}>0$. Let $p\leq Cn^{1/(2\alpha +1)}$,
$\lVert \theta _{0}-\theta _{*,p}\rVert \leq C\delta _{n}$ for $C>0$ and
$\eta =\iota p^{-\bar\alpha}\leq 1$ for $\iota ,\bar\alpha >0$ with
$\alpha >2\bar\alpha >1$. Then Assumptions \ref{assu:a1} and
\ref{assu:a3} are satisfied for large enough $n$ with
$\kappa _{1}=0$, $\kappa _{2}=\bar\alpha $, $\kappa _{3}=0$ with
$\beta =1$.
\end{lem}

\begin{proof}
It is enough to verify Assumption~\ref{assu:GLM_local_reg} with
$k_{1}=\bar\alpha $, $k_{2}=k_{3}=k_{4}=k_{5}=0$. Taking $n$ large enough
this yields the growth conditions in Proposition~\ref{prop:ExpoSurrogate}. Let $\theta \in \mathcal{B}$ and
$v\in \mathbb{R}^{p}$ with $\lVert v\rVert=1$. Then
$\nabla \mathcal{G}(\theta )^{\top }v=\Phi (v)$, so
$\lVert \nabla \mathcal{G}(\theta )^{
\top }v\rVert_{L^{2}} = 1$. Moreover, by the Sobolev embeddings in
\eqref{eq:Phi_C1}, \eqref{eq:Phi_C1_Rp} we have for
$\bar\alpha >1/2$
\begin{align}
\begin{aligned}
\bigl\lVert \nabla \mathcal{G}(\theta )\bigr\rVert _{L^{\infty}} & =\sup
_{v\in \mathbb{R}^{p}:\lVert v\rVert =1}\bigl\lVert \Phi (v)\bigr\rVert _{L^{
\infty}}\lesssim
p^{\bar\alpha},
\\
\bigl\lVert \mathcal{G}(\theta )\bigr\rVert _{L^{\infty}} & \lesssim \lVert
\theta \rVert _{\bar\alpha}\leq p^{\bar\alpha}\lVert \theta -\theta
_{*,p} \rVert + \lVert \theta _{*,p}\rVert _{\bar\alpha}
\leq \iota+c_{0}.
\end{aligned}
\label{eq:GLM_conditions_1} 
\end{align}
This verifies Assumption~\ref{assu:GLM_local_reg}(ii), (iii) with the
$k_{i}$ specified above. Assumption~\ref{assu:GLM_local_reg}(iv) follows
from
$\lVert \mathcal{G}(\theta )-\mathcal{G}(\theta ')\rVert _{L^{2}} =
\lVert \theta -\theta '\rVert $ with $\beta =1$ for
$\lVert \theta '\rVert _{\alpha}, \lVert \theta \rVert _{\alpha}\leq r$
such that as before
$\lVert \mathcal{G}(\theta )\rVert _{L^{\infty}}\lesssim \lVert
\theta \rVert _{\bar\alpha}\leq r$.
\end{proof}

We prove now the existence of an initialiser
$\theta _{\operatorname{init}}$ satisfying
$\lVert \theta _{\operatorname{init}}-\theta _{*,p}\rVert \leq \eta /8$
with high $\mathbb{P}^{n}_{\theta _{0}}$-probability, which can be computed
by gradient descent in polynomial time (with gradients given by affine
linear transformations). The proof is given in Section~\ref{sec:proofSpecificModels}. It relies on results from M-estimation
\cite{vandegeer2007} and is similar to \cite[Theorem B.6]{nickl2020a},
but technically more involved since the errors $\varepsilon _{i}$ in
\eqref{eq:nonGaussRegression} are generally not sub-Gaussian.

\begin{lem}[initialiser for GLMs]%
\label{lem:GLM_initialiser}
Grant the Assumptions of Lemma~\ref{lem:GLM_conditions}. Then there exist
$\iota >0$ and
$\theta _{\operatorname{init}}=\theta _{\operatorname{init}}(Z^{(n)})$
which is the output of $O((1+\log n)n^{2\alpha /(2\alpha +1)})$ many iterations
of gradient descent such that
\begin{align*}
\mathbb{P}^{n}_{\theta _{0}} \bigl(\lVert \theta _{
\operatorname{init}}-
\theta _{*,p}\rVert > \iota p^{-\bar\alpha} \bigr)\leq
c_{2}e^{-c_{1}n\delta _{n}^{2}}.
\end{align*}
\end{lem}

To formulate explicit sampling guarantees we reduce the key algorithmic
quantities to their dependence on $n$, $p$. By Lemma~\ref{lem:GLM_conditions} and Theorem~\ref{thm:surrogate} with the rescaled
Gaussian prior $\Pi $ from \eqref{eq:SievePrior} we establish for some
$c_{1},c_{2}>0$ the following curvature and Lipschitz constants
\begin{align}
m_{\pi}=p,\qquad \Lambda _{\pi}=p^{2\alpha +1}, \qquad m
\geq c_{1}n, \qquad \Lambda \leq c_{2}np^{\bar\alpha}.
\label{eq:GLM_LipCurv}
\end{align}
For simplicity, consider a sampling accuracy $0<\varepsilon \leq 1$ and choose
specifically the step size
\begin{align}
\gamma _{\varepsilon} = \min \bigl(\varepsilon ^{2}/p^{2\bar\alpha +1},
\varepsilon /\bigl(np^{4\bar\alpha +1}\bigr)\bigr). \label{eq:gamma_GLM} 
\end{align}
With this we obtain polynomial time sampling guarantees for the general
GLMs described in Example~\ref{exa:negative}.

\begin{thm}%
\label{thm:GLM_guarantees}
Let the posterior distribution $\Pi (\cdot |Z^{(n)})$ arise from a GLM
satisfying the assumptions in Lemma~\ref{lem:GLM_conditions} and with the
rescaled Gaussian prior $\Pi $ from \eqref{eq:SievePrior}. Suppose
$p\leq Cn^{1/(2\alpha +1+\bar\alpha )}$,
$\lVert \theta _{0}-\theta _{*,p}\rVert \leq C\delta _{n}$ and
$\lVert \theta _{\operatorname{init}}-\theta _{*,p}\rVert \leq \iota p^{-
\bar\alpha}$. Let $0<\varepsilon \leq 1$ and let
$(\vartheta _{k})_{k\geq 1}$ be the Markov chain with iterates in
\eqref{eq:langevin} with step size $\gamma _{\varepsilon}$ from
\eqref{eq:gamma_GLM}. Suppose for constants $c_{3},c_{4}>0$ that
$J_{\operatorname{in}}\geq c_{3} (1+\log \varepsilon ^{-1})/(n\gamma _{
\varepsilon})$,
$J\geq c_{4}\delta _{n}^{2}/(\varepsilon ^{2}\gamma _{\varepsilon})$,
$128(J+J_{\operatorname{in}})p^{2}\leq \exp (n\delta _{n}^{2}/1025)$.

Then there exist constants $c_{5}$, $c_{6}$, $c_{7}$ such that the following
holds on an event with
$\mathbb{P}_{\theta _{0}}^{n}\times \mathbf{P}$-probability at least
$1-c_{6} \exp   (-c_{5}n\delta _{n}^{2}  )$:
\begin{enumerate}[label=(\roman*)]
\item[{(i)}]
$W_{2}^{2}  (\mathcal{L}(\vartheta _{k}),\Pi (\cdot |Z^{(n)})
  )\leq 6e^{-n\delta _{n}^{2}}+c_{7}\varepsilon ^{2}$, $\quad J_{
\operatorname{in}}\leq k\leq J+J_{\operatorname{in}}$.
\item[{(ii)}] For all Lipschitz functions
$h:\mathbb{R}^{p}\rightarrow \mathbb{R}$ with
$\lVert h\rVert _{\mathrm{Lip}}=1$
\begin{align*}
& \Biggl\llvert \frac{1}{J}\sum_{k=1+J_{\operatorname{in}}}^{J+J_{
\operatorname{in}}}h(
\vartheta _{k})-
\int _{\Theta}h(\theta )\pi \bigl( \theta |Z^{(n)}\bigr)\,d
\theta \Biggr\rrvert \leq \varepsilon .
\end{align*}
\end{enumerate}
\end{thm}

\begin{proof}
As in the proof of Corollary~\ref{cor:surrogateTruePosterior} the upper
bound on $k\leq J+J_{\operatorname{in}}$ means
$\tilde{\vartheta}_{k}=\vartheta _{k}$. By adjusting the constant
$c_{6}$ in the statement it is enough to prove the claim for $n$ large
enough. For $\iota $ and $n$ large enough, we verify from
\eqref{eq:GLM_LipCurv} with $\eta =\iota p^{-\bar\alpha}$ and the upper
bound on $p$ the conditions $2c_{0}\Lambda _{\pi}\leq \eta m/32$,
$1\leq n\delta _{n}^{2}\leq \eta ^{2}m/p$ and $A\lesssim 1$. With step
size $\gamma _{\varepsilon}$ and $k\geq J_{\operatorname{in}}$ we have
$B(\gamma _{\varepsilon})\lesssim \varepsilon ^{2}$, so (i) follows from
Theorem~\ref{thm:MCMC_surrogate_Wasserstein}, and (ii) from Corollary~\ref{cor:surrogateTruePosterior}.
\end{proof}

Combined with Corollary~\ref{cor:PosteriorMean} for $\beta =1$ we can compute
the posterior mean from the Markov chain
$(\vartheta _{k})_{k\geq 1}$ to approximate the target $\theta _{0}$ up
to an error of order $O(\varepsilon )=O(\delta _{n})$. Due to the
$n$-dependent upper bound on $p$, $\delta _{n}$ is generally not the optimal
rate of convergence (cf. Remark~\ref{rem:bias}), but it is the optimal
rate when using instead Theorem~\ref{thm:MCMC_surrogate_functionals} and
the Markov chain $(\tilde{\vartheta}_{k})_{k\geq 1}$ with
$p=Cn^{1/(2\alpha +1)}$.

While Theorem~\ref{thm:GLM_guarantees} applies to GLMs with general (smooth)
link functions, let us compare to the canonical link as in Example~\ref{exa:GLM}, so the log-likelihood function is concave.
\begin{rem}[log-concave likelihood functions]
\label{rem3.8}
In Example~\ref{exa:GLM} with bounded $A''$ we directly verify (e.g., using
the Lipschitz bounds from Step~3 in the proof of Lemma~\ref{lem:GLM_initialiser}) that $f=-\log \pi (\cdot |Z^{(n)})$ satisfies
\eqref{eq:stronglyConvex} with $m_{f}=p$,
$\Lambda _{f}\lesssim np^{2\bar\alpha}$. Take now as step size
\begin{align*}
\tilde{\gamma}_{\varepsilon} = \min \bigl(\varepsilon ^{2}p/
\bigl(np^{4
\bar\alpha}\bigr),\varepsilon p^{2}/\bigl(n^{4}p^{8\bar\alpha}
\bigr)\bigr)\ll \gamma _{
\varepsilon}.
\end{align*}
Applying the results of \citep{durmus2019} as in the proof of Theorem~\ref{thm:MCMC_surrogate_Wasserstein} we get
$W_{2}^{2}  (\mathcal{L}(\vartheta _{k}),\Pi (\cdot |Z^{(n)})
  )\lesssim \varepsilon ^{2}$ for
$k\gtrsim (1+\log \varepsilon ^{-1})/(p\tilde{\gamma}_{
\varepsilon})\gg J_{\operatorname{in}}$. Comparing to Theorem~\ref{thm:GLM_guarantees}, we obtain for Gaussian and logistic regression
up to constants much smaller mixing times by taking into consideration
the local curvature of the log-likelihood function.
\end{rem}

\subsection{Density estimation}
\label{sec3.3}

Suppose that we observe an i.i.d. sample $Z^{(n)}=(X_{i})_{i=1}^{n}$ from
a density $p_{\theta}$ relative to $\nu _{\mathcal{X}}$. We assume the
following standard parametrisation (see e.g., \citep{vandervaart2008})
\begin{equation}
p_{\theta}(x)= \frac{e^{\Phi (\theta )(x)}}{\int _{\mathcal{X}}
e^{\Phi (\theta )(x)}\,d\nu _{\mathcal{X}}(x)}=e^{
\Phi (\theta )(x)-A(\Phi (\theta ))},\quad \theta \in
\Theta ,x\in \mathcal{X}, \label{eq:density_p} 
\end{equation}
where
$A(\Phi (\theta ))=\log \int _{\mathcal{X}}e^{\Phi (\theta )(x)}\,d\nu _{
\mathcal{X}}(x)$. Similar to GLMs with canonical link function in Examples
\ref{exa:GLM}, \ref{exa:negative} the log-likelihood function
\begin{equation*}
\ell _{n}(\theta )=\sum_{i=1}^{n}
\bigl(\Phi (\theta ) (X_{i})-A\bigl( \Phi (\theta )\bigr) \bigr)
\end{equation*}
is strongly concave only on bounded subsets in $\Theta $. As above, we
take $\eta =\iota p^{-\bar\alpha}$ for some $\iota $, $\alpha $. A~suitable
initialiser can be obtained from a similar argument as in Lemma~\ref{lem:GLM_initialiser}.

\begin{thm}
\label{thm:density}
Let the posterior distribution $\Pi (\cdot |Z^{(n)})$ arise according to
\eqref{eq:density_p} with the rescaled Gaussian prior $\Pi $
from~\eqref{eq:SievePrior}. Suppose
$\lVert \theta _{0}\rVert _{\alpha}\leq c_{0}$ for $\alpha >1$,
$c_{0}>0$. Let $p\leq Cn^{1/(2\alpha +1+\bar\alpha )}$,
$\lVert \theta _{0}-\theta _{*,p}\rVert \leq C\delta _{n}$ and
$\lVert \theta _{\operatorname{init}}-\theta _{*,p}\rVert \leq \iota p^{-
\bar\alpha}$ for $\iota $, $\bar\alpha $ with $\alpha >2\bar\alpha >1$.

Then the conclusions of Theorem~\ref{thm:GLM_guarantees} hold true for
$0<\varepsilon \leq 1$ and the Markov chain with iterates
$(\vartheta _{k})_{k\geq 1}$ in \eqref{eq:langevin} and step size
$\gamma _{\varepsilon}$ from \eqref{eq:gamma_GLM}, assuming that there
are constants $c_{3},c_{4}>0$ such that
$J_{\operatorname{in}}\geq c_{3} (1+\log \varepsilon ^{-1})/(n\gamma _{
\varepsilon})$,
$J\geq c_{4}\delta _{n}^{2}/(\varepsilon ^{2}n\gamma _{\varepsilon})$,
$128(J+J_{\operatorname{in}})p^{2}\leq \exp (n\delta _{n}^{2}/1025)$.
\end{thm}

The proof is postponed to Section~\ref{sec:proofSpecificModels}.

\subsection{Nonlinear Bayesian inverse problems}
\label{subsec:Darcy's-problem-1}

Nonlinear Bayesian inverse problems refer to posterior inference on
$\theta $ for data arising from regression problems
\eqref{eq:nonGaussRegression} with (generally) non-linear forward operators
$\mathcal{G}$. For an overview on such problems see
\cite{stuart2010,nickl2022}. Posterior sampling guarantees were obtained
for specific PDEs and Gaussian measurement errors $\varepsilon _{i}$ by
\cite{nickl2022a,bohr2021a}. Our results in Section~\ref{sec:mainResults} extend those to more general exponential family error
distributions, e.g., Poisson.

We consider now another nonlinear inverse problem in detail. To describe
it, in this section we write $\nabla u$ for the gradient with respect to
the space variable and
$\nabla \cdot u=\sum_{i=1}^{d}\partial _{i}u$ for the divergence
operator. For a \emph{source} $g_{1}\in C^{\infty}(\mathcal{X})$,
\emph{boundary} \emph{values}
$g_{2}\in C^{\infty}(\partial \mathcal{X})$ and \emph{conductivity}
$f\in C^{\gamma}(\mathcal{X})$, $\gamma \in \mathbb{N}$, let
$u\equiv u_{f}$ be the solution to the boundary value problem
\begin{equation}
\begin{cases} \mathcal{L}_{f}u=g_{1} &
\operatorname{in } \operatorname{\mathcal{X} ,}
\\
u=g_{2} & \operatorname{on } \operatorname{\partial \mathcal{X}},
\end{cases} %
\label{eq:PDE} 
\end{equation}
with the divergence form differential operator
$\mathcal{L}_{f}u=\nabla \cdot (f\nabla u)$. For strictly positive
$f$ the operator $\mathcal{L}_{f}$ is uniformly elliptic and classical
solutions $u_{f}\in C^{2}(\mathcal{X})$ exist by standard elliptic PDE
theory (e.g., Theorem~6.14 in \citep{gilbarg2001}). Details on the analytical
properties of the PDE \eqref{eq:PDE} relevant to our analysis are collected
in Section~\ref{sec:PDE_facts}.

The function $u_{f}$ typically represents the density of some quantity
and the PDE describes diffusion within $\mathcal{X}$ at equilibrium
\citep{evans2010}. For a fixed $f_{\operatorname{min}}>0$ let
\begin{equation}
f_{\theta}=f_{\operatorname{min}}+\exp \bigl(\Phi (\theta ) \bigr), \quad
\theta \in \ell ^{2}(\mathbb{N}), \label{eq:f_theta} 
\end{equation}
and consider the measurement model in Section~\ref{subsec:Nonparametric-regression} with
$\mathcal{G}(\theta )=u_{f_{\theta}}$ for known $g_{1}$ and $g_{2}$. Determining
the unknown conductivity $f_{\theta}$ from noisy observations of
$u_{f_{\theta}}$ is a popular example in the inverse problem literature,
called \emph{Darcy's problem},\emph{ }see
\citep{dashti2017,bonito2017,nickl2020} and the references therein. Since
the map $\theta \mapsto \mathcal{G}(\theta )$ is non-linear (see
\eqref{eq:G} and \cite[(5.14)]{nickl2020}), this is a non-linear inverse
problem, and the log-likelihood function in \eqref{eq:GLM_logLik} is not
concave. For Gaussian measurement errors, posterior contraction in this
model for different Gaussian process priors is studied by
\citep{giordano2020b}.

We derive now polynomial time posterior sampling guarantees using Proposition~\ref{prop:ExpoSurrogate}. The verification of the Assumption~\ref{assu:GLM_local_reg} requires techniques from the theory of elliptic
operators. These draw on well known results in the literature, but the
stability and Lipschitz bounds in Sections~\ref{sec:stability} and
\ref{sec:analytical} are new to the best of our knowledge. In order to
use elliptic $L^{2}$-PDE-regularity theory and to simplify the proofs using
Sobolev embeddings, we restrict in the following to $d\leq 3$. The existence
of a suitable initialiser is postulated here. A~proof is beyond the scope
of this paper.

\begin{thm}
\label{Thm:Darcy}%
Let the posterior distribution $\Pi (\cdot |Z^{(n)})$ arise according to
a nonparametric regression model with log-likelihood function
\eqref{eq:GLM_logLik} with $g\in C^{3}(\mathcal{I})$, forward operator
$\mathcal{G}(\theta )=u_{f_{\theta}}$ and with the rescaled Gaussian prior
$\Pi $ from~\eqref{eq:SievePrior}. Suppose $d\leq 3$,
$c_{X}^{-1}\leq p_{X}(x)\leq c_{X}$ for all $x\in \mathcal{X}$ for some
$c_{X}>0$, as well as $\lVert \theta _{0}\rVert _{\alpha}\leq c_{0}$ for
$\alpha \geq 21/d$, $c_{0}>0$. Let $p\leq Cn^{1/(2\alpha +1+14/d)}$,
$\lVert \theta _{0}-\theta _{*,p}\rVert \leq C\delta _{n}$ and
$\lVert \theta _{\operatorname{init}}-\theta _{*,p}\rVert \leq p^{-8/d}$.
Moreover, assume that the solutions $u_{f_{\theta}}$ satisfy for all
$c>0$, some $\mu ,c'>0$, possibly depending on $c$ and
$\alpha '>1/d+1/2$,
\begin{equation}
\inf_{x\in \mathcal{X},\lVert \Phi (\theta )\rVert _{C^{1}}\leq c} \biggl(\frac{1}{2}\Delta
u_{f_{\theta}}(x)+\mu \bigl\lVert \nabla u_{f_{
\theta }}(x)\bigr\rVert
^{2} \biggr)\geq c',\quad \theta \in h^{\alpha '}(
\mathcal{X}). \label{eq:Darcy_stabilityCondition} 
\end{equation}
Then the conclusions of Theorem~\ref{thm:GLM_guarantees} hold true for
$0<\varepsilon \leq 1$ and the Markov chain with iterates
$(\vartheta _{k})_{k\geq 1}$ in \eqref{eq:langevin} and step size
\begin{align}
\gamma _{\varepsilon} = \min \bigl(\varepsilon ^{2}/p^{1+16/d},
\varepsilon /\bigl(np^{1+26/d}\bigr)\bigr), \label{eq3.13}
\end{align}
assuming that there are constants $c_{3},c_{4}>0$ such that
$J_{\operatorname{in}}\geq c_{3} (1+\log \varepsilon ^{-1})/(np^{-6/d}
\gamma _{\varepsilon})$,
$J\geq c_{4}\delta _{n}^{2}/(\varepsilon ^{2}p^{-8/d}\gamma _{
\varepsilon})$,
$128(J+J_{\operatorname{in}})p^{2}\leq \exp (n\delta _{n}^{2}/1025)$.
\end{thm}

The proof is postponed to Section~\ref{subsec:Darcy's-problem}. Condition
\eqref{eq:Darcy_stabilityCondition} ensures injectivity of the forward
operator, which is necessary to show the stability lower bound in Assumption~\ref{assu:GLM_local_reg}(iv). This condition holds for a large class of
models $f$, $g_{1}$, $g_{2}$ (see the discussion after
\citep[Proposition~2.1.6]{nickl2022}), for instance, as soon as
$g_{1}>0$ on $\mathcal{X}$.

\section{Remaining proofs}
\label{sec:Proofs}

We denote for a metric space $T$ and a metric $d$ by
$N(T,d,\varepsilon )$ the minimal number of closed $d$-balls of radius
$\varepsilon $ necessary to cover $T$ and by
$H(T,d,\varepsilon )=\log N(T,d,\varepsilon )$ the metric entropy.

\subsection{Proof of Theorem~\ref{thm:localisation}}
\label{subsec:Sufficient-moment-conditions}

Observe first the following two lemmas.

\begin{lem}[Bernstein's inequality]
\label{lem:bernstein}%
Let $X_{1},\dots ,X_{n}$ be real-valued centred and independent random
variables such that
$\mathbb{E}_{\theta _{0}}|X_{i}|^{q}\leq (q!/2)\sigma ^{2}c^{q-2}$ for
some $\sigma >0$, $c>0$ and all $1\leq i\leq n$, $q\geq 2$. Then
\begin{equation*}
\mathbb{P}_{\theta _{0}}^{n} \Biggl( \Biggl\llvert \sum
_{i=1}^{n}X_{i} \Biggr\rrvert \geq
\sqrt{2n\sigma ^{2}t}+ct \Biggr)\leq 2e^{-t},\quad t\geq 0.
\end{equation*}
\end{lem}

\begin{lem}
\label{lem:concentrationEmpiricalProc}
Let $\mathcal{U}$ be a measurable subset of $\mathbb{R}^{p}$ with diameter
$\sup_{\theta ,\theta '\in \mathcal{U}}\lVert \theta -\theta '
\rVert =D>0$. Let $h_{\theta}:\mathcal{Z}\rightarrow \mathbb{R}$,
$\theta \in \mathcal{U}$, be a family of functions such that for some
$\sigma _{p},c_{p}>0$, all $q\geq 2$ and all $i=1,\dots ,n$
\begin{align}
\mathbb{E}_{\theta _{0}} \bigl\llvert h_{\theta}(Z_{i})
\bigr\rrvert ^{q} & \leq (q!/2)\sigma _{p}^{2}c_{p}^{q-2},
\quad \theta \in \mathcal{U}, \label{eq:h_moment_1} 
\\
\mathbb{E}_{\theta _{0}} \bigl\llvert h_{\theta}(Z_{i})-h_{\theta '}(Z_{i})
\bigr\rrvert ^{q} & \leq (q!/2)c_{p}^{q}\bigl
\lVert \theta -\theta '\bigr\rVert ^{q},\quad \theta ,
\theta '\in \mathcal{U}. \label{eq:h_moment_2} 
\end{align}
Consider the empirical process
$(\mathcal{Z}_{n}(\theta ),\theta \in \mathcal{U})$ with
$\mathcal{Z}_{n}(\theta )=\sum_{i=1}^{n}(h_{\theta}(Z_{i})-
\mathbb{E}_{\theta _{0}}h_{\theta}(Z_{i}))$. Then there exists a universal
constant $M\geq 1$ such that for all $t\geq 1$, $t'\geq 0$
\begin{equation*}
\mathbb{P}_{\theta _{0}}^{n} \Bigl(\sup_{\theta \in \mathcal{U}}
\bigl\llvert \mathcal{Z}_{n}(\theta ) \bigr\rrvert \geq
Mc_{p}D (\sqrt{np}+p+\sqrt{nt}+t )+3 \Bigl(\sqrt{2n\sigma
_{p}^{2}t'}+c_{p}t'
\Bigr) \Bigr) \le e^{-t}+2e^{-t'}.
\end{equation*}
\end{lem}

\begin{proof}
Write $\mathcal{Z}_{n}(\theta )=\sum_{i=1}^{n}h_{\theta ,i}$ with independent
and centred random variables
$h_{\theta ,i}=h_{\theta}(Z_{i})-\mathbb{E}_{\theta _{0}}h_{\theta}(Z_{i})$.
The moment assumptions in (\ref{eq:h_moment_1}) and (\ref{eq:h_moment_2})
hold for the $h_{\theta ,i}$ with constants $3\sigma _{p}$ and
$3c_{p}$. Fix any $\theta ,\theta '\in \mathcal{U}$. Then the Bernstein
inequality in Lemma~\ref{lem:bernstein} gives for $t\geq 0$
\begin{align}
\mathbb{P}_{\theta _{0}} \Bigl( \bigl\llvert \mathcal{Z}_{n}\bigl(
\theta '\bigr) \bigr\rrvert \geq 3\sqrt{2n \sigma
_{p}^{2}t}+3c_{p}t \Bigr) & \leq
2e^{-t}, \label{eq:bernstein_h_1} 
\\
\mathbb{P}_{\theta _{0}} \bigl( \bigl\llvert \mathcal{Z}_{n}(\theta
)- \mathcal{Z}_{n}\bigl(\theta '\bigr) \bigr\rrvert \geq
3c_{p}\bigl\lVert \theta -\theta ' \bigr\rVert
\sqrt{2nt}+3c_{p}\bigl\lVert \theta -\theta '\bigr\rVert t
\bigr) & \leq 2e^{-t}. \label{eq:bernstein_h_2} 
\end{align}
The last line implies that $\mathcal{Z}_{n}$ has a mixed tail with respect
to the metrics
$d_{1}(\theta ,\theta ')=3c_{p}\lVert \theta -\theta '\rVert $,
$d_{2}(\theta ,\theta ')=\sqrt{2n}d_{1}(\theta ,\theta ')$ in the sense
of \citep[Equation (3.8)]{dirksen2015a}. Since the set $\mathcal{U}$ has
diameter
$\sup_{\theta ,\theta '\in \mathcal{U}}d_{1}(\theta ,\theta ')=3c_{p}D$
with respect to $d_{1}$ and diameter $3c_{p}D\sqrt{2n}$ with respect to
$d_{2}$, using Proposition~4.3.34 and equation~(4.171) in
\citep{gine2016} yields for the metric entropy integrals with respect to
$d_{1}$ and $d_{2}$ the upper bounds
\begin{align*}
\gamma _{d_{1}}(\mathcal{U}) & =
\int _{0}^{\infty}H(\mathcal{U},d_{1},
\varepsilon )\,d\varepsilon \leq
\int _{0}^{3c_{p}D}\log N\bigl(\bigl\{\theta \in
\mathbb{R}^{p}:\lVert \theta \rVert \leq D\bigr\},\lVert \cdot \rVert ,
\varepsilon /(3c_{p})\bigr)\,d\varepsilon
\\*
& =
\int _{0}^{3c_{p}D}H\bigl(\bigl\{\theta \in
\mathbb{R}^{p}:\lVert \theta \rVert \leq 1\bigr\},\lVert \cdot \rVert ,
\varepsilon /(3c_{p}D)\bigr)\,d \varepsilon
\\
& \leq
\int _{0}^{3c_{p}D}p\log (9c_{p}D/\varepsilon
)\,d\varepsilon =3c_{p}Dp
\int _{0}^{1}\log (3/\varepsilon )\,d\varepsilon ,
\end{align*}
and in the same way
\begin{equation*}
\gamma _{d_{2}}(\mathcal{U})=
\int _{0}^{\infty}\sqrt{H(\mathcal{U},d_{2},
\varepsilon )}\,d\varepsilon \leq 3c_{p}D\sqrt{np}
\int _{0}^{1}\sqrt{ \log (3/\varepsilon )}\,d
\varepsilon .
\end{equation*}
Observe for $z\geq 2$ the inequalities
\begin{align*}
\int _{0}^{1} \log (z/x)\,dx &= \log z+1,
\\
\int _{0}^{1} \sqrt{\log (z/x)}\,dx &\leq
\frac{2\log z}{2\log z-1} \sqrt{\log z}
\end{align*}
(as stated after equation (76) in \cite{nickl2020a}). So,
$\gamma _{d_{1}}(\mathcal{U})\leq 9c_{p}Dp$,
$\gamma _{d_{2}}(\mathcal{U})\leq 6c_{p}D\sqrt{np}$. Then, together with
the mixed tail property in (\ref{eq:bernstein_h_2}), infer from Theorem~3.5 of \citep{dirksen2015a} the existence of an absolute constant
$M\geq 1$ such that for any $t\geq 1$
\begin{equation*}
\mathbb{P}_{\theta _{0}}^{n} \Bigl(\sup_{\theta \in \mathcal{U}}
\bigl\llvert \mathcal{Z}_{n}(\theta )-\mathcal{Z}_{n}\bigl(
\theta '\bigr) \bigr\rrvert \geq Mc_{p}D(\sqrt{np}+p+
\sqrt{nt}+t) \Bigr)\leq e^{-t}.
\end{equation*}
The result follows from the triangle inequality and from applying (\ref{eq:bernstein_h_1})
to $t=t'$.
\end{proof}

With this let us prove the main result of this section.
\begin{proof}[Proof of Theorem~\ref{thm:localisation}]
Define
$b(\theta )=\nabla \ell _{n}(\theta )-\mathbb{E}_{\theta _{0}}^{n}
\nabla \ell _{n}(\theta )$,
$\Sigma (\theta )=\nabla ^{2}\ell _{n}(\theta )-\mathbb{E}_{\theta _{0}}^{n}
\nabla ^{2}\ell _{n}(\theta )$. For $C_{3}=2\sqrt{8C}C_{1}$ set
$\tau _{1}=C_{3}n\delta _{n}p^{\kappa _{1}}$,
$\tau _{2}=(C_{2}/2)np^{-\kappa _{3}}$. We prove the claim for the event
$\mathcal{E}=\mathcal{E}_{1}\cap \mathcal{E}_{2}$, where
\begin{align*}
\mathcal{E}_{1} & = \bigl\{ \bigl\lVert b(\theta _{*,p})
\bigr\rVert \leq \tau _{1} \bigr\} ,\qquad \mathcal{E}_{2}=
\Bigl\{ \sup_{\theta \in
\mathcal{B}}\bigl\lVert \Sigma (\theta )\bigr\rVert
_{\operatorname{op}}\leq \tau _{2} \Bigr\} .
\end{align*}
Recall the min-max characterisation of the eigenvalues of a symmetric matrix
$A\in \mathbb{R}^{p\times p}$ such that
\begin{equation*}
\bigl\lVert \nabla ^{2}A\bigr\rVert _{\operatorname{op}}=\sup
_{v\in \mathbb{R}^{p}:
\lVert v\rVert \leq 1}v^{\top}\nabla ^{2}Av.
\end{equation*}
The mean boundedness assumptions in (ii) (with $q=2$) then give for
$\theta \in \mathcal{B}$ that
$\lVert \mathbb{E}_{\theta _{0}}^{n}\nabla ^{2}\ell _{n}(\theta )
\rVert _{\operatorname{op}}\leq C_{1}np^{\kappa _{2}}$. So, on
$\mathcal{E}$
\begin{align*}
& \bigl\lVert \nabla ^{2}\ell _{n}(\theta )\bigr\rVert
_{\operatorname{op}} \leq \tau _{2}+C_{1}np^{\kappa _{2}}\leq
(C_{2}/2+C_{1})np^{\kappa _{2}},
\\
& \bigl\lVert \nabla \ell _{n}(\theta _{*,p})\bigr\rVert
\leq (C_{3}+C_{1})n \delta _{n}p^{\kappa _{1}}.
\end{align*}
This verifies Assumption~\ref{assu:a1}(ii) with
$c_{\operatorname{max}}=\max (2C_{1}\sqrt{8C},C_{2}/2)+C_{1}$, while the
mean curvature lower bound in (iii) yields by Weyl's inequality for
$\theta \in \mathcal{B}$
\begin{align*}
\lambda _{\min} \bigl(-\nabla ^{2}\ell _{n}(
\theta ) \bigr)&= \lambda _{\min} \bigl(\mathbb{E}_{\theta _{0}}^{n}
\bigl[-\nabla ^{2} \ell _{n}(\theta ) \bigr]-\Sigma (\theta
) \bigr)
\\
& \geq \sum_{i=1}^{n}\lambda
_{\min} \bigl(\mathbb{E}_{\theta _{0}}^{n} \bigl[-\nabla
^{2}\ell (\theta ,Z_{i}) \bigr] \bigr)-\bigl\lVert \Sigma (
\theta )\bigr\rVert _{\operatorname{op}}\geq (C_{2}/2)np^{-\kappa _{3}}.
\end{align*}
From this obtain Assumption~\ref{assu:a1}(iii) for
$c_{\operatorname{min}}=C_{2}/2$.

We are therefore left with showing
$\mathbb{P}_{\theta _{0}}^{n}(\mathcal{E}^{c})\leq C'e^{-Cn\delta _{n}^{2}}$.
By adjusting $C'$ it suffices to prove this for $n$ large enough. We will
use a contraction argument for quadratic forms, commonly used in random
matrix theory. For $0<\delta \leq 1$ and
$N=N(\{v\in \mathbb{R}^{p}:\lVert v\rVert \leq 1\},\lVert \cdot
\rVert ,\delta )$ let $v_{1},\dots ,v_{N}$ be the centres of a minimal
open cover for the Euclidean unit ball with radius $\delta $. This implies
for $v\in \mathbb{R}^{p}$ with $\lVert v\rVert \leq 1$ and
$i=1,\dots ,N$ with $\lVert v-v_{i}\rVert \leq \delta $ that
\begin{align*}
v^{\top}\Sigma (\theta )v & =v_{i}^{\top}\Sigma (
\theta )v_{i}+(v-v_{i})^{
\top}\Sigma (\theta )
(v-v_{i})+2(v-v_{i})^{\top}\Sigma (\theta
)v_{i}
\\
& \leq v_{i}^{\top}\Sigma (\theta )v_{i}+\delta
^{2}\bigl\lVert \Sigma ( \theta )\bigr\rVert _{\operatorname{op}}+2\delta
\bigl\lVert \Sigma (\theta ) \bigr\rVert _{\operatorname{op}}\leq v_{i}^{\top}
\Sigma (\theta )v_{i}+3 \delta \bigl\lVert \Sigma (\theta )\bigr\rVert
_{\operatorname{op}}.
\end{align*}
For the same $v_{i}$ we also get
$|v^{\top}b(\theta _{*,p})|\leq |v_{i}^{\top}b(\theta _{*,p})|+
\delta \lVert b(\theta _{*,p})\rVert $. Taking $\delta =1/4$ and maximising
over $v$ in the unit ball and over $i$ then gives
\begin{align*}
\bigl\lVert b(\theta _{*,p})\bigr\rVert & \leq \frac{4}{3}\max
_{i=1,\dots ,N} \bigl\llvert v_{i}^{
\top}b(\theta
_{*,p}) \bigr\rrvert ,
\\
\bigl\lVert \Sigma (\theta )\bigr\rVert _{\operatorname{op}} & \leq 4\max
_{i=1,
\dots ,N}\sup_{\theta \in \mathcal{B}} \bigl\llvert
v_{i}^{\top}\Sigma (\theta )v_{i} \bigr\rrvert .
\end{align*}
By applying union bounds this means for $j=1,2$
\begin{align*}
\mathbb{P}_{\theta _{0}}^{n}\bigl(\mathcal{E}^{c}
\bigr) &\leq \mathbb{P}_{
\theta _{0}}^{n}\bigl(\mathcal{E}_{1}^{c}
\bigr)+\mathbb{P}_{\theta _{0}}^{n}\bigl( \mathcal{E}_{2}^{c}
\bigr)
\\
& \leq N\sup_{v\in \mathbb{R}^{p}:\lVert v\rVert \leq 1} \Bigl( \mathbb{P}_{\theta _{0}}^{n}
\bigl( \bigl\llvert v^{\top}b(\theta _{*,p}) \bigr\rrvert >3
\tau _{1}/4 \bigr)+\mathbb{P}_{\theta _{0}}^{n} \Bigl(\sup
_{\theta \in
\mathcal{B}} \bigl\llvert v^{\top}\Sigma (\theta )v \bigr
\rrvert >\tau _{2}/4 \Bigr) \Bigr).
\end{align*}
Proposition~4.3.34 of \citep{gine2016} and the growth conditions yield
$N\leq e^{p\log 12}\leq e^{3Cn\delta _{n}^{2}}$.

To prove the wanted high probability bounds we are left with establishing
that the probabilities in the last display are each smaller than
$C'e^{-4Cn\delta _{n}^{2}}$ for some $C'>0$. For this we apply the two
last lemmas to the empirical processes $b(\theta _{*,p})$ and
$\Sigma (\theta )$, uniformly for $\theta \in \mathcal{B}$. First, consider
$v^ {\top }b(\theta _{*,p})=\sum_{i=1}^{n}(h_{\theta}(Z_{i})-
\mathbb{E}_{\theta _{0}}h_{\theta}(Z_{i}))$ with
$h_{\theta}(Z_{i})=v^{\top}\nabla \ell _{n}(\theta _{*,p},Z_{i})$. The
mean boundedness conditions in (ii) show
$\mathbb{E}_{\theta _{0}}|h_{\theta}(Z_{i})|^{q}\leq (q!/2)\sigma ^{2}c^{q-2}$
for $\sigma =C_{1}p^{\kappa _{1}}$, $c=C_{1}p^{\kappa _{2}}$ and all
$q\geq 2$. We can therefore apply Lemma~\ref{lem:bernstein} to
$t=4Cn\delta _{n}^{2}$, so
\begin{align*}
\sqrt{2n\sigma ^{2}t}+ct\leq 3\tau _{1}/4 \leq
C_{1}\sqrt{8C}n\delta _{n}p^{
\kappa _{1}}+4C_{1}Cp^{\kappa _{2}}
\delta _{n}.
\end{align*}
By the growth conditions
$4C\delta _{n}p^{\kappa _{2}}\leq \sqrt{8C}$ for large enough $n$ and then
$\sqrt{2n\sigma ^{2}t}+ct\leq 3\tau _{1}/4$, implying
$\mathbb{P}_{\theta _{0}}^{n}(|v^{\top}b(\theta _{*,p})|>3\tau _{1}/4)
\leq 2e^{-4Cn\delta _{n}^{2}}$.

Next, consider
$v^{\top }\Sigma (\theta )v=\sum_{i=1}^{n}(h_{\theta}(Z_{i})-
\mathbb{E}_{\theta _{0}}h_{\theta}(Z_{i}))$ with
$h_{\theta}(Z_{i})=v^{\top}\nabla ^{2}\ell _{n}(\theta ,Z_{i})v$. Using
again the conditions in part (ii) verifies (\ref{eq:h_moment_1}) and (\ref{eq:h_moment_2})
with $\sigma _{p}=C_{1}p^{\kappa _{2}}$,
$c_{p}=C_{1}p^{\kappa _{4}}$. The set $\mathcal{U}=\mathcal{B}$ has diameter
$D=\sup_{\theta ,\theta '\in \mathcal{B}}\lVert \theta -\theta '
\rVert =2\eta $. If $M$ is the constant from the statement of Lemma~\ref{lem:concentrationEmpiricalProc} and $t=t'=4Cn\delta _{n}^{2}$, then
\begin{align*}
& Mc_{p}D (\sqrt{np}+p+\sqrt{nt}+t )+3 \Bigl(\sqrt{2n \sigma
_{p}^{2}t'}+c_{p}t'
\Bigr)
\\
&\quad  \leq MC_{1}p^{\kappa _{4}}2\eta \bigl(3\sqrt{C}n\delta
_{n}+5Cn \delta _{n}^{2} \bigr)+3 \bigl(2
\sqrt{2C}C_{1}n\delta _{n}p^{\kappa _{2}}+4CC_{1}n
\delta _{n}^{2}p^{\kappa _{4}} \bigr)
\\
&\quad \leq 5MC_{3} n\max \bigl(\delta _{n}p^{\kappa _{2}},\eta
\delta _{n}p^{
\kappa _{4}},\delta _{n}^{2}p^{\kappa _{4}}
\bigr).
\end{align*}
Taking $n$ large enough, by the growth conditions this is smaller than
$\tau _{2}/4$. Lemma~\ref{lem:concentrationEmpiricalProc} now implies the
wanted upper bound
$\mathbb{P}_{\theta _{0}}^{n}(\sup_{\theta \in \mathcal{B}}|v^{\top}
\Sigma (\theta )v|>\tau _{2}/4)\leq 3e^{-4Cn\delta _{n}^{2}}$. This finishes
the proof.
\end{proof}

\subsection{Proof of Theorem~\ref{thm:GaussianContraction}}
\label{sec:proofGaussContraction}

Consider the following three auxiliary results.

\begin{lem}
\label{lem:smallBall}%
In the setting of Theorem~\ref{thm:GaussianContraction} there exists for
any large enough $r\geq \max (4,2c_{0})$ a constant
$c\equiv c(C,\alpha ,c_{0},r)>0$ with
$\Pi (\mathcal{B}_{n,r})\geq e^{-cn\delta _{n}^{2}}$.
\end{lem}

\begin{proof}
We can write $\Pi \sim (n\delta _{n}^{2})^{-1/2}\bar{\Pi}$ for the unscaled
probability measure $\bar{\Pi}\sim N(0,\Lambda _{\alpha}^{-1})$. The reproducing
kernel Hilbert space of $\bar{\Pi}$ is equipped with the norm
$\lVert \cdot \rVert _{\alpha}$ on $\mathbb{R}^{p}$. Since
$\lVert \theta _{*,p}\rVert _{\alpha}\leq c_{0}$, we have for
$r\geq 2c_{0}$
\begin{equation*}
\bigl\{ \theta \in \mathbb{R}^{p}:\lVert \theta -\theta
_{*,p} \rVert \leq \delta _{n},\lVert \theta -\theta
_{*,p}\rVert _{\alpha} \leq r/2 \bigr\} \subset
\mathcal{B}_{n,r}.
\end{equation*}
The small-ball calculus from \citep[Corollary~2.6.18]{gine2016} allows
then for lower bounding the wanted probability as
\begin{align}
\begin{aligned}
\Pi (\mathcal{B}_{n,r}) &\geq e^{-n\delta _{n}^{2}\lVert \theta _{*,p}
\rVert _{\alpha}^{2}/2}\bar{\Pi} \bigl(\theta
\in \mathbb{R}^{p}: \lVert \theta \rVert \leq n^{1/2}\delta
_{n}^{2},\lVert \theta \rVert _{\alpha}\leq (r/2)
\bigl(n\delta _{n}^{2}\bigr)^{1/2} \bigr)
\\
& \geq e^{-n\delta _{n}^{2}c_{0}^{2}/2} \bigl(\bar{\Pi} \bigl(\theta \in \mathbb{R}^{p}:
\lVert \theta \rVert \leq n^{1/2}\delta _{n}^{2}
\bigr)-\bar{\Pi} \bigl(\theta \in \mathbb{R}^{p}:\lVert \theta \rVert
_{\alpha}>(r/2) \bigl(n\delta _{n}^{2}
\bigr)^{1/2} \bigr) \bigr).
\end{aligned}
\label{eq:smallBall_step1} 
\end{align}
Observe for any $\varepsilon >0$ the metric entropy bound
\begin{equation}
H \bigl(\bigl\{\theta \in \mathbb{R}^{p}:\lVert \theta \rVert
_{\alpha} \leq 1\bigr\},\lVert \cdot \rVert ,\varepsilon \bigr)\leq H
\bigl(B, \lVert \cdot \rVert _{L^{2}((0,1))},\varepsilon \bigr)\lesssim
\varepsilon ^{-1/\alpha}, \label{eq:thm_postContraction_entropy} 
\end{equation}
where $B$ is the unit ball in the $L^{2}((0,1))$-Sobolev space of fractional
order $\alpha\in \mathbb{R}$, concluding by
\citep[Theorem~4.10.3]{triebel1978}. It follows from
\citep[Theorem~1.2]{li1999} with $J\equiv 1$ and
$\varepsilon =n^{1/2}\delta _{n}^{2}$ for a universal constant
$c'>0$ that
\begin{equation*}
\bar{\Pi} \bigl(\theta \in \mathbb{R}^{p}:\lVert \theta \rVert \leq
n^{1/2} \delta _{n}^{2} \bigr)\geq
e^{-c'(n^{1/2}\delta _{n}^{2})^{-2/(2
\alpha -1)}}=e^{-c'n\delta _{n}^{2}}.
\end{equation*}
On the other hand, let $V$ be a $p$-dimensional standard Gaussian random
vector, defined on probability space with probability measure
$\mathbb{P}$ and expectation operator $\mathbb{E}$. Noting
$\lVert \Lambda _{\alpha }^{-1/2}V\rVert _{\alpha}=\lVert V\rVert $, we
have
\begin{equation*}
\bar{\Pi} \bigl(\theta \in \mathbb{R}^{p}:\lVert \theta \rVert
_{
\alpha}>(r/2) \bigl(n\delta _{n}^{2}
\bigr)^{1/2} \bigr)=\mathbb{P} \bigl( \lVert V\rVert >(r/2) \bigl(n\delta
_{n}^{2}\bigr)^{1/2} \bigr).
\end{equation*}
For $r\geq 4$ we get from $p\leq Cn\delta _{n}^{2}$ that
$\mathbb{E}\lVert V\rVert \leq p^{1/2}\leq (r/4)(Cn\delta _{n}^{2})^{1/2}$.
By a standard concentration inequality for Lipschitz-functionals of Gaussian
random vectors (Theorem~2.5.7 of \citep{gine2016} with
$F=\lVert \cdot \rVert $) this means that the last display is upper bounded
by
\begin{align}
& \mathbb{P} \bigl(\lVert V\rVert -\mathbb{E}\lVert V\rVert >(r/4) \bigl(n
\delta _{n}^{2}\bigr)^{1/2} \bigr)\leq
e^{-(r^{2}/16)Cn\delta _{n}^{2}}. \label{eq:gaussianConcentration} 
\end{align}
Conclude now with (\ref{eq:smallBall_step1}).
\end{proof}

\begin{lem}
\label{lem:normalisingFactors}%
Consider the setting of Theorem~\ref{thm:GaussianContraction} and recall
the event
\begin{equation*}
\bar{\mathcal{D}}_{n}(c) = \biggl\{
\int _{\lVert \theta -\theta _{*,p}
\rVert \leq \delta _{n}}e^{\ell _{n}(\theta )-\ell _{n}(\theta _{0})} \pi (\theta )\,d\theta \leq
e^{-cn\delta _{n}^{2}} \biggr\}
\end{equation*}
from \eqref{eq:barDnc}. There exists $c'>0$ such that
$\mathbb{P}_{\theta _{0}}^{n}(\bar{\mathcal{D}}_{n}(c'))\leq 2e^{-n
\delta _{n}^{2}}$.
\end{lem}

\begin{proof}
For large enough $r$ and $c>0$ Lemma~\ref{lem:smallBall} shows
$\Pi (\mathcal{B}_{n,r})\geq e^{-cn\delta _{n}^{2}}$. With $c_{r}$ the
constant from (\ref{eq:ApproxCondition}) choose $c'=c+7c_{r}$ and consider
the probability measure
$\nu _{n}=\Pi (\cdot \cap \mathcal{B}_{n,r})/\Pi (\mathcal{B}_{n,r})$,
supported on $\mathcal{B}_{n,r}$. Introducing the functions
$h(x)=\int _{\mathcal{B}_{n,r}}(\ell (\theta _{0},x)-\ell (\theta ,x))d
\nu _{n}(\theta )$, the Jensen inequality implies
\begin{align*}
\mathbb{P}_{\theta _{0}}^{n}\bigl(\mathcal{D}^{c}\bigr)
& =\mathbb{P}_{\theta _{0}}^{n} \biggl(\Pi (\mathcal{B}_{n,r})
\int _{\Theta}e^{\ell _{n}(\theta )-
\ell _{n}(\theta _{0})}\,d\nu _{n}(\theta )\leq
e^{-c'n\delta _{n}^{2}} \biggr)
\\
& \leq \mathbb{P}_{\theta _{0}}^{n} \Biggl(\sum
_{i=1}^{n}h(Z_{i}) \geq
7c_{r}n\delta _{n}^{2} \Biggr).
\end{align*}
We find from Fubini's theorem and (\ref{eq:ApproxCondition})
\begin{align*}
\bigl\llvert \mathbb{E}_{\theta _{0}}h(Z_{i}) \bigr\rrvert & =
\biggl\llvert
\int _{\mathcal{B}_{n,r}} \mathbb{E}_{\theta _{0}}\bigl(\ell (\theta
_{0})-\ell (\theta )\bigr)\,d\nu _{n}( \theta ) \biggr
\rrvert \leq c_{r}\delta _{n}^{2},
\end{align*}
while we get for $q\geq 2$ from (\ref{eq:ApproxCondition})
\begin{align*}
\mathbb{E}_{\theta _{0}} \bigl\llvert h(Z_{i})-
\mathbb{E}_{\theta _{0}}h(Z_{i}) \bigr\rrvert ^{q} &\leq
2^{q+1}\mathbb{E}_{\theta _{0}} \bigl\llvert h(Z_{i})
\bigr\rrvert ^{q}
\\
& \leq 2^{q+1}
\int _{\mathcal{B}_{n,r}}\mathbb{E}_{\theta _{0}} \bigl\llvert \ell (\theta
_{0})-\ell (\theta ) \bigr\rrvert ^{q}\,d\nu
_{n}(\theta ) \leq (q!/2)8(c_{r}\delta_{n})^{2}(2c_{r})^{q-2}.
\end{align*}
The claim follows from Lemma~\ref{lem:bernstein} applied to
$\sigma ^{2}=8c_{r}^{2}\delta _{n}^{2}$, $c=2c_{r}$ and
$t=n\delta _{n}^{2}$ such that
\begin{equation*}
\mathbb{P}_{\theta _{0}}^{n} \Biggl(\sum
_{i=1}^{n}h(Z_{i})\geq
7c_{r}n \delta _{n}^{2} \Biggr)\leq
\mathbb{P}_{\theta _{0}}^{n} \Biggl(\sum
_{i=1}^{n}\bigl(h(Z_{i})-
\mathbb{E}_{\theta _{0}}h(Z_{i})\bigr)\geq 6c_{r}n\delta
_{n}^{2} \Biggr) \leq 2e^{-n\delta _{n}^{2}}.
\end{equation*}\qedhere
\end{proof}

\begin{lem}
\label{lem:setA}%
Consider the setting of Theorem~\ref{thm:GaussianContraction} and let
$\mathcal{A}=  \{ \theta \in \mathbb{R}^{p}:\lVert \theta \rVert _{
\alpha}\leq L'  \} $ for $L'>0$. If $L'$ is large enough, then there
exists a constant $c=c(C,c_{L'})$ such that
\begin{align*}
H\bigl(\mathcal{A},h,L'\delta _{n}\bigr)\leq cn\delta
_{n}^{2},\qquad \Pi ( \mathcal{A}) & \geq
1-e^{-((L')^{2}/16)n\delta _{n}^{2}}.
\end{align*}
\end{lem}

\begin{proof}
Apply first the upper bound on the Hellinger distance in (\ref{eq:hellingerCondition})
and then (\ref{eq:thm_postContraction_entropy}) with
$\varepsilon =c_{L'}\delta _{n}=(L'c_{L'}/L')\delta _{n}$ to the extent
that, noting $\delta _{n}^{-1/\alpha}\leq n\delta _{n}^{2}$,
\begin{align*}
H\bigl(\mathcal{A},h,L'\delta _{n}\bigr) &\leq H\bigl(
\mathcal{A},\lVert \cdot \rVert ,L'c_{L'}\delta
_{n}\bigr)
\\
&  =H \bigl(\bigl\{\theta \in \mathbb{R}^{p}:\lVert \theta \rVert
_{
\alpha}\leq 1\bigr\},\lVert \cdot \rVert ,c_{L'}\delta
_{n} \bigr) \lesssim C(c_{L'})^{-1/\alpha}n\delta
_{n}^{2}.
\end{align*}
This proves the wanted metric entropy bound. Next, if
$\theta \in \mathbb{R}^{p}$ has norm
$\lVert \theta \rVert \leq (L'/2)\delta _{n}$, then
$\lVert \theta \rVert _{\alpha}\leq p^{\alpha}\lVert \theta \rVert
\leq L'/2$, and thus
\begin{equation*}
\bigl\{ \theta =\theta _{1}+\theta _{2}\in
\mathbb{R}^{p}:\lVert \theta _{1}\rVert \leq
\bigl(L'/2\bigr)\delta _{n},\lVert \theta _{2}
\rVert _{
\alpha}\leq L'/2 \bigr\} \subset \mathcal{A}.
\end{equation*}
Denoting by $\Phi $ the standard Gaussian distribution function and recalling
$\Pi \sim (n\delta _{n}^{2})^{-1/2}\bar{\Pi}$ from the proof of Lemma~\ref{lem:smallBall}, Borell's inequality
\citep[Theorem~11.17]{ghosal2017} provide the lower bound
\begin{align*}
\Pi (\mathcal{A}) & \geq \bar{\Pi} \bigl(\theta =\theta _{1}+\theta
_{2} \in \mathbb{R}^{p}:\lVert \theta _{1}\rVert
\leq \bigl(L'/2\bigr)n^{1/2}\delta _{n}^{2},
\lVert \theta _{2}\rVert _{\alpha}\leq \bigl(L'/2
\bigr)n^{1/2}\delta _{n} \bigr)
\\
& \geq \Phi \bigl(\Phi ^{-1} \bigl(\bar{\Pi} \bigl(\theta \in
\mathbb{R}^{p}:\lVert \theta \rVert \leq \bigl(L'/2
\bigr)n^{1/2}\delta _{n}^{2} \bigr) \bigr)+
\bigl(L'/2\bigr)n^{1/2}\delta _{n} \bigr).
\end{align*}
Applying now \citep[Theorem~1.2]{li1999} to $J\equiv 1$ and
$\varepsilon =(L'/2)n^{1/2}\delta _{n}^{2}$, and using the inequality
$y\geq -2\Phi ^{-1}(e^{-y^{2}/4})$ for $y=(L'/2)n^{1/2}\delta _{n}$ and
large enough $L'$, which holds for $y\geq 2\sqrt{2\pi}$ by standard computations
for $\Phi $, we find for $c'>0$ that
\begin{equation*}
\Pi (\mathcal{A})\geq \Phi (\Phi ^{-1} \bigl(e^{-c'n\delta _{n}^{2}}
\bigr)-2\Phi ^{-1}\bigl(e^{-((L')/16)n\delta _{n}^{2}} \bigr).
\end{equation*}
Possibly increasing $L'$ even further, we ensure that
$c'\leq (L')^{2}/16$. This implies at last
\begin{equation*}
\Pi (\mathcal{A})\geq \Phi \bigl(-\Phi ^{-1}\bigl(e^{-((L')^{2}/16)n\delta _{n}^{2}}
\bigr)\bigr) \geq 1-e^{-((L')^{2}/16)n\delta _{n}^{2}}.
\end{equation*}
From this obtain the claim.
\end{proof}

\begin{proof}[Proof of Theorem~\ref{thm:GaussianContraction}]
Let $\bar{\mathcal{D}}^{c}_{n}(c')$ be the high probability event from
Lemma~\ref{lem:normalisingFactors}. Since that lemma already shows the
claimed bound on the normalising factors in Assumption~\ref{assu:a3}, we
only have to prove the posterior contraction. Consider for $L,L'>0$ the
sets
\begin{align}
\mathcal{A} & = \bigl\{ \theta \in \mathbb{R}^{p}:\lVert \theta \rVert
_{\alpha}\leq L' \bigr\} ,\qquad \mathcal{U}= \bigl\{ \theta
\in \mathcal{A}:h(\theta ,\theta _{0})\leq L\delta _{n}
\bigr\} . \label{eq:A_U} 
\end{align}
Due to the bias bound
$\lVert \theta _{*,p}-\theta _{0}\rVert \leq C\delta _{n}$, invoking the
upper and lower bounds in \eqref{eq:hellingerCondition} and the triangle
inequality for the Hellinger distance yields
\begin{equation*}
\mathcal{U}\subset \bigl\{ \theta \in \mathbb{R}^{p}:\lVert \theta -
\theta _{*,p}\rVert ^{\beta}\leq c_{L'}(L+c_{L'}C)
\delta _{n} \bigr\} .
\end{equation*}
It is therefore enough to show that the posterior contracts in Hellinger
distance on the event $\bar{\mathcal{D}}_{n}(c')$, that is, for any
$C_{1}>0$ and large enough $L$, $L'$ there exist $C_{2},C_{3}>0$ with
\begin{equation}
\mathbb{P}_{\theta _{0}}^{n} \bigl(\Pi \bigl(\mathcal{U}^{c}|Z^{(n)}
\bigr)>e^{-C_{1}n
\delta _{n}^{2}},\bar{\mathcal{D}}_{n}\bigl(c'
\bigr) \bigr)\leq C_{3}e^{-C_{2}n
\delta _{n}^{2}}. \label{eq:hellingerContraction} 
\end{equation}
First, the entropy bound in Lemma~\ref{lem:setA} implies by
\citep[Theorem~7.1.4]{gine2016} the existence of tests $\Psi $ with values
in $[0,1]$ such that for all $n$, any large enough $L$, $L'$ and some
$C'>0$
\begin{equation}
\mathbb{E}_{\theta _{0}}^{n}\Psi \leq e^{-C'L'n\delta _{n}^{2}}, \qquad \sup
_{\theta \in \mathcal{U}^{c}\cap \mathcal{A}}\mathbb{E}_{
\theta}^{n}(1-\Psi )\leq
e^{-C'L'n\delta _{n}^{2}}. \label{eq:tests} 
\end{equation}
We therefore only have to prove \eqref{eq:hellingerContraction} for the
probability in question restricted to $\{\Psi =0\}$. On
$\bar{\mathcal{D}}_{n}(c')$, we can lower bound the normalising factors
in the posterior density such that for all
$\theta \in \mathbb{R}^{p}$
\begin{align}
\pi \bigl(\theta |Z^{(n)}\bigr) & = \frac{e^{\ell _{n}(\theta )-\ell _{n}(\theta _{0})}\pi (\theta )}{\int _{\Theta}e^{\ell _{n}(\theta )-\ell _{n}(\theta _{0})}\pi (\theta )\,d\theta} \leq
e^{c'n\delta _{n}^{2}}e^{\ell _{n}(\theta )-\ell _{n}(\theta _{0})} \pi (\theta ). \label{eq:densityUpperBound} 
\end{align}
The Markov inequality and Fubini's theorem now yield
\begin{align*}
& \mathbb{P}_{\theta _{0}}^{n} \bigl(\Pi \bigl(
\mathcal{U}^{c}|Z^{(n)}\bigr)>e^{-C_{1}n
\delta _{n}^{2}},\{\Psi =0\},
\bar{\mathcal{D}}_{n}\bigl(c'\bigr) \bigr)
\\
& \quad \leq \mathbb{P}_{\theta _{0}}^{n} \biggl((1-\Psi )
\int _{
\mathcal{U}^{c}}e^{\ell _{n}(\theta )-\ell _{n}(\theta _{0})}\pi ( \theta )\,d\theta
>e^{-(C_{1}+c')n\delta _{n}^{2}} \biggr)
\\
& \quad \leq e^{(C_{1}+c')n\delta _{n}^{2}}
\int _{\mathcal{U}^{c}} \mathbb{E}_{\theta _{0}}^{n} \bigl((1-
\Psi )e^{\ell _{n}(\theta )-
\ell _{n}(\theta _{0})} \bigr)\pi (\theta )\,d\theta
\\
& \quad \leq e^{(C_{1}+c')n\delta _{n}^{2}}
\int _{\mathcal{U}^{c}} \mathbb{E}_{\theta}^{n} (1-\Psi )
\pi (\theta )\,d\theta .
\end{align*}
Integrating separately over the sets
$\mathcal{U}^{c}\cap \mathcal{A}$ and
$\mathcal{U}^{c}\cap \mathcal{A}^{c}$, the second bound on the tests in
(\ref{eq:tests}) and the excess mass condition from Lemma~\ref{lem:setA}, together with $\Psi \leq 1$, give for any large enough
$L'$
\begin{align*}
\int _{\mathcal{U}^{c}\cap \mathcal{A}}\mathbb{E}_{\theta}^{n}(1- \Psi )
\pi (\theta )\,d\theta & \leq e^{-C'L'n\delta _{n}^{2}},
\\
\int _{\mathcal{U}^{c}\cap \mathcal{A}^{c}}\mathbb{E}_{\theta}^{n}(1- \Psi )
\pi (\theta )\,d\theta & \leq \Pi \bigl(\mathcal{A}^{c}\bigr)\leq
e^{-((L')^{2}/16)n
\delta _{n}^{2}}.
\end{align*}
Taking $L'$ sufficiently large this shows (\ref{eq:hellingerContraction})
and finishes the proof.
\end{proof}

\subsection{Proofs for specific models}
\label{sec:proofSpecificModels}

\begin{proof}[Proof of Proposition~\ref{prop:ExpoSurrogate}]
We write $(Y,X)$ for a generic copy of $(Y_{i},X_{i})$. Observe for
$v\in \mathbb{R}^{p}$ the identities
\begin{align}
\ell (\theta ) & =Yb_{\theta}(X)-A\bigl(b_{\theta}(X)\bigr),\qquad
\mathbb{E}_{
\theta _{0}}Y=\mathbb{E}_{\theta _{0}}A'
\bigl(b_{\theta _{0}}(X)\bigr), \label{eq:likelihoodIdentities} 
\\
v^{\top}\nabla \ell (\theta ) & =\bigl(Y-A'
\bigl(b_{\theta}(X)\bigr)\bigr)v^{\top}\nabla b_{
\theta}(X),
\label{eq4.13}
\\
v^{\top}\nabla ^{2}\ell (\theta )v & =\bigl(Y-A'
\bigl(b_{\theta}(X)\bigr)\bigr)v^{\top} \nabla ^{2}b_{\theta}(X)v-A''
\bigl(b_{\theta}(X)\bigr) \bigl(v^{\top}\nabla b_{\theta}(X)
\bigr)^{2}, \label{eq4.14}
\end{align}
where for $i,j=1,\dots ,p$
\begin{align}
\partial _{\theta _{i}}b_{\theta} & = \bigl(\bigl(A'
\bigr)^{-1}\circ g^{-1}\bigr)'\bigl( \mathcal{G}(
\theta )\bigr)\partial _{\theta _{i}} \mathcal{G}(\theta ), \label{eq4.15}
\\
\partial _{\theta _{i}\theta _{j}}b_{\theta} & = \bigl(\bigl(A'
\bigr)^{-1}\circ g^{-1}\bigr)''
\bigl( \mathcal{G}(\theta )\bigr)\partial _{\theta _{i}}\mathcal{G}(\theta )
\partial _{\theta _{j}}\mathcal{G}(\theta )+\bigl(\bigl(A'
\bigr)^{-1}\circ g^{-1}\bigr)'\bigl( \mathcal{G}(
\theta )\bigr)\partial _{\theta _{i}\theta _{j}}\mathcal{G}( \theta ). \label{eq:der2}
\end{align}
Denote the range of the map
$(\theta ,x)\mapsto \mathcal{G}(\theta )(x)$, evaluated for
$\theta $ in the convex hull
$R=\operatorname{conv}(\mathcal{B}\cup \{\theta _{0}\})$ and
$x\in \mathcal{X}$, by $R$. Similarly, let $\tilde{R}$ be the range of
$(\theta ,x)\mapsto b_{\theta}(x)$ evaluated over $R$. Then, $R$ and
$\tilde{R}$ are bounded subsets of $\mathbb{R}$. Since $A$ is smooth and
convex and $g\in C^{3}(\mathcal{I})$ is invertible, this means that
$(A')^{-1}\circ g^{-1}$ and its first three derivatives evaluated over
$R$, as well as $A$ and its first three derivatives evaluated over
$\tilde{R}$ are uniformly bounded by a constant $c_{A}>0$, such that also
$A''$ is lower bounded there by $c_{A}^{-1}$. Without loss of generality
let $c_{A}\geq 1$. The bias condition and the upper bound in
\ref{assu:GLM_local_reg}(iv) imply
$\lVert \mathcal{G}(\theta _{0})-\mathcal{G}(\theta _{*,p})\rVert _{L^{2}}
\leq C\delta _{n}$. Given the partial derivatives of $b_{\theta}$ and Assumption~\ref{assu:GLM_local_reg} we thus get
\begin{align}
& \lVert b_{\theta _{0}}-b_{\theta _{*,p}}\rVert _{L^{2}}\leq
c_{A} \bar c_{\max}\delta _{n},\qquad \lVert
b_{\theta }-b_{\theta _{*,p}} \rVert _{L^{2}}\leq c_{A}
\bar c_{\max}\eta , \label{eq4.17}
\\
& \lVert \nabla b_{\theta }\rVert _{L^{\infty}(\mathcal{X},\mathbb{R}^{p})} \leq c_{A}
\bar c_{\max}p^{k_{1}},\qquad \lVert \nabla b_{\theta }
\rVert _{L^{2}(\mathcal{X},\mathbb{R}^{p\times p})}\leq c_{A}\bar c_{
\max}p^{k_{3}},
\label{eq4.18}
\\
& \bigl\lVert \nabla ^{2}b_{\theta }\bigr\rVert _{L^{\infty}(\mathcal{X},
\mathbb{R}^{p\times p})}
\leq 2c_{A}\bar c_{\max}^{2} p^{\max (2k_{1},k_{2})},
\qquad \bigl\lVert \nabla ^{2}b_{\theta }\bigr\rVert
_{L^{2}(\mathcal{X},
\mathbb{R}^{p\times p})}\leq 2c_{A}\bar c_{\max}^{2}
p^{\max (2k_{3},k_{4})}, \label{eq4.19}
\\
& \bigl\lVert \nabla ^{2}b_{\theta }-\nabla ^{2}b_{\theta '}
\bigr\rVert _{L^{
\infty}(\mathcal{X},\mathbb{R}^{p\times p})}\leq 5c_{A}2\bar c_{\max}^{2}p^{
\max (3k_{1},k_{1}+k_{2})}
\bigl\lVert \theta -\theta '\bigr\rVert . \label{eq4.20}
\end{align}
This also yields
\begin{align}
\lVert \nabla b_{\theta }\rVert _{L^{q}(\mathcal{X},\mathbb{R}^{p})} &\leq \lVert \nabla
b_{\theta }\rVert _{L^{\infty}(\mathcal{X},
\mathbb{R}^{p})}^{(q-2)/q}\lVert \nabla
b_{\theta }\rVert _{L^{2}(
\mathcal{X},\mathbb{R}^{p})}^{2/q}\leq c_{A}
\bar c_{\max}p^{(q-2)k_{1}/q+2k_{3}/q}, \label{eq:b_der2} 
\\
\quad
\bigl\lVert \nabla ^{2} b_{\theta }\bigr\rVert _{L^{q}(\mathcal{X},\mathbb{R}^{p
\times p})}
& \leq \bigl\lVert \nabla ^{2} b_{\theta }\bigr\rVert
_{L^{\infty}(
\mathcal{X},\mathbb{R}^{p\times p})}^{(q-2)/q}\bigl\lVert \nabla ^{2}
b_{
\theta }\bigr\rVert _{L^{2}(\mathcal{X},\mathbb{R}^{p\times p})}^{2/q} \leq
c_{A}2\bar c_{\max}^{2}p^{(q-2)\max (2k_{1},k_{2})/q+2\max (2k_{3},k_{4})/q}.
\label{eq:b_der3} 
\end{align}
By standard properties of exponential families there exists
$0<\lambda \leq 1$ such that uniformly for $x\in \mathcal{X}$
\begin{align*}
\mathbb{E}_{\theta _{0}} \bigl[e^{\lambda Y}|X=x \bigr] &
=e^{A(\lambda +b_{\theta _{0}}(x))-A(b_{\theta _{0}}(x))}\leq e^{2c_{A}}
\end{align*}
after possibly increasing the constant $c_{A}$, and
$\mathbb{E}_{\theta _{0}}[e^{\lambda |Y|} \rrvert X=x]\leq 2e^{2c_{A}}$.
In particular, for all $q\geq 2$,
\begin{equation}
\mathbb{E}_{\theta _{0}} \bigl[ \bigl\llvert Y-A'
\bigl(b_{\theta }(x)\bigr) \bigr\rrvert ^{q} |X=x \bigr]\leq q!
\lambda ^{-q}\mathbb{E}_{\theta _{0}}\bigl[e^{
\lambda  \llvert Y \rrvert +\lambda c_{A}}|X=x
\bigr]\leq q!\lambda ^{-q}2e^{3c_{A}}. \label{eq:expYA}
\end{equation}
With these preparations let us verify the conditions of Theorem~\ref{thm:localisation} and Theorem~\ref{thm:GaussianContraction}. We use
repeatedely the results \eqref{eq:likelihoodIdentities}-\eqref{eq:expYA}
without explicit mention.

\emph{Theorem~\ref{thm:localisation}(i).} This follows from Assumption~\ref{assu:GLM_local_reg}(i).

\emph{Theorem~\ref{thm:localisation}(ii).} Let
$\theta ,\theta '\in \mathcal{B}$. By the Cauchy-Schwarz inequality
\begin{align*}
\bigl\llvert \mathbb{E}_{\theta _{0}}v^{\top}\nabla \ell (\theta
_{*,p}) \bigr\rrvert & = \bigl\llvert \mathbb{E}_{\theta _{0}}\bigl[
\bigl(A'\bigl(b_{\theta _{0}}(X)\bigr)-A'
\bigl(b_{
\theta _{*,p}}(X)\bigr)\bigr)v^{\top}\nabla b_{\theta _{*,p}}(X)
\bigr] \bigr\rrvert
\\
&\leq c_{X}c_{A}\lVert b_{\theta _{0}}-b_{\theta _{*,p}}
\rVert _{L^{2}} \lVert \nabla b_{\theta _{*,p}}\rVert _{L^{2}(\mathcal{X},\mathbb{R}^{p})}
\leq c_{X}c_{A}^{3}\bar c_{\max}^{2}
\delta _{n}p^{k_{3}}.
\end{align*}
Hence, by conditioning on $X$ and the tower property
\begin{align*}
\mathbb{E}_{\theta _{0}} \bigl\llvert v^{\top}\nabla \ell (\theta
_{*,p}) \bigr\rrvert ^{q} & \leq q!\lambda
^{-q}2e^{3c_{A}}c_{X} \lVert \nabla
b_{\theta _{*,p}} \rVert _{L^{q}(\mathcal{X},\mathbb{R}^{p})}^{q}\leq q!\lambda
^{-q}2e^{3c_{A}}c_{X} c_{A}^{q}
\bar c_{\max}^{q}p^{(q-2)k_{1}+2k_{3}},
\end{align*}
as well as
\begin{align*}
\bigl(\mathbb{E}_{\theta _{0}} \bigl\llvert v^{\top}\nabla
^{2}\ell (\theta )v \bigr\rrvert ^{q}\bigr)^{1/q}
& \leq \bigl(\mathbb{E}_{\theta _{0}} \bigl\llvert \bigl(Y-A'
\bigl(b_{\theta}(X)\bigr)\bigr)v^{T}\nabla ^{2}
b_{\theta}(X)v \bigr\rrvert ^{q}\bigr)^{1/q}+\bigl(
\mathbb{E}_{\theta _{0}} \bigl\llvert A''
\bigl(b_{\theta}(X)\bigr) \bigl(v^{T} \nabla
b_{\theta}(X)\bigr)^{2} \bigr\rrvert ^{q}
\bigr)^{1/q}
\\
& \leq \bigl(q!\lambda ^{-q}2e^{3c_{A}}c_{X}
\bigr)^{1/q}\bigl\lVert \nabla ^{2} b_{
\theta }\bigr
\rVert _{L^{q}(\mathcal{X},\mathbb{R}^{p\times p})}+c_{A}c_{X}^{1/q} \lVert
\nabla b_{\theta }\rVert _{L^{2q}(\mathcal{X},\mathbb{R}^{p})}^{2},
\end{align*}
so
\begin{align*}
\mathbb{E}_{\theta _{0}} \bigl\llvert v^{\top}\nabla ^{2}
\ell (\theta )v \bigr\rrvert ^{q} & \leq q!\lambda ^{-q}2e^{3c_{A}}c_{X}c_{A}^{3q}2^{q}
\bar c_{\max}^{2q}p^{(q-2)
\max (2k_{1},k_{2})+2\max (2k_{3},k_{4})}.
\end{align*}
Next, decompose
\begin{align*}
 v^{\top} \bigl(\nabla ^{2}\ell (\theta )-\nabla
^{2}\ell \bigl(\theta '\bigr) \bigr)v &=
\bigl(A'\bigl(b_{\theta '}(X)\bigr)-A'
\bigl(b_{\theta}(X)\bigr)\bigr)v^{\top}\nabla ^{2}b_{
\theta}(X)v+
\bigl(Y-A'\bigl(b_{\theta '}(X)\bigr)\bigr)v^{\top}
\bigl(\nabla ^{2}b_{\theta}(X)- \nabla ^{2}b_{\theta '}(X)
\bigr)v
\\
& \quad{} +\bigl(A''\bigl(b_{\theta '}(X)
\bigr)-A''\bigl(b_{\theta}(X)\bigr)\bigr)
\bigl(v^{\top}\nabla b_{\theta}(X)\bigr)^{2}
\\
&\quad{} -A''
\bigl(b_{\theta '}(X)\bigr) \bigl(\bigl(v^{\top}\nabla
b_{\theta}(X)\bigr)^{2}-\bigl(v^{
\top}\nabla
b_{\theta '}(X)\bigr)^{2}\bigr).
\end{align*}
From this we find
\begin{align*}
& \bigl(\mathbb{E}_{\theta _{0}} \bigl\llvert v^{\top} \bigl(\nabla
^{2}\ell (\theta )- \nabla ^{2}\ell \bigl(\theta
'\bigr) \bigr)v \bigr\rrvert ^{q}\bigr)^{1/q}
\\
&\quad  \leq c_{A}^{1/q}\lVert b_{\theta '}-b_{\theta }
\rVert _{L^{\infty}} \bigl\lVert \nabla ^{2} b_{\theta }\bigr
\rVert _{L^{\infty}(\mathcal{X},
\mathbb{R}^{p\times p})} + \bigl(q!\lambda ^{-q}2e^{3c_{A}}
\bigr)^{1/q}\bigl\lVert \nabla ^{2} b_{\theta }-\nabla
^{2} b_{\theta '}\bigr\rVert _{L^{\infty}(
\mathcal{X},\mathbb{R}^{p\times p})}
\\
&\qquad{} + c_{A}^{1/q}\lVert b_{\theta '}-b_{\theta }
\rVert _{L^{\infty}} \lVert \nabla b_{\theta }\rVert ^{2}_{L^{\infty}(\mathcal{X},
\mathbb{R}^{p})}
+ c_{A}^{1/q}\bigl(\lVert \nabla b_{\theta }\rVert
_{L^{
\infty}(\mathcal{X},\mathbb{R}^{p})}+\lVert \nabla b_{\theta '} \rVert _{L^{\infty}(\mathcal{X},\mathbb{R}^{p})}\bigr)
\lVert \nabla b_{
\theta '}-\nabla b_{\theta }\rVert _{L^{\infty}(\mathcal{X},
\mathbb{R}^{p})}
\\
&\quad  \leq 15\bigl(q!\lambda ^{-q}2e^{3c_{A}}c_{A}
\bigr)^{1/q}c_{A}^{3}\bar c_{\max}^{2}p^{
\max (3k_{1},k_{1}+k_{2})}
\bigl\lVert \theta -\theta '\bigr\rVert .
\end{align*}
In all, we verify the claim by setting
$C_{1}=60\max (c_{\max},1)\lambda ^{-1}e^{3c_{A}}c_{A}^{4}\bar c_{
\max}^{2}$, $\kappa _{1}=k_{3}$,
$\kappa _{2}=\max (k_{1},2k_{3},k_{4})$,
$\kappa _{4}=\max (3k_{1},k_{1}+k_{2})$.

\emph{Theorem~\ref{thm:localisation}(iii).} We have for
$\theta \in \mathcal{B}$
\begin{align*}
-\mathbb{E}_{\theta _{0}}\bigl(v^{\top}\nabla ^{2}\ell (
\theta )v\bigr) & = \mathbb{E}_{\theta _{0}}\bigl[\bigl(A'
\bigl(b_{\theta}(X)\bigr)-A'\bigl(b_{\theta _{0}}(X)\bigr)
\bigr)v^{
\top}\nabla ^{2}b_{\theta}(X)v\bigr]+
\mathbb{E}_{\theta _{0}}\bigl[A''\bigl(b_{
\theta}(X)
\bigr) \bigl(v^{\top}\nabla b_{\theta}(X)\bigr)^{2}\bigr]
\\
& \geq -c_{A}c_{X}\lVert b_{\theta }-b_{\theta _{0}}
\rVert _{L^{2}} \bigl\lVert \nabla ^{2} b_{\theta }\bigr
\rVert _{L^{2}(\mathcal{X},\mathbb{R}^{p
\times p})}+c_{A}^{-1}c^{-1}_{X}
\bigl\lVert v^{\top }\nabla b_{\theta } \bigr\rVert
^{2}_{L^{2}}
\\
& \geq -2c_{A}^{2}\bar c_{\max}^{2}(
\delta _{n}+\eta ) p^{\max (2k_{3},k_{4})} +c_{A}^{-2}
\bar c_{\min}^{2}p^{-2k_{5}}\geq (1/2)c_{A}^{-2}
\bar c_{
\min}^{2}p^{-2k_{5}},
\end{align*}
taking $n$ large enough in the last inequality and noting the growth bounds
on $\delta _{n}$, $\eta $. The claim follows with
$C_{2}=(1/2)c_{A}^{-2}\bar c_{\min}^{2}$, $\kappa _{3}=2k_{5}$.

\emph{Theorem~\ref{thm:GaussianContraction}(i).} Let $r>0$ and
$\theta \in \mathcal{B}_{n,r}$. From Assumption~\ref{assu:GLM_local_reg}(iv) we have
$\lVert \mathcal{G}(\theta )\rVert _{L^{\infty}}\leq \bar c_{r}$, and by
the same arguments as after \eqref{eq:der2} we can find a constant
$c_{A,r}\geq c_{A}\geq 1$ such that by the convexity of $A$
\begin{align*}
\bigl\llvert \mathbb{E}_{\theta _{0}} \bigl(\ell (\theta _{0})-
\ell ( \theta ) \bigr) \bigr\rrvert &= \bigl\llvert \mathbb{E}_{\theta _{0}}
\bigl(A'\bigl(b_{
\theta _{0}}(X)\bigr) \bigl(b_{\theta _{0}}(X)-b_{\theta}(X)
\bigr)+A\bigl(b_{\theta}(X)\bigr)-A\bigl(b_{
\theta _{0}}(X)\bigr) \bigr)
\bigr\rrvert
\\
&  \leq \frac{c_{X}c_{A,r}^{3}}{2}\lVert b_{\theta _{0}}-b_{
\theta }\rVert
_{L^{2}}^{2}\leq c_{X}c_{A,r}^{3}
\bigl(\lVert b_{\theta _{0}}-b_{
\theta _{*,p}}\rVert _{L^{2}}^{2}+
\lVert b_{\theta _{*,p}}-b_{\theta } \rVert _{L^{2}}^{2}
\bigr)\leq 2c_{X}c_{A,r}^{5}\bar
c_{\max}^{2}\delta _{n}^{2},
\end{align*}
as well as
\begin{align*}
\bigl(\mathbb{E}_{\theta _{0}} \bigl\llvert \ell (\theta _{0})-
\ell (\theta ) \bigr\rrvert ^{q}\bigr)^{1/q} &\leq \bigl(
\mathbb{E}_{\theta _{0}} \bigl\llvert Y\bigl(b_{\theta _{0}}(X)-b_{\theta}(X)
\bigr) \bigr\rrvert ^{q}\bigr)^{1/q} + \bigl(
\mathbb{E}_{\theta _{0}} \bigl\llvert A\bigl(b_{\theta}(X)\bigr)-A
\bigl(b_{\theta _{0}}(X)\bigr) \bigr\rrvert ^{q}\bigr)^{1/q}
\\
& \leq c_{X}^{1/q}\bigl(\bigl(q!\lambda
^{-q}2e^{3c_{A}}\bigr)^{1/q}+c_{A,r}\bigr)
\lVert b_{\theta _{0}}-b_{\theta }\rVert _{L^{q}}
\\
& \leq c_{A,r}^{1-2/q}c_{X}^{1/q}\bigl(
\bigl(q!\lambda ^{-q}2e^{3c_{A}}\bigr)^{1/q}+c_{A,r}
\bigr) \lVert b_{\theta _{0}}-b_{\theta }\rVert _{L^{2}}^{2/q},
\end{align*}
so
$\mathbb{E}_{\theta _{0}}|\ell (\theta _{0})-\ell (\theta )|^{q}
\leq q!c_{A,r}^{q}c_{X}\lambda ^{-q}2^{q+1}e^{3c_{A}} 4\bar c_{\max}^{2}
\delta _{n}^{2}$. The wanted constant $c_{r}$ is obtained from the last
three displays and the upper and lower bounds for the Hellinger distance
below.

\emph{Theorem~\ref{thm:GaussianContraction}(ii).} Let $r>0$,
$\theta ,\theta '\in \ell ^{2}(\mathbb{N})$ with
$\lVert \theta \rVert _{\alpha},\lVert \theta '\rVert _{\alpha}\leq r$. Let $c_{A,r}$ as in the last paragraph. 
The squared Hellinger distance between $\mathbb{P}_{\theta}$ and
$\mathbb{P}_{\theta '}$, equals
\begin{align*}
h^{2}\bigl(\theta ,\theta '\bigr) & =2-
\int _{\mathcal{Z}}2\sqrt{p_{\theta}(z)p_{
\theta '}(z)}\,d\nu
(z)=2\bigl(1-\mathbb{E}_{\theta _{0}}e^{-x}\bigr)
\end{align*}
with the random variable
$x=(A(b_{\theta}(X))+A(b_{\theta '}(X)))/2-A(u)\geq 0$ by convexity of
$A$, and $u=(b_{\theta}(X)+b_{\theta '}(X))/2$. Then $b_{\theta}(X)-u=(b_{\theta}(X)-b_{\theta'}(X))/2$. With this and some $\Xi,\bar\Xi\in\tilde R$ rewrite $x$ as
\begin{align*}
x & = \frac{A\bigl(b_{\theta}(X)\bigr)-A(u)}{2}+ \frac{A
\bigl(b_{\theta '}(X)\bigr)-A(u)}{2} =\frac{(b_{\theta}(X)-b_{\theta '}(X))^{2}}{4}\left(A''(\Xi)+A''(\bar\Xi)\right).
\end{align*}
From this we find that
$(c_{A,r}^{-1}/2)(b_{\theta}(X)-b_{\theta '}(X))^{2}\leq x\leq (c_{A,r}/2)(b_{
\theta}(X)-b_{\theta '}(X))^{2}$, so
\begin{equation*}
\bigl(c_{\max}c_{A,r}^{-2}/2\bigr)\bigl\lVert
\mathcal{G}(\theta )-\mathcal{G}\bigl( \theta '\bigr)\bigr\rVert
^{2}_{L^{2}}\leq x\leq \bigl(c_{\max}c^{2}_{A,r}/2
\bigr)\bigl\lVert \mathcal{G}(\theta )-\mathcal{G}\bigl(\theta '
\bigr)\bigr\rVert ^{2}_{L^{2}}
\end{equation*}
and thus by Assumption~\ref{assu:GLM_local_reg}(iv),
$(c_{\max}c_{A,r}^{-2}/2)c_{r}\lVert \theta -\theta '\rVert ^{2\beta}
\leq \mathbb{E}_{\theta _{0}}x\leq (c_{\max}c^{2}_{A,r}/2)c_{r}
\lVert \theta -\theta '\rVert ^{2}$. Arguing now as in
\citep[Proposition~1]{birge2004}, with $0\leq x\leq c'$ for some
$c'\equiv c'(c_{r})$ and convexity
%
\begin{equation}
e^{-x}\leq \frac{x}{c'}e^{-c'}+ \biggl(1-
\frac{x}{c'} \biggr)= \frac{e^{-c'}-1}{c'}x+1, \label{eq:Birge} 
\end{equation}
the result follows from
$h^{2}(\theta ,\theta ')\leq 2\mathbb{E}_{\theta _{0}}x$ and
$h^{2}(\theta ,\theta ')\geq 2\frac{1-e^{-z}}{z}\mathbb{E}_{\theta _{0}}x$.
\end{proof}

\begin{proof}[Proof of Lemma~\ref{lem:GLM_initialiser}]
With the true regression function in \eqref{eq:nonGaussRegression} define
$\beta _{0}=\Phi ^{-1}(g^{-1}(\Phi (\theta _{0})(\cdot )))$. From
$\theta _{0}\in h^{\alpha}(\mathbb{N})$ for $\alpha > 1/2$,
\eqref{eq:sobolev_equivalence}, \eqref{eq:Phi_C1} and the stability of
Sobolev spaces under composition with smooth functions (e.g.,
\cite[(1.93)]{yagi2009}) infer
$\beta _{0}\in h^{\alpha}(\mathbb{N})$. In order to estimate
$\beta _{0}$ from the data define the penalised least squares loss function
\begin{align*}
R_{\lambda}(\beta ) = \frac{1}{n}\sum_{i=1}^{n}
\bigl(Y_{i}-\Phi (\beta ) (X_{i})\bigr)^{2}+
\lambda \lVert \beta \rVert ^{2}_{\alpha},\quad \lambda >0,\beta
\in \mathbb{R}^{p}.
\end{align*}
Its minimiser $\hat\beta $ exists uniquely by strong convexity. To compute
$\hat\beta $ by a concrete algorithm, denote by
$\beta _{\operatorname{init}}=\beta ^{(k^{*})}\in \mathbb{R}^{p}$ the output
after $k^{*}$ gradient descent steps with iterates
\begin{align*}
\beta ^{(k+1)} = \beta ^{(k)} - \tilde{\gamma}\nabla
R_{\lambda}\bigl( \beta ^{(k)}\bigr),
\end{align*}
starting from $0$ with step size $\tilde{\gamma}$ specified below. Let
$\beta _{*,p}=(\beta _{0,1},\dots ,\beta _{0,p})$ be the
$\mathbb{R}^{p}$-projection of $\beta _{0}$, and set
$\hat\mu =\Phi (\hat\beta )$, $\mu _{*,p}=\Phi (\beta _{*,p})$,
$\mu _{0}=\Phi (\beta _{0})$,
$\mu _{\operatorname{init}}=\Phi (\beta _{\operatorname{init}})$. Define
$\theta _{\operatorname{init}}\in \mathbb{R}^{p}$ with components
\begin{align*}
\theta _{\operatorname{init},k} := \bigl\langle g(\mu _{
\operatorname{init}}),e_{k}
\bigr\rangle _{L^{2}}.
\end{align*}
We show below for large enough $k^{*}$ and
$\lambda =(\delta _{n}p^{\bar\alpha})^{4\alpha}$ that there exist constants
$M_{1},M_{2},c'_{1},c'_{2}>0$ with
\begin{align}
\mathbb{P}^{n}_{\theta _{0}} \bigl(\lVert \mu _{\operatorname{init}}- \mu
_{0}\rVert _{L^{2}}\leq M_{1}p^{-\bar\alpha},
\lVert \beta _{
\operatorname{init}}\rVert _{\bar\alpha}\leq M_{2} \bigr)
\geq 1- c'_{2} e^{-c'_{1} n\delta _{n}^{2}}. \label{eq:initial_1}
\end{align}
From \eqref{eq:Phi_C1} and
$\lVert \beta _{0}\rVert _{\alpha}<\infty $, $\alpha >1/2$, we find
$\lVert \mu _{\operatorname{init}}\rVert _{L^{\infty}},\lVert \mu _{0}
\rVert _{L^{\infty}}\leq c$ with high $\mathbb{P}^{n}_{\theta _{0}}$-probability
for a constant $c>0$, and the derivative of $g$ is bounded on
$[-c,c]$. Hence, by \eqref{eq:initial_1} we have with high
$\mathbb{P}^{n}_{\theta _{0}}$-probability
\begin{align*}
\lVert \theta _{\operatorname{init}}-\theta _{0}\rVert =\bigl\lVert g(\mu
_{
\operatorname{init}})-g(\mu _{0})\bigr\rVert _{L^{2}}\lesssim
\lVert \mu _{
\operatorname{init}}-\mu _{0}\rVert _{L^{2}}\leq
M_{1} p^{-\bar\alpha}.
\end{align*}
The result follows from
$\lVert \theta _{*,p}-\theta _{0}\rVert \lesssim \delta _{n}< p^{-
\bar\alpha}$, noting that $\delta _{n}=n^{-\alpha /(2\alpha +1)}$,
$p\leq n^{1/(2\alpha +1)}$, $\bar\alpha >1/2$, and from choosing
$\iota $ large enough. We are left with proving \eqref{eq:initial_1}. We
do this in several steps.

\emph{Step~1.} Set $\lambda =(\delta _{n}p^{\bar\alpha})^{4\alpha}$ and recall the regression notation in \eqref{eq:nonGaussRegression} with the errors $\varepsilon_i$. Introduce for a function
$u:\mathcal{X}\rightarrow \mathbb{R}$ the empirical norm and
the \emph{empirical inner product}
\begin{align*}
\lVert u\rVert _{(n)}^{2} = \frac{1}{n}\sum
_{i=1}^{n} u(X_{i})^{2}, \qquad
\langle \varepsilon ,u\rangle _{(n)} = \frac{1}{n}\sum
_{i=1}^{n} \varepsilon _{i}
u(X_{i}).
\end{align*}
With this notation, rewrite $R_{\lambda}(\hat\beta )\leq R_{\lambda}(\beta _{*,p})$ as the basic inequality 
\begin{align}
\lVert \hat \mu -\mu _{0}\rVert _{(n)}^{2} +
\lambda \lVert \hat \beta \rVert _{\alpha}^{2} \leq 2\langle
\varepsilon ,\hat\mu - \mu _{*,p}\rangle _{(n)} + \lVert \mu
_{*,p}-\mu _{0}\rVert _{(n)}^{2} +
\lambda \lVert \beta _{*,p}\rVert _{\alpha}^{2}.\label{eq:basic_inequality}
\end{align}
By \eqref{eq:Phi_C1_Rp} and $\alpha >2\bar\alpha $ we have for
the bias the upper bound
\begin{align*}
\lVert \mu _{*,p}-\mu _{0}\rVert ^{2}_{(n)}
&\leq \lVert \mu _{*,p}- \mu _{0}\rVert
^{2}_{L^{\infty}} \lesssim \lVert \beta _{*,p}-\beta
_{0} \rVert ^{2}_{\bar\alpha} =\sum
_{k\geq p+1}k^{2\bar\alpha}\beta ^{2}_{0,k}
\leq p^{-2(\alpha -\bar\alpha )}\lVert \beta _{0}\rVert ^{2}_{\alpha}
\lesssim p^{-2\bar\alpha},
\end{align*}
so for $M_{1}'>0$,
$\lVert \mu _{*,p}-\mu _{0}\rVert _{(n)}\leq 2M_{1}'p^{-\bar\alpha}$. Similar to the proof of Theorem 10.2 in \cite{geer2000empirical} define for $\beta \in \mathbb{R}^{p}$
\begin{align*}
\tilde\beta = (\beta -\beta _{*,p})/(\lVert \beta \rVert _{\alpha}+
\lVert \beta _{0}\rVert _{\alpha}),\quad Z(\beta )=\langle \varepsilon ,\Phi (\tilde\beta )\rangle _{(n)}
\lVert \Phi (\tilde \beta )\rVert ^{1/(2\alpha )-1}_{(n)}.
\end{align*}
We infer from the basic inequality \eqref{eq:basic_inequality}
that either
\begin{align}
\lVert \hat \mu -\mu _{*,p}\rVert _{(n)}^{2}
+ \lambda \lVert \hat \beta \rVert _{\alpha}^{2} \leq 2\lVert
\hat \mu -\mu _{*,p} \rVert ^{1-1/(2\alpha )}_{(n)}\bigl(\lVert
\hat \beta \rVert _{\alpha}+ \lVert \beta _{0}\rVert
_{\alpha}\bigr)^{1/(2\alpha )}|Z(\hat\beta )|, \label{eq:fundEq1} 
\end{align}
or
\begin{align}
\lVert \hat \mu -\mu _{*,p}\rVert _{(n)}^{2}
+ \lambda \lVert \hat \beta \rVert _{\alpha}^{2} \leq 
\lVert \mu _{*,p}- \mu _{0}\rVert
_{(n)}^{2} + \lambda \lVert \beta _{*,p}\rVert
_{
\alpha}^{2}. \label{eq:fundEq2}
\end{align}
Consider first \eqref{eq:fundEq1} and suppose
$\lVert \hat \beta \rVert _{\alpha} > \lVert \beta _{0}\rVert _{
\alpha}$. We obtain the two inequalities 
\begin{align*}
  &\lVert \hat \mu -\mu _{*,p}\rVert _{(n)}
\lesssim \left(\lVert
\hat \beta \rVert _{\alpha}^{1/(2\alpha )}|Z(\hat\beta )|\right)^{1/(1+1/(2\alpha))},\quad \lambda \lVert \hat \beta \rVert _{\alpha}^{2} \lesssim \lVert
\hat \mu -\mu _{*,p} \rVert ^{1-1/(2\alpha )}_{(n)}\lVert
\hat \beta \rVert _{\alpha}^{1/(2\alpha )}|Z(\hat\beta )|. 
\end{align*}
Inserting the first inequality into the second allows for an upper bound on $\lVert \hat\beta\rVert _{\alpha}$ in terms of $\lambda$ and $Z(\hat \beta)$. Using that bound in the first inequality, we find in all
\begin{equation*}
\lVert \hat \mu -\mu _{*,p}\rVert _{(n)}\lesssim \lambda
^{-1/(4
\alpha )} \bigl\llvert Z(\hat\beta ) \bigr\rrvert ,\qquad \lVert \hat \beta
\rVert _{\alpha} \lesssim \lambda ^{(1+1/(2\alpha ))/2} \bigl\llvert Z(\hat\beta
) \bigr\rrvert .
\end{equation*}
We show in Step~4 below that
$\sup_{\beta \in \mathbb{R}^{p}}|Z(\beta )|\lesssim \delta _{n}$ with
high $\mathbb{P}^{n}_{\theta _{0}}$-probability. Therefore, the choice
$\lambda =(\delta _{n}p^{\bar\alpha})^{4\alpha}$ yields on the same high-probability event
\begin{align}
\lVert \hat \mu -\mu _{*,p}\rVert _{(n)}\lesssim
p^{-\bar\alpha}, \qquad \lVert \hat \beta \rVert _{\alpha} \lesssim 1.
\label{eq:initial_3} 
\end{align}
On the other hand, if
$\lVert \hat \beta \rVert _{\alpha} \leq \lVert \beta _{0}\rVert _{
\alpha}$, then solving \eqref{eq:fundEq1} gives
$\lVert \hat \mu -\mu _{*,p}\rVert _{(n)}^{1+1/(2\alpha )} \lesssim |Z(
\hat\beta )|$, and thus \eqref{eq:initial_3} with high
$\mathbb{P}^{n}_{\theta _{0}}$-probability. Finally, in the case of
\eqref{eq:fundEq2}, \eqref{eq:initial_3} holds by the bias bound and
$\lVert \beta _{*,p}\rVert _{\alpha}\leq \lVert \beta _{0}\rVert _{
\alpha}<\infty $, as well as by noting $\delta_n\leq p^{-2\bar \alpha}$ such that $\lambda\leq p^{-\bar \alpha}$. In all, we have shown that for large enough constants
$M'_{1},M'_{2},c_{1}'',c_{2}''>0$
\begin{align*}
\mathbb{P}^{n}_{\theta _{0}} \bigl(\lVert \hat \mu -\mu
_{*,p}\rVert _{(n)} \leq M_{1}'p^{-\bar\alpha},
\lVert \hat \beta \rVert _{\alpha}\leq M_{2}' \bigr)
\geq 1-c_{2}'' e^{-c_{1}''n\delta _{n}^{2}}.
\end{align*}

\emph{Step~2.} As in Step IV of the proof from
\cite[Theorem B.6]{nickl2020a} with minor notational changes we can show
the restricted isometry type bound
\begin{align*}
\mathbb{P}^{n}_{\theta _{0}} \biggl( \biggl\llvert
\frac{\lVert \hat \mu -\mu _{*,p}\rVert ^{2}_{(n)}}{\lVert \hat \mu -\mu _{*,p}\rVert ^{2}_{L^{2}}}-1 \biggr\rrvert \leq \frac{1}{2} \biggr)\geq 1-\bar
c_{2}e^{-\bar c_{1} n
\delta _{n}^{2}}
\end{align*}
for constants $\bar c_{1},\bar c_{2}>0$, noting that the upper bounds on
$|Z_{i}|$ and $\mathbb{E}[Z_{i}^{2}]$ in the proof need to be replaced
by $U=p^{2\bar\alpha}$ using the Sobolev embedding
\eqref{eq:Phi_C1_Rp} such that $n/U\geq n\delta _{n}^{2}$ for
$\bar\alpha >1/2$. As in \cite[(197)]{nickl2020a} we obtain from this and
the previous step constants $c'''_{1},c'''_{2}>0$ such that
\begin{align*}
\mathbb{P}^{n}_{\theta _{0}} \bigl(\lVert \hat \mu -\mu
_{0}\rVert _{L^{2}} \leq 4M_{1}'p^{-\bar\alpha},
\lVert \hat \beta \rVert _{\alpha}\leq M_{2}' \bigr)
\geq 1-c'''_{2}
e^{-c'''_{1}n\delta _{n}^{2}}.
\end{align*}

\emph{Step~3.} The loss function $f=R_{\lambda}$ satisfies
\eqref{eq:stronglyConvex} with curvature and Lipschitz constants
$m_{f}=\lambda $,
\begin{align*}
\Lambda _{f} &= \sup_{v\in \mathbb{R}^{p}:\lVert v\rVert =1}v^{T}
\nabla ^{2} R_{\lambda}(\beta )v=\sup_{v\in \mathbb{R}^{p}:\lVert v
\rVert =1}
\frac{2}{n}\sum_{i=1}^{n}\bigl(\Phi
(v) (X_{i})\bigr)^{2}+2\lambda v^{T} \Sigma
_{\alpha}v.
\end{align*}
By standard results from convex optimisation (e.g.,
\cite[Theorem~1]{dalalyan2017}) we find
$\beta _{\operatorname{init}}\in \mathbb{R}^{p}$ with
$\lVert \beta _{\operatorname{init}}-\hat \beta \rVert \leq M_{1}'p^{-
\bar\alpha}$ after at least $k^{*}$ steps of gradient descent as discussed
above such that
\begin{align}
2\frac{R_{\lambda}(0)-R_{\lambda}(\hat\beta )}{m_{f}} \biggl(1- \frac{m_{f}}{2\Lambda _{f}} \biggr)^{k^{*}}\leq
M_{1}'p^{-\bar\alpha}. \label{eq:initial_5} 
\end{align}
Hence,
\begin{align*}
\lVert \beta _{\operatorname{init}}\rVert _{\bar\alpha} \leq \lVert \beta
_{\operatorname{init}}-\hat \beta \rVert _{\bar\alpha} + \lVert \hat \beta \rVert
_{\alpha}\leq p^{\bar\alpha}\lVert \beta _{
\operatorname{init}}-\hat \beta
\rVert + \lVert \hat \beta \rVert _{
\alpha} \lesssim 1.
\end{align*}
Together with the result in Step~2 obtain \eqref{eq:initial_1}. To finish
the proof we determine a lower bound on $k^{*}$. Let us first find a crude
upper bound on $R_{\lambda}(0)=\sum_{i=1}^{n}Y_{i}^{2}/n$. From
\eqref{eq:expYA} we have
$\mathbb{E}^{n}_{\theta _{0}}[\sum_{i=1}^{n}|Y_{i}|]\leq a_{1}n$ for
$a_{1}>0$. By possibly increasing $a_{1}$, we find from the Bernstein inequality
and \eqref{eq:expYA} for some $a_{2}>0$
\begin{align*}
\mathbb{P}^{n}_{\theta _{0}} \bigl(R_{\lambda}(0)\geq
4a_{1}^{2}n \bigr)\leq \mathbb{P}^{n}_{\theta _{0}}
\Biggl( \Biggl(\sum_{i=1}^{n} \llvert
Y_{i} \rrvert \Biggr)^{2}\geq 4a_{1}^{2}n^{2}
\Biggr) \leq \mathbb{P}^{n}_{\theta _{0}} \Biggl(\sum
_{i=1}^{n}\bigl( \llvert Y_{i} \rrvert -
\mathbb{E}^{n}_{\theta _{0}} \llvert Y_{i} \rrvert \bigr)
\geq a_{1}n\delta _{n} \Biggr)\leq 2e^{-a_{2}n\delta _{n}^{2}}.
\end{align*}
Hence, $R_{\lambda}(\hat\beta )\leq R_{\lambda}(0)\lesssim n$ with high
$\mathbb{P}^{n}_{\theta _{0}}$-probability. Moreover, using the Sobolev
embedding \eqref{eq:Phi_C1_Rp} and
$\lambda =(\delta _{n}p^{\bar\alpha})^{4\alpha}$ we have
$\Lambda _{f}\lesssim p^{2\bar\alpha}+\lambda p^{2\alpha}=\lambda (
\delta _{n}^{2}\lambda ^{1/(2\alpha )-1}+p^{2\alpha})$. Taking logarithms
and rearranging in \eqref{eq:initial_5} shows that it is enough to take
$(1+\log n)n^{2\alpha /(2\alpha +1)}\lesssim k^{*}$.

\emph{Step~4.} We argue now that there exist $C_{1},C_{2},C_{3}>0$ with
\begin{align}
\mathbb{P}^{n}_{\theta _{0}} \Bigl(\sup_{\beta \in \mathbb{R}^{p}}
\bigl\llvert Z( \beta ) \bigr\rrvert \geq C_{1}\delta _{n}
\Bigr)\leq C_{3} e^{-C_{2} n\delta _{n}^{2}}, \label{eq:initial_2} 
\end{align}
following the rescaling argument of
\cite[Chapter~10]{geer2000empirical} for unbounded regression functions.
Using the norm equivalence \eqref{eq:sobolev_equivalence} and
\cite[(4.184)]{gine2016}, together with standard extension properties of
Sobolev norms, we have the metric entropy bound
\begin{equation*}
H \bigl(\bigl\{\Phi (\beta ):\beta \in \mathbb{R}^{p},\lVert \beta
\rVert _{
\alpha}\leq 1\bigr\},\lVert \cdot \rVert _{L^{\infty}},
\varepsilon \bigr) \lesssim \varepsilon ^{-1/\alpha}.
\end{equation*}
Up to a constant, the same entropy bound holds for the shifted and normalised
vectors $\tilde\beta $:
\begin{equation*}
H \bigl(\bigl\{\Phi (\tilde\beta ):\beta \in \mathbb{R}^{p}\bigr\},
\lVert \cdot \rVert _{L^{\infty}},\varepsilon \bigr)\lesssim \varepsilon
^{-1/
\alpha}.
\end{equation*}
The key property is the following uniform boundedness due to the Sobolev
embedding \eqref{eq:SobolevEmbedding} and the normalisation:
\begin{align*}
\sup_{\beta \in \mathbb{R}^{p}}\bigl\lVert \Phi (\tilde\beta )\bigr\rVert
_{L^{
\infty}}\lesssim 1.
\end{align*}
The inequalities in the last two displays also hold with the empirical
norm $\lVert \cdot \rVert _{(n)}$ replacing the
$\lVert \cdot \rVert _{L^{\infty}}$ norm. Finally, the errors
$\varepsilon _{i}$ have exponential moments as discussed in the proof of
Proposition~\ref{prop:ExpoSurrogate}. Arguing as in
\cite[Lemma~8.4]{geer2000empirical}, but replacing Corollary~8.3 there
with Corollary~8.8, proves \eqref{eq:initial_2}.
\end{proof}

\begin{proof}[Proof of Theorem~\ref{thm:density}]
The operators $A:L^{\infty}(\mathcal{X})\rightarrow \mathbb{R}$,
$A(u)=\log \int _{\mathcal{X}}e^{u(x)}\,d\nu _{\mathcal{X}}(x)$, are Fr\'{e}chet
differentiable with derivatives obtained according to the chain rule for
$u,h\in L^{\infty}(\mathcal{X})$ by
\begin{align*}
DA(u)[h] & = \frac{\int _{\mathcal{X}}h(x)e^{u(x)}\,d\nu _{\mathcal{X}}(x)}
{\int _{\mathcal{X}}e^{u(x)}\,d\nu _{\mathcal{X}}(x)},
\\*
D^{2}A(u)[h,h] & = \frac{\int _{\mathcal{X}}h(x)h'(x)e^{u(x)}\,d\nu _{\mathcal{X}}(x)}
{\int _{\mathcal{X}}e^{u(x)}\,d\nu _{\mathcal{X}}(x)}
- \frac{  (\int _{\mathcal{X}}h(x)e^{u(x)}\,d\nu _{\mathcal{X}}(x)  )^{2}}
{  (\int _{\mathcal{X}}e^{u(x)}\,d\nu _{\mathcal{X}}(x)  )^{2}}.
\end{align*}
In particular,
\begin{align}
DA\bigl(\Phi (\theta )\bigr)\bigl[\Phi \bigl(\theta '\bigr)\bigr] &
=\bigl\langle \Phi \bigl(\theta '\bigr),p_{
\theta }\bigr
\rangle _{L^{2}}, \label{eq:Density_Grad} 
\\
D^{2}A\bigl(\Phi (\theta )\bigr)\bigl[\Phi \bigl(\theta '
\bigr),\Phi \bigl(\theta '\bigr)\bigr] & =
\int \Phi \bigl( \theta '\bigr) (x)^{2}p_{\theta}(x)d
\nu _{\mathcal{X}}(x)- \biggl(
\int _{
\mathcal{X}}\Phi \bigl(\theta '\bigr)
(x)p_{\theta}(x)\,d\nu _{\mathcal{X}}(x) \biggr)^{2}
\nonumber\\[-8pt]\label{eq:Density_Hessian} 
\\[-8pt]
& =
\int _{\mathcal{X}} \bigl(\Phi \bigl(\theta '\bigr) (x)-
\bigl\langle \Phi \bigl(\theta '\bigr),p_{
\theta }\bigr\rangle
_{L^{2}} \bigr)^{2}p_{\theta}(x)\,d\nu
_{\mathcal{X}}(x).
\nonumber
\end{align}
Denote by $X$ a generic copy of $Z_{i}=X_{i}$. Observe for
$v\in \mathbb{R}^{p}$ the identities
\begin{align}
\ell (\theta ) & =\Phi (\theta ) (X)-A\bigl(\Phi (\theta )\bigr),\qquad
\mathbb{E}_{\theta _{0}}\ell (\theta )=\bigl\langle \Phi (\theta
),p_{
\theta _{0}}\bigr\rangle _{L^{2}}-A\bigl(\Phi (\theta )\bigr),
\nonumber
\\
v^{\top}\nabla \ell (\theta ) & =\Phi (v) (X)-\bigl\langle \Phi
(v),p_{
\theta }\bigr\rangle _{L^{2}},\qquad \mathbb{E}_{\theta _{0}}v^{\top}
\nabla \ell (\theta )=\bigl\langle \Phi (v),p_{\theta _{0}}-p_{\theta }
\bigr\rangle _{L^{2}},
\nonumber
\\
v^{\top}\nabla ^{2}\ell (\theta )v & =-D^{2}A\bigl(
\Phi (\theta )\bigr)\bigl[\Phi (v), \Phi (v)\bigr]. \label{eq:density_identities} 
\end{align}
Note that $v^{\top}\nabla ^{2}\ell (\theta )v$ is non-random. Due to the
Sobolev embedding \eqref{eq:sobolev_equivalence}, fixing $r>0$ yields a
constant $c_{r}>0$ with
\begin{equation}
\bigl\lVert \Phi (\theta )\bigr\rVert _{L^{\infty}}\leq c_{r},
\quad c_{r}^{-1} \leq p_{\theta}(x)\leq
c_{r}, x\in \mathcal{X},\lVert \theta \rVert _{\alpha}\leq
r. \label{eq:density_bounds} 
\end{equation}
Fix now $\theta $, $\theta '$ with
$\lVert \theta \rVert _{\alpha},\lVert \theta '\rVert _{\alpha}\leq r$.
Since the $e_{k}$ are the eigenfunctions of the Dirichlet Laplacian, it
follows from integrating by parts that
$\langle e_{k},1\rangle _{L^{2}}=0$, so
$\int _{\mathcal{X}}\Phi (\theta ')\,d\nu _{\mathcal{X}}=0$. Let
$c_{X}=\nu _{X}(\mathcal{X})$. Then, we have by (\ref{eq:density_bounds})
and the Cauchy-Schwarz inequality
\begin{equation*}
\int _{\mathcal{X}} \bigl(\Phi \bigl(\theta '\bigr) (x)-
\bigl\langle \Phi \bigl(\theta '\bigr),p_{
\theta }\bigr\rangle
_{L^{2}} \bigr)^{2}\,d\nu _{\mathcal{X}}(x)\leq \bigl\lVert
\Phi \bigl(\theta '\bigr)\bigr\rVert _{L^{2}}^{2}+c_{X}
\bigl\langle \Phi \bigl(\theta '\bigr),p_{
\theta }\bigr\rangle
_{L^{2}}^{2}\leq \bigl\lVert \theta '\bigr\rVert
^{2}+c_{X}c_{r} \bigl\lVert \theta
'\bigr\rVert ^{2},
\end{equation*}
and therefore
\begin{align}
c_{r}^{-1}\bigl\lVert \theta '\bigr\rVert
^{2} & \leq D^{2}A\bigl(\Phi (\theta )\bigr)\bigl[ \Phi
\bigl(\theta '\bigr),\Phi \bigl(\theta '\bigr)\bigr]\leq
c_{r}(1+c_{X}c_{r})\bigl\lVert \theta
'\bigr\rVert ^{2}. \label{eq:Density_Hessian_UpperLower} 
\end{align}
On the other hand, it follows from (\ref{eq:density_bounds}) and (\ref{eq:Density_Grad})
that $A(\Phi (\cdot ))$ is uniformly Lipschitz on the set
$\{\theta :\lVert \theta \rVert _{\alpha}\leq r\}$, because
\begin{align}
\begin{aligned}
\bigl\llvert A\bigl(\Phi (\theta )\bigr)-A\bigl(\Phi \bigl(\theta '
\bigr)\bigr) \bigr\rrvert & = \biggl\llvert
\int _{0}^{1}DA\bigl( \Phi \bigl(\theta
'+t\bigl(\theta -\theta '\bigr)\bigr)\bigr)\bigl[\Phi
\bigl(\theta -\theta '\bigr)\bigr]\,dt \biggr\rrvert
\\
& \leq \sup_{0\leq t\leq 1}\lVert p_{\theta '+t(\theta -\theta ')} \rVert
_{L^{\infty}}\bigl\lVert \Phi \bigl(\theta -\theta '\bigr)\bigr
\rVert _{L^{2}} \leq c_{r}\bigl\lVert \theta -\theta
'\bigr\rVert .
\end{aligned}
\label{eq:Density_Lipschitz} 
\end{align}
With these preparations let us verify the conditions of Theorem~\ref{thm:localisation} and Theorem~\ref{thm:GaussianContraction}.

\emph{Theorem~\ref{thm:GaussianContraction}(i).}\textbf{ }Consider
$r>0$ and $\theta \in \mathcal{B}_{n,r}$. It follows for
$\theta _{0}\in h^{\alpha}(\mathbb{N})$
\begin{align*}
\mathbb{E}_{\theta _{0}} \bigl(\ell (\theta _{0})-\ell (\theta )
\bigr) & =\bigl\langle \Phi (\theta _{0}-\theta ),p_{\theta _{0}}
\bigr\rangle _{L^{2}}- \bigl(A\bigl(\Phi (\theta _{0})\bigr)-A
\bigl(\Phi (\theta )\bigr) \bigr)
\\
& =DA\bigl(\Phi (\theta _{0})\bigr)\bigl[\Phi (\theta
_{0}-\theta )\bigr]- \bigl(A\bigl(\Phi ( \theta _{0})
\bigr)-A\bigl(\Phi (\theta )\bigr) \bigr)
\\
& =
\int _{0}^{1}(1-t)D^{2}A\bigl(\Phi
\bigl(\theta '+t\bigl(\theta -\theta '\bigr)\bigr)\bigr)
\bigl[ \Phi \bigl(\theta -\theta '\bigr),\Phi \bigl(\theta -\theta
'\bigr)\bigr]\,dt.
\end{align*}
The bias condition
$\lVert \theta _{0}-\theta _{*,p}\rVert \leq C\delta _{n}$ and the triangle
inequality imply
$\lVert \theta -\theta _{0}\rVert \lesssim \delta _{n}$. The first part
of inequality \eqref{eq:ApproxCondition} follows then from
\eqref{eq:Density_Hessian_UpperLower}. For the second part note that for
$q\geq 2$ and a constant $c_{A}$ by (\ref{eq:Density_Lipschitz})
\begin{align*}
\mathbb{E}_{\theta _{0}} \bigl\llvert \ell (\theta )-\ell (\theta
_{0}) \bigr\rrvert ^{q} &\leq c_{A}^{q-2}
\mathbb{E}_{\theta _{0}} \bigl\llvert \ell ( \theta )-\ell (\theta
_{0}) \bigr\rrvert ^{2}
\\
&  \lesssim c_{A}^{q-2} \bigl(\mathbb{E}_{\theta _{0}}
\bigl\llvert \Phi (\theta ) (X)-\Phi (\theta _{0}) (X) \bigr\rrvert
^{2}+ \bigl\llvert A\bigl(\Phi ( \theta )\bigr)-A\bigl(\Phi (\theta
_{0})\bigr) \bigr\rrvert ^{2} \bigr)
\\
&  \lesssim c_{A}^{q-2} \bigl(\bigl\lVert \Phi (
\theta -\theta _{0}) \bigr\rVert _{L^{2}}^{2}+\lVert
\theta -\theta _{0}\rVert ^{2} \bigr) \lesssim
c_{A}^{q-2}\lVert \theta -\theta _{0}\rVert
^{2}\lesssim c_{A}^{q-2} \delta
_{n}^{2}.
\end{align*}

\emph{Theorem~\ref{thm:GaussianContraction}(ii).}\textbf{ }Let $r>0$,
$\lVert \theta \rVert _{\alpha},\lVert \theta '\rVert _{\alpha}\leq r$. Arguing as in the
proof of the corresponding statement in Proposition~\ref{prop:ExpoSurrogate} we get $h^{2}(\theta ,\theta ')=2-2e^{-x}$ for a random variable $x$. Upper and lower bounding this non-random quantity gives by
\eqref{eq:Density_Hessian_UpperLower}
\begin{equation*}
(c_{r}^{-1}/2)\bigl\lVert \theta -\theta '\bigr
\rVert ^{2}\leq x\leq (c_{r}/2)(1+c_{X}c_{r})
\bigl\lVert \theta -\theta '\bigr\rVert ^{2}\leq
(c_{r}/2)(1+c_{X}c_{r}) (2r)^{2}.
\end{equation*}
Conclude by \eqref{eq:Birge}.

\emph{Theorem~\ref{thm:localisation}(ii).} Let
$\theta \in \mathcal{B}$, $v\in \mathbb{R}^{p}$ with
$\lVert v\rVert =1$. We have by the Sobolev embedding
\eqref{eq:Phi_C1} and \eqref{eq:GLM_conditions_1} the inequality
$\lVert \Phi (\theta )\rVert _{L^{\infty}}\lesssim 1$ and thus
$\sup_{\theta \in \mathcal{B}}\lVert D^{2}A(\Phi (\theta ))[\Phi (v),
\Phi (v)]\rVert _{L^{\infty}}<\infty $. It follows from the identities
\eqref{eq:density_identities}, \eqref{eq:Density_Hessian_UpperLower} and
the Sobolev embedding \eqref{eq:Phi_C1_Rp} for $\bar\alpha >1/2$
\begin{align*}
\bigl\llvert \mathbb{E}_{\theta _{0}}v^{\top}\nabla \ell (\theta
_{*,p}) \bigr\rrvert & = \bigl\llvert \bigl\langle \Phi
(v),p_{\theta _{0}}-p_{\theta _{*,p}}\bigr\rangle _{L^{2}} \bigr\rrvert
\leq \lVert p_{\theta _{0}}-p_{\theta _{*,p}}\rVert _{L^{2}},
\\
\bigl\llvert v^{\top}\nabla \ell (\theta ) \bigr\rrvert & \leq 2\bigl
\lVert \Phi (v)\bigr\rVert _{L^{
\infty}}\lesssim p^{\bar\alpha},
\\
\bigl\llvert v^{\top}\nabla ^{2}\ell (\theta )v \bigr\rrvert &
\lesssim 1.
\end{align*}
As in \eqref{eq:density_bounds} the log-likelihood function is uniformly
upper and lower bounded on $\mathcal{B}$, so
\begin{equation}
\lVert p_{\theta _{0}}-p_{\theta _{*,p}}\rVert _{L^{2}}^{2}
\lesssim \mathbb{E}_{\theta _{0}} \bigl\llvert e^{\ell (\theta _{0})}-e^{\ell (
\theta _{*,p})}
\bigr\rrvert ^{2}\lesssim \mathbb{E}_{\theta _{0}} \bigl\llvert \ell
(\theta _{0})-\ell (\theta _{*,p}) \bigr\rrvert
^{2}\lesssim \lVert \theta _{0}-\theta _{*,p}
\rVert \lesssim \delta _{n}^{2}. \label{eq:Density_pDiff} 
\end{equation}
Together with the last display this means for $q\geq 2$ and a possibly
different constant $c_{A}>0$
\begin{align*}
\mathbb{E}_{\theta _{0}} \bigl\llvert v^{\top}\nabla \ell (\theta
_{*,p}) \bigr\rrvert ^{q} & \leq \bigl\lVert
v^{\top }\nabla \ell (\theta _{*,p})\bigr\rVert
_{L^{\infty}}^{q-2} \mathbb{E}_{\theta _{0}} \bigl\llvert
v^{\top}\nabla \ell (\theta _{*,p}) \bigr\rrvert ^{2}
\\
& \leq c_{A}^{q-2}p^{\bar\alpha (q-2)}\bigl(2\bigl\lVert \Phi
(v)\bigr\rVert _{L^{2}}^{2}+2 \bigl\langle \Phi
(v),p_{\theta _{*,p}}\bigr\rangle ^{2}\bigr)\lesssim
c_{A}^{q-2}p^{
\bar\alpha (q-2)},
\\
\mathbb{E}_{\theta _{0}} \bigl\llvert v^{\top}\nabla ^{2}
\ell (\theta )v \bigr\rrvert ^{q} & \leq c_{A}^{q}.
\end{align*}
At last, if $\theta '\in \mathcal{B}$, then
\begin{align*}
\bigl\llvert v^{\top}\bigl(\nabla ^{2}\ell (\theta )-\nabla
^{2}\ell \bigl(\theta '\bigr)\bigr)v \bigr\rrvert &\lesssim
\int _{\mathcal{X}} \bigl(\Phi (v) (x)-\bigl\langle \Phi
(v),p_{
\theta }\bigr\rangle _{L^{2}} \bigr)^{2} \llvert
p_{\theta}-p_{\theta '} \rrvert (x)\,d\nu _{\mathcal{X}}(x)
\\
& \quad{} +
\int _{\mathcal{X}} \bigl\llvert \bigl(\Phi (v) (x)-\bigl\langle \Phi
(v),p_{\theta }\bigr\rangle _{L^{2}} \bigr)^{2}- \bigl(
\Phi (v) (x)- \bigl\langle \Phi (v),p_{\theta '}\bigr\rangle _{L^{2}}
\bigr)^{2} \bigr\rrvert d \nu _{\mathcal{X}}(x)
\\
&  \lesssim \bigl\lVert \Phi (v)\bigr\rVert _{L^{\infty}}\bigl\lVert
\Phi (v) \bigr\rVert _{L^{2}}\lVert p_{\theta }-p_{\theta '}
\rVert _{L^{2}}
\\
& \quad{} + \bigl\llvert \bigl\langle \Phi (v),p_{\theta '}-p_{\theta }
\bigr\rangle _{L^{2}} \bigr\rrvert
\int _{\mathcal{X}} \bigl\llvert 2\Phi (v) (x)-\bigl\langle \Phi
(v),p_{\theta }+p_{
\theta '}\bigr\rangle _{L^{2}} \bigr\rrvert
p_{\theta '}(x)\,d\nu _{\mathcal{X}}(x)
\\
&  \lesssim \lVert v\rVert _{\bar\alpha}\bigl\lVert \theta -\theta
' \bigr\rVert + \bigl\llvert \bigl\langle \Phi (v),p_{\theta '}-p_{\theta }
\bigr\rangle _{L^{2}} \bigr\rrvert \lesssim p^{\bar\alpha}\bigl\lVert
\theta -\theta '\bigr\rVert .
\end{align*}
We therefore verify the conditions in Theorem~\ref{thm:localisation}(ii)
with $\kappa _{1}=0$, $\kappa _{2}=\bar\alpha $,
$\kappa _{4}=\bar\alpha $.

\emph{Theorem~\ref{thm:localisation}(iii).} Use the identities (\ref{eq:density_identities})
and (\ref{eq:Density_Hessian_UpperLower}) to obtain the result with
$\kappa _{3}=0$ from
\begin{equation*}
\inf_{\theta \in \mathcal{B}}\mathbb{E}_{\theta _{0}} \bigl[-v^{\top}
\nabla ^{2}\ell (\theta )v \bigr]=\inf_{\theta \in \mathcal{B}}D^{2}A
\bigl( \Phi (\theta )\bigr)\bigl[\Phi (v),\Phi (v)\bigr]\gtrsim \lVert v\rVert
^{2}=1.
\end{equation*}

\emph{Theorem~\ref{thm:localisation}(iv).} It is easy to check that the
growth conditions and the lower bound on $\eta $ from Assumption~\ref{assu:a1} hold for all large enough $n$.
\end{proof}

\subsection{Darcy's problem}
\label{subsec:Darcy's-problem}

We recall first some facts for Sobolev spaces and relevant analytical properties
of the PDE (\ref{eq:PDE}). We then proceed by giving a stability estimate
for the forward operator $\mathcal{G}$ based on the condition
\eqref{eq:Darcy_stabilityCondition}, analyse further analytical properties
of $\mathcal{G}$ and conclude with the proof of Theorem~\ref{Thm:Darcy}. Details for the first two subsections can be found in
\citep[Section~5A]{taylormichael2010}. In the following,
$\nabla _{\theta}$ and $\nabla ^{2}_{\theta}$ denote the gradient and Hessian
with respect to the parameter $\theta $ and $\nabla $, $\nabla ^{2}$ with
respect to the space variable.

\subsubsection{Sobolev spaces}
\label{subsec:Sobolev}

For multi-indices $i=(i_{1},\dots ,i_{d})$ let $D^{i}$ be the weak partial
derivative operators. Denote the classical $L^{2}(\mathcal{X})$-Sobolev
spaces of integer order $s\geq 0$ by
\begin{equation*}
H^{s}(\mathcal{X})= \biggl\{ w\in L^{2}(\mathcal{X}):
\lVert w\rVert _{H^{s}}^{2}= \sum_{|i|\leq s}
\bigl\lVert D^{i}w\bigr\rVert _{L^{2}}^{2}<\infty
\biggr\} .
\end{equation*}
They satisfy a Sobolev embedding
\citep[Proposition~4.3]{taylormichael2010},
\begin{equation}
H^{s}(\mathcal{X})\subset C^{k}(\mathcal{X}),\quad
s>k+d/2. \label{eq:SobolevEmbedding} 
\end{equation}
Let $H_{0}^{s}(\mathcal{X})$ be the subspace of functions in
$H^{s}(\mathcal{X})$ that vanish on the boundary of $\mathcal{X}$ in the
trace sense. Their topological dual spaces are denoted by
$(H_{0}^{s}(\mathcal{X}))^{*}$. Another scale of Sobolev spaces
$\tilde{H}^{s}(\mathcal{X})$ is induced by the eigensystem
$(\lambda _{k},e_{k})_{k=1}^{n}$ of the negative Dirichlet Laplacian, where
\begin{equation*}
\tilde{H}^{s}(\mathcal{X})= \Biggl\{ f\in L^{2}(
\mathcal{X}):\lVert f \rVert _{\tilde{H}^{s}}^{2}=\sum
_{k=1}^{\infty}\lambda _{k}^{s}
\langle f,e_{k}\rangle _{L^{2}}^{2}<\infty \Biggr\} ,
\end{equation*}
which is equipped with the inner product
\begin{equation*}
\langle f,g\rangle _{\tilde{H}^{s}}=\sum_{k=1}^{\infty}
\lambda _{k}^{s} \langle f,e_{k}\rangle
_{L^{2}}\langle g,e_{k}\rangle _{L^{2}}.
\end{equation*}
Due to the presence of a boundary they generally differ from the Sobolev
spaces $H^{s}(\mathcal{X})$, but it can be shown that
\begin{equation}
\tilde{H}^{s}(\mathcal{X})=H_{0}^{s}(
\mathcal{X}),\quad s=1,2,\qquad \tilde{H}^{s}(\mathcal{X})\subset
H_{0}^{s}(\mathcal{X}),\quad s\in \mathbb{N},
\label{eq:normEquiv} 
\end{equation}
and the $\lVert \cdot \rVert _{H^{s}}$- and
$\lVert \cdot \rVert _{\tilde{H}^{s}}$-norms are equivalent on
$H^{s}(\mathcal{X})$. By Weyl's law the eigenvalues satisfy for
$0<c_{1}<c_{2}<\infty $
\begin{equation*}
c_{1}k^{2/d}\leq \llvert \lambda _{k} \rrvert
\leq c_{2}k^{2/d},\quad k\geq 1,
\end{equation*}
and hence the map
$\Phi :h^{s/d}(\mathbb{N})\rightarrow \tilde{H}^{s}(\mathcal{X})$,
$\Phi (\theta )=\sum_{k=1}^{\infty}\theta _{k}e_{k}$ is an isomorphism
with
\begin{equation}
c_{1}\lVert \theta \rVert _{s/d}^{2}\leq \bigl
\lVert \Phi (\theta )\bigr\rVert _{H^{s}}^{2} \leq
c_{2}\lVert \theta \rVert _{s/d}^{2}.
\label{eq:sobolev_equivalence} 
\end{equation}
It follows from the last three displays and the Sobolev embedding
\eqref{eq:SobolevEmbedding} that for $\gamma >k+d/2$,
$k\in \mathbb{N}\cup \{0\}$,
%
\begin{equation}
\bigl\lVert \Phi (\theta )\bigr\rVert _{C^{k}}\lesssim \bigl\lVert \Phi
(\theta ) \bigr\rVert _{H^{\gamma}}\lesssim \lVert \theta \rVert
_{\gamma /d}. \label{eq:Phi_C1} 
\end{equation}
In particular, if $\theta \in \mathbb{R}^{p}$, then
\begin{equation}
\bigl\lVert \Phi (\theta )\bigr\rVert _{C^{k}}\lesssim p^{\gamma /d}
\lVert \theta \rVert . \label{eq:Phi_C1_Rp} 
\end{equation}

\subsubsection{Some PDE facts}
\label{sec:PDE_facts}

For $f\in C^{1}(\mathcal{X})$ the divergence form operator
$\mathcal{L}_{f}$ takes functions in $H_{0}^{2}(\mathcal{X})$ to
$L^{2}(\mathcal{X})$. If $f$ is strictly positive on $\mathcal{X}$, then
it has (e.g., by \citep[Theorem~6.3.4]{evans2010}) a linear, continuous
inverse operator
$\mathcal{L}_{f}^{-1}:L^{2}(\mathcal{X})\rightarrow H_{0}^{2}(
\mathcal{X})$. In particular, we have
\begin{equation}
g_{2}=0\Rightarrow u_{f}=G(f)=\mathcal{L}_{f}^{-1}g_{1}.
\label{eq:G} 
\end{equation}
We require the following quantitative elliptic regularity estimates with
explicit constants depending on the conductivity.

\begin{lem}
\label{lem:L}%
Let $f\in C^{\gamma +1}(\mathcal{X})$, $\gamma \geq 0$, and
$w\in H^{\gamma +2}(\mathcal{X})$. Then
\begin{equation*}
\lVert \mathcal{L}_{f}w\rVert _{H^{\gamma}}\leq 2\lVert f\rVert
_{C^{
\gamma +1}}\lVert w\rVert _{H^{\gamma +2}}.
\end{equation*}
\end{lem}

\begin{proof}
It suffices to note that
\begin{equation*}
\lVert \mathcal{L}_{f}w\rVert _{H^{\gamma}}=\lVert f\Delta w+
\nabla f \cdot \nabla w\rVert _{H^{\gamma}}\leq 2\lVert f\rVert
_{C^{\gamma +1}} \lVert w\rVert _{H^{\gamma +2}}.
\end{equation*}\qedhere
\end{proof}

\begin{lem}
\label{lem:Linv}%
For $c>0$ consider $f\in C^{1}(\mathcal{X})$ with
$f\geq f_{\operatorname{min}}$, $\lVert f\rVert _{C^{1}}\leq c$. There
exists a constant $C\equiv C(f_{\operatorname{min}},c)$ such that the following
statements hold:
\begin{enumerate}[label=(\roman*)]
\item[{(i)}] $w\in L^{2}(\mathcal{X})$:
$\lVert \mathcal{L}_{f}^{-1}w\rVert _{H^{2}}\leq C\lVert w\rVert _{L^{2}}$,
\item[{(ii)}] $w\in (H_{0}^{2})^{*}(\mathcal{X})$:
$\lVert \mathcal{L}_{f}^{-1}w\rVert _{L^{2}}\leq C\lVert w\rVert _{(H_{0}^{2})^{*}}$,
\item[{(iii)}] $w\in H^{1}(\mathcal{X};\mathbb{R}^{d})$:
$\lVert \mathcal{L}_{f}^{-1}(\nabla \cdot w)\rVert _{H^{1}}\leq C
\lVert w\rVert _{L^{2}}$.
\end{enumerate}
\end{lem}
\begin{proof}
Parts (i) and (ii) follow from \citep[Lemmas 21 and 23]{nickl2020}. For
(iii) use duality to find for $z\in H_{0}^{1}(\mathcal{X})$
\begin{align*}
\lVert \mathcal{L}_{f}z\rVert _{(H_{0}^{1})^{*}} & =\sup
_{\varphi
\in H_{0}^{1},\lVert \varphi \rVert _{H^{1}}\leq 1} \bigl\llvert \langle \mathcal{L}_{f}z,
\varphi \rangle _{L^{2}} \bigr\rrvert \geq \bigl\llvert \langle
\mathcal{L}_{f}z,z\rangle _{L^{2}} \bigr\rrvert \lVert z\rVert
_{H^{1}}^{-1}
\\
& =\langle f\nabla z,\nabla z\rangle _{L^{2}}\lVert z\rVert
_{H^{1}}^{-1} \gtrsim \lVert \nabla z\rVert
_{L^{2}}^{2}\lVert z\rVert _{H^{1}}^{-1}
\gtrsim \lVert z\rVert _{H^{1}},
\end{align*}
concluding by the Poincar\'{e} inequality. Applying this to
$z=\mathcal{L}_{f}^{-1}(\nabla \cdot w)$ for
$w\in H^{1}(\mathcal{X};\mathbb{R}^{d})$ yields
\begin{equation*}
\bigl\lVert \mathcal{L}_{f}^{-1}(\nabla \cdot w)\bigr
\rVert _{H^{1}}\lesssim \lVert \nabla \cdot w\rVert _{(H_{0}^{1})^{*}}.
\end{equation*}
Since the partial derivative operators are bounded operators from
$H^{1}(\mathcal{X})$ to $L^{2}(\mathcal{X})$, the result follows by duality,
the divergence theorem and the Cauchy-Schwarz inequality such that
\begin{equation*}
\lVert \nabla \cdot w\rVert _{(H_{0}^{1})^{*}} =\sup_{\varphi \in H_{0}^{1}(
\mathcal{X}),\lVert \varphi \rVert _{H^{1}}\leq 1}
\biggl\llvert
\int _{
\mathcal{X}}(-\nabla \cdot \varphi )w\,d\nu _{\mathcal{X}}
\biggr\rrvert \leq \sup_{\varphi \in L^{2}(\mathcal{X}),\lVert \varphi \rVert _{L^{2}}
\leq 1} \biggl\llvert
\int _{\mathcal{X}}\varphi w\,d\nu _{\mathcal{X}} \biggr\rrvert =
\lVert w\rVert _{L^{2}}.
\end{equation*}\qedhere
\end{proof}
\begin{lem}
\label{lem:boundedSolutions}%
Let $\gamma >k+d/2$ for $k\in \mathbb{N}\cup \{0\}$ and let
$f\in H^{\gamma}(\mathcal{X})$, $f\geq f_{\operatorname{min}}$. For all
$c>0$ there exists a constant
$C\equiv C(\gamma ,f_{\operatorname{min}},\mathcal{X},g_{1},g_{2},c)$ such
that
\begin{equation*}
\sup_{\lVert f\rVert _{H^{\gamma}}\leq c}\lVert u_{f}\rVert _{H^{
\gamma +1}}\leq
C,\qquad \sup_{\lVert f\rVert _{H^{\gamma}}\leq c} \lVert u_{f}\rVert
_{C^{k+1}}\leq C.
\end{equation*}
\end{lem}
\begin{proof}
If $u_{f}=g_{2}=0$ on $\partial X$, then \citep[Lemma~23]{nickl2020} shows
the first inequality for a constant
$C\equiv C(\gamma ,f_{\operatorname{min}},\mathcal{X},g_{1})$ and the second
one follows from the Sobolev embedding
$H^{\gamma}(\mathcal{X})\subset C^{k}(\mathcal{X})$. For general
$g_{2}\in C^{\infty}(\partial \mathcal{X})$ we can assume without loss
of generality that it extends to a function in
$C^{\infty}(\mathcal{X})$ (e.g., by taking $g_{2}$ as the solution of the
PDE (\ref{eq:PDE}) for the standard Laplacian with $f\equiv 1$,
$g_{1}=0$, which is smooth, cf. \citep[Theorem~8.14]{gilbarg2001}) and
note that $\bar{u}_{f}=u_{f}-g_{2}$ solves the PDE (\ref{eq:PDE}) with
right-hand side $g_{1}=f-\mathcal{L}_{f}g_{2}$ and $\bar{u}_{f}=0$ on
$\partial \mathcal{X}$. Then what has been shown so far applies to
$\bar{u}_{f}$, and we obtain the second inequality (and thus also the first)
with
\begin{equation*}
\sup_{\lVert f\rVert _{H^{\gamma}}\leq c}\lVert u_{f}\rVert _{C^{k+1}}
\leq \sup_{\lVert f\rVert _{H^{\gamma}}\leq c}\lVert \bar{u}_{f} \rVert
_{C^{k+1}}+\lVert g_{2}\rVert _{C^{k+1}}\leq C+\lVert
g_{2} \rVert _{C^{k+1}}.
\end{equation*}\qedhere
\end{proof}

\subsubsection{A stability estimate}
\label{sec:stability}

Here we derive a stability lower bound on the difference of solutions
$u_{f}-u_{f}'$ from the condition \eqref{eq:Darcy_stabilityCondition}.

\begin{lem}
\label{lem:DarcyCondition}%
Let $f,f'\in C^{1}(\mathcal{X})$ with $f=f'$ on
$\partial \mathcal{X}$ and
$\lVert f\rVert _{C^{1}},\lVert f'\rVert _{C^{1}}\leq c$ for some
$c>0$ and suppose for $\mu $, $c'>0$ that
\begin{equation}
\inf_{x\in \mathcal{X}} \biggl(\frac{1}{2}\Delta
u_{f}(x)+\mu \bigl\lVert \nabla u_{f}(x)\bigr\rVert
_{\mathbb{R}^{d}}^{2} \biggr)\geq c'. \label{eq:Darcy_stability}
\end{equation}
Then there exists a constant
$C\equiv C(\gamma ,f_{\operatorname{min}},\mathcal{X},g_{1},g_{2},c)>0$
such that
\begin{enumerate}[label=(\roman*)]
\item[{(i)}] $h\in H_{0}^{1}(\mathcal{X})$:
$\lVert \mathcal{L}_{h}u_{f}\rVert _{L^{2}}\geq C\lVert h\rVert _{L^{2}}$,
\item[{(ii)}]
$\lVert f-f'\rVert _{L^{2}}\leq C\lVert u_{f}-u_{f'}\rVert _{H^{2}}$.\vadjust{\goodbreak}
\end{enumerate}
\end{lem}

\begin{proof}
For $h\in C_{c}^{\infty}(\mathcal{X})$ the claim in (i) follows from
\citep[Lemma~1]{nickl2022a} (it is easy to check that the requirement
$f\in C^{\infty}(\mathcal{X})$ in the proof can be reduced to
$f\in C^{1}(\mathcal{X})$), and extends by taking limits to
$h\in H_{0}^{1}(\mathcal{X})$. For (ii) the proof of
\citep[Proposition~3]{nickl2022a} applies: Take
$h=f-f'\in H_{0}^{1}(\mathcal{X})$ such that
$\mathcal{L}_{h}u_{f}=\mathcal{L}_{f'}(u_{f'}-u_{f})$ (cf. (\ref{eq:GDiff})
below) and hence by (i) and Lemma~\ref{lem:L}
\begin{equation*}
\bigl\lVert f-f'\bigr\rVert _{L^{2}}=\lVert h\rVert
_{L^{2}}\lesssim \lVert \mathcal{L}_{h}u_{f}
\rVert _{L^{2}}\lesssim \lVert u_{f}-u_{f'} \rVert
_{H^{2}}.
\end{equation*}\qedhere
\end{proof}

The stability bound is next extended to the forward operators.

\begin{lem}
\label{lem:Darcy_stability}%
Let $\theta ,\theta '\in h^{\alpha}(\mathbb{N})$, $d\alpha >1+d/2$, with
$\lVert \theta \rVert _{\alpha},\lVert \theta '\rVert _{\alpha}\leq c$
for some $c>0$, and suppose that (\ref{eq:Darcy_stability}) holds for
$f=f_{\theta}$. Then there exists a constant
$C\equiv C(\alpha ,f_{\operatorname{min}},\mathcal{X},g_{1},g_{2},c)>0$
such that
\begin{equation*}
\bigl\lVert \theta -\theta '\bigr\rVert ^{\beta}\leq C\bigl
\lVert \mathcal{G}( \theta )-\mathcal{G}\bigl(\theta '\bigr)\bigr
\rVert _{L^{2}},\qquad \beta = \frac{\alpha +d}{\alpha -d}.
\end{equation*}
\end{lem}

\begin{proof}
Let $\gamma =d\alpha $ such that $\beta =(\gamma +1)/(\gamma -1)$. Use
first $x\le e^{x}-1$ for $x\geq 0$ and (\ref{eq:Phi_C1}) to the extent
that
\begin{equation*}
\bigl\lVert \theta -\theta '\bigr\rVert =\bigl\lVert \Phi (\theta
)-\Phi \bigl(\theta '\bigr) \bigr\rVert _{L^{2}}\leq \bigl
\lVert e^{\Phi (\theta ')}\bigr\rVert _{L^{\infty}} \bigl\lVert
e^{\Phi (\theta )}-e^{\Phi (\theta ')}\bigr\rVert _{L^{2}}\lesssim \lVert
f_{\theta }-f_{\theta '}\rVert _{L^{2}}.
\end{equation*}
Apply Lemma~\ref{lem:DarcyCondition}(ii) to the last term. To conclude
observe for
$w=u_{f_{\theta}}-u_{f_{\theta '}}=\mathcal{G}(\theta )-\mathcal{G}(
\theta ')$ that $\lVert w\rVert _{H^{\gamma +1}}\leq C$ for a constant
depending on
$\alpha $, $f_{\operatorname{min}}$, $\mathcal{X}$, $g_{1}$, $g_{2}$, $c$ by
Lemma~\ref{lem:boundedSolutions} and (\ref{eq:Phi_C1}), and that by an interpolation
inequality for Sobolev spaces
\citep[Theorems 1.15 and 1.35]{yagi2009}
\begin{equation*}
\lVert w\rVert _{H^{2}}\lesssim \lVert w\rVert _{L^{2}}^{(\gamma -1)/(
\gamma +1)}
\lVert w\rVert _{H^{\gamma +1}}^{2/(\gamma +1)}\lesssim \lVert w\rVert
_{L^{2}}^{1/\beta}.
\end{equation*}\qedhere
\end{proof}

\subsubsection{Analytical properties of the forward operator}
\label{sec:analytical}

We study the forward operator $\mathcal{G}$ with the aim of verifying Assumption~\ref{assu:GLM_local_reg}. Note from the Sobolev embedding
\eqref{eq:Phi_C1} for $\lVert \theta \rVert _{\gamma /d}\leq r$ and a constant
$C_{r}>0$
\begin{equation}
\lVert f_{\theta }\rVert _{C^{k}}\leq f_{\operatorname{min}}+\bigl
\lVert e^{
\Phi (\theta )}\bigr\rVert _{C^{k}}\leq C_{r}.
\label{eq:f_theta_C1} 
\end{equation}

\begin{lem}
\label{lem:G_Lipschitz}%
Let $\theta ,\theta '\in h^{\alpha}(\mathbb{N})$, $d\alpha >1+d/2$, with
$\lVert \theta \rVert _{\alpha},\lVert \theta '\rVert _{\alpha}\leq c$
for some $c>0$. Then there exists a constant $C\equiv C(\alpha ,c)$ such
that
\begin{equation*}
\bigl\lVert \mathcal{G}(\theta )-\mathcal{G}\bigl(\theta '\bigr)
\bigr\rVert _{L^{2}}\leq C \bigl\lVert \theta -\theta '\bigr
\rVert .
\end{equation*}
\end{lem}

\begin{proof}
We have
$(\mathcal{G}(\theta )-\mathcal{G}(\theta '))(x)=g_{2}(x)-g_{2}(x)=0$ for
$x\in \partial X$ such that
$\mathcal{G}(\theta )-\mathcal{G}(\theta ')\in H_{0}^{2}(\mathcal{X})$
by Lemma~\ref{lem:boundedSolutions}, and
\begin{align*}
\mathcal{L}_{f_{\theta '}} \bigl(\mathcal{G}(\theta )-\mathcal{G}\bigl( \theta
'\bigr) \bigr) & =(\mathcal{L}_{f_{\theta '}}-\mathcal{L}_{f_{
\theta}})
\mathcal{G}(\theta )+\mathcal{L}_{f_{\theta}}\mathcal{G}( \theta )-
\mathcal{L}_{f_{\theta '}}\mathcal{G}\bigl(\theta '\bigr)
\\
& =\mathcal{L}_{f_{\theta '}-f_{\theta}}\mathcal{G}(\theta )+g_{1}-g_{1}=
\mathcal{L}_{f_{\theta '}-f_{\theta}}\mathcal{G}(\theta ).
\end{align*}
This allows for applying $\mathcal{L}_{f_{\theta '}}^{-1}$, and we get
\begin{equation}
\mathcal{G}(\theta )-\mathcal{G}\bigl(\theta '\bigr)=
\mathcal{L}_{f_{\theta '}}^{-1} \mathcal{L}_{f_{\theta '}-f_{\theta}}\mathcal{G}(
\theta ). \label{eq:GDiff} 
\end{equation}
Lemma~\ref{lem:Linv} combined with (\ref{eq:f_theta_C1}) yields
\begin{align*}
& \bigl\lVert \mathcal{G}(\theta )-\mathcal{G}\bigl(\theta '\bigr)
\bigr\rVert _{L^{2}} \leq C_{r}\bigl\lVert
\mathcal{L}_{f_{\theta '}-f_{\theta }}\mathcal{G}( \theta )\bigr\rVert _{(H_{0}^{2})^{*}}.
\end{align*}
By duality, the divergence theorem and writing
$\mathcal{G}(\theta )=u_{f_{\theta}}$ the last term equals
\begin{align*}
\sup_{\varphi \in H_{0}^{2},\lVert \varphi \rVert _{H^{2}}\leq 1} \biggl\llvert
\int _{\mathcal{X}}\varphi \nabla \cdot (f_{\theta '}-f_{
\theta})
\nabla u_{f_{\theta}} \biggr\rrvert & =\sup_{\varphi \in H_{0}^{2},
\lVert \varphi \rVert _{H^{2}}\leq 1} \biggl
\llvert
\int _{\mathcal{X}}(f_{
\theta '}-f_{\theta})\nabla \varphi
\cdot \nabla u_{f_{\theta}} \biggr\rrvert
\\
& \leq \lVert f_{\theta '}-f_{\theta }\rVert _{L^{2}}\sup
_{\lVert
\varphi \rVert _{H^{2}}\leq 1}\lVert \nabla \varphi \cdot \nabla u_{f_{
\theta }}
\rVert _{L^{2}}
\\*
& \leq \lVert f_{\theta '}-f_{\theta }\rVert _{L^{2}}\lVert
u_{f_{
\theta }}\rVert _{C^{1}}.
\end{align*}
The result follows then from Lemma~\ref{lem:boundedSolutions} and (\ref{eq:Phi_C1})
such that
\begin{equation*}
\lVert f_{\theta '}-f_{\theta }\rVert _{L^{2}} \lesssim \bigl
\lVert \Phi \bigl( \theta '-\theta \bigr)\bigr\rVert
_{L^{2}}=\bigl\lVert \theta '-\theta \bigr\rVert .
\end{equation*}\qedhere
\end{proof}

\begin{prop}
\label{prop:DarcyDeriv}%
Let $\theta \in \mathbb{R}^{p}$, $v\in \mathbb{R}^{p}$ and set
$f_{\theta ,v}=e^{\Phi (\theta )}\Phi (v)$,
$f_{\theta ,v,2}=e^{\Phi (\theta )}\Phi (v)^{2}$. Then we have for
$x\in \mathcal{X}$ the formulas
\begin{align*}
v^{\top}\nabla _{\theta}\mathcal{G}(\theta ) (x) & =- \bigl(
\mathcal{L}_{f_{\theta}}^{-1}\mathcal{L}_{f_{\theta ,v}}u_{f_{\theta}}
\bigr) (x),
\\
v^{\top}\nabla _{\theta}^{2}\mathcal{G}(\theta ) (x)v &
=2 \bigl( \mathcal{L}_{f_{\theta}}^{-1}\mathcal{L}_{f_{\theta ,v}}
\mathcal{L}_{f_{
\theta}}^{-1}\mathcal{L}_{f_{\theta ,v}}u_{f_{\theta}}
\bigr) (x)- \bigl(\mathcal{L}_{f_{\theta}}^{-1}\mathcal{L}_{f_{\theta ,v,2}}u_{f_{
\theta}}
\bigr) (x).
\end{align*}
\end{prop}
\begin{proof}
Let us write $G(f)=u_{f}$ such that
$\mathcal{G}(\theta )=G(f_{\theta})$. We will establish for
\begin{equation*}
G:H^{\alpha}(\mathcal{X})\cap \bigl\{f:f(x)>0,x\in \bar{\mathcal{X}}\bigr
\} \rightarrow C(\mathcal{X})
\end{equation*}
and $h,h'\in H^{\alpha}(\mathcal{X})$ as
$\lVert h\rVert _{H^{\alpha}}\rightarrow 0$ and
$\lVert h'\rVert _{H^{\alpha}}\rightarrow 0$, respectively, that
\begin{align}
\bigl\lVert G(f+h)-G(f)-A_{1}(f)[h]\bigr\rVert _{L^{\infty}} & =O
\bigl(\lVert h \rVert _{H^{\alpha}}^{2} \bigr), \label{eq:gradConv}
\\
\bigl\lVert A_{1}\bigl(f+h'\bigr)[h]-A_{1}(f)[h]-A_{2}(f)
\bigl[h,h'\bigr]\bigr\rVert _{L^{\infty}} & =O \bigl(\bigl\lVert
h'\bigr\rVert _{H^{\alpha}}^{2} \bigr), \label{eq:HessianConv}
\end{align}
with continuous linear operators
$A_{1}(f):H^{\alpha}(\mathcal{X})\rightarrow C(\mathcal{X})$,
$A_{2}(f):H^{\alpha}(\mathcal{X})\times H^{\alpha}(\mathcal{X})
\rightarrow C(\mathcal{X})$ given by
\begin{align}
A_{1}(f)[h] & =-\mathcal{L}_{f}^{-1}
\mathcal{L}_{h}G(f), \label{eq:DG} 
\\
A_{2}(f)\bigl[h,h'\bigr] & =\mathcal{L}_{f}^{-1}
\mathcal{L}_{h'}\mathcal{L}_{f}^{-1}
\mathcal{L}_{h}G(f)+\mathcal{L}_{f}^{-1}
\mathcal{L}_{h}\mathcal{L}_{f}^{-1}
\mathcal{L}_{h'}G(f). \label{eq:D2G} 
\end{align}
This implies that $G$ is two-times continuously Fr\'{e}chet differentiable
with derivatives $DG(f)=A_{1}(f)$, $D^{2}G(f)=A_{2}(f)$. Since the map
$\theta \mapsto f_{\theta}=f_{\operatorname{min}}+e^{\Phi (\theta )}$ in
(\ref{eq:f_theta}) satisfies on $\mathbb{R}^{p}$
\begin{align*}
v^{\top}\nabla _{\theta} f_{\theta} & =e^{\Phi (\theta )}
\Phi (v)=f_{
\theta ,v},\qquad v^{\top}\nabla _{\theta}^{2}f_{\theta}v=e^{\Phi (
\theta )}
\Phi (v)^{2}=f_{\theta ,v,2},
\end{align*}
the claim follows from the chain rule and from replacing
$G(f_{\theta})=u_{f_{\theta}}$.

Consider now $h,h'\in H^{\alpha}(\mathcal{X})$ with sufficiently small
$\lVert \cdot \rVert _{H^{\alpha}}$-norms such that $f+h$, $f+h'$ are strictly
positive on $\mathcal{X}$ and $G(f+h)$, $G(f+h')$, $G(f)$ are well-defined
with values in $H^{2}(\mathcal{X})$. Using (\ref{eq:GDiff}) twice we get
\begin{align}
\begin{aligned}
G(f+h)-G(f)-A_{1}(f)[h] & =-\mathcal{L}_{f}^{-1}
\mathcal{L}_{h} \bigl(G(f+h)-G(f) \bigr)
\\
& =\mathcal{L}_{f}^{-1}\mathcal{L}_{h}
\mathcal{L}_{f}^{-1}\mathcal{L}_{h}G(f+h).
\end{aligned}
\label{eq:firstDer} 
\end{align}
Since $G(f+h)=\mathcal{L}_{f+h}^{-1}g_{1}$, applying several times Lemmas
\ref{lem:Linv}(i) and \ref{assu:GLM_local_reg} and suppressing constants
depending on $\lVert f\rVert _{C^{1}}$, we get
\begin{equation*}
\bigl\lVert \mathcal{L}_{f}^{-1}\mathcal{L}_{h}
\mathcal{L}_{f}^{-1} \mathcal{L}_{h}G(f+h)\bigr
\rVert _{H^{2}}\lesssim \bigl(1+\lVert h\rVert _{C^{1}}\bigr)
\lVert h\rVert _{C^{1}}^{2},
\end{equation*}
implying (\ref{eq:gradConv}) by (\ref{eq:firstDer}) and the Sobolev embedding
$H^{2}(\mathcal{X})\subset L^{\infty}(\mathcal{X})$ for $d\leq 3$. Next,
we have
\begin{align*}
& \mathcal{L}_{f} \bigl(A_{1}\bigl(f+h'
\bigr)[h]-A_{1}(f)[h] \bigr)
\\
&\quad  =-\mathcal{L}_{f+h'-h'}\mathcal{L}_{f+h'}^{-1}
\mathcal{L}_{h}G\bigl(f+h'\bigr)+ \mathcal{L}_{f}^{-1}
\mathcal{L}_{h}G(f)
\\
&\quad  =\mathcal{L}_{h'}\mathcal{L}_{f+h'}^{-1}
\mathcal{L}_{h}G\bigl(f+h'\bigr)- \mathcal{L}_{h}
\bigl(G\bigl(f+h'\bigr)-G(f) \bigr)
\\
&\quad  =\mathcal{L}_{h'}\mathcal{L}_{f+h'}^{-1}
\mathcal{L}_{h}G\bigl(f+h'\bigr)+ \mathcal{L}_{h}
\mathcal{L}_{f}^{-1}\mathcal{L}_{h'}G
\bigl(f+h'\bigr).
\end{align*}
With this write
\begin{align*}
& A_{1}\bigl(f+h'\bigr)[h]-A_{1}(f)[h]-A_{2}(f)
\bigl[h,h'\bigr]
\\
&\quad  =\mathcal{L}_{f}^{-1}\mathcal{L}_{h'} \bigl(
\mathcal{L}_{f+h'}^{-1}- \mathcal{L}_{f}^{-1}
\bigr)\mathcal{L}_{h}G\bigl(f+h'\bigr)+
\mathcal{L}_{f}^{-1} \mathcal{L}_{h'}
\mathcal{L}_{f}^{-1}\mathcal{L}_{h} \bigl(G
\bigl(f+h'\bigr)-G(f) \bigr)
\\
& \qquad{} +\mathcal{L}_{f}^{-1}\mathcal{L}_{h}
\mathcal{L}_{f}^{-1} \mathcal{L}_{h'} \bigl(G
\bigl(f+h'\bigr)-G(f) \bigr)=:R_{1}+R_{2}+R_{3}.
\end{align*}
Arguing as after (\ref{eq:firstDer}) gives
\begin{align*}
\lVert R_{2}\rVert _{L^{\infty}} & =\bigl\lVert
\mathcal{L}_{f}^{-1} \mathcal{L}_{h'}
\mathcal{L}_{f}^{-1}\mathcal{L}_{h}
\mathcal{L}_{f}^{-1} \mathcal{L}_{h'}G
\bigl(f+h'\bigr)\bigr\rVert _{L^{\infty}}
\\
& \lesssim \bigl(1+\bigl\lVert h'\bigr\rVert _{C^{1}}\bigr)
\lVert h\rVert _{C^{1}}\bigl\lVert h' \bigr\rVert
_{C^{1}}^{2},
\end{align*}
and the same upper bound applies to
$\lVert R_{3}\rVert _{L^{\infty}}$. At last, for
$w\in L^{2}(\mathcal{X})$ observe that
\begin{equation}
\bigl(\mathcal{L}_{f+h'}^{-1}-\mathcal{L}_{f}^{-1}
\bigr)w=\mathcal{L}_{f+h'}^{-1}( \mathcal{L}_{f}-
\mathcal{L}_{f+h'})\mathcal{L}_{f}^{-1}w=-
\mathcal{L}_{f+h'}^{-1} \mathcal{L}_{h'}
\mathcal{L}_{f}^{-1}w, \label{eq:LinvDiff} 
\end{equation}
and so (\ref{eq:HessianConv}) follows from arguing as in the last display,
such that
\begin{equation*}
\lVert R_{1}\rVert _{L^{\infty}}\lesssim \bigl(1+\bigl\lVert
h'\bigr\rVert _{C^{1}}^{2}\bigr) \lVert h\rVert
_{C^{1}}\bigl\lVert h'\bigr\rVert _{C^{1}}^{2}.
\end{equation*}\qedhere
\end{proof}
%

\subsubsection{Proof of Theorem~\ref{Thm:Darcy}}
\label{sec4.4.5}

\begin{proof}
Below we will verify Assumption~\ref{assu:GLM_local_reg} with
$\beta =(\alpha +d)/(\alpha -d)$, $\eta =p^{-8/d}$ and $k_{1}=1/d$,
$k_{2}=7/d$, $k_{3}=0$, $k_{4}=2/d$, $k_{5}=3/d$. This implies
for $n$ large enough the growth conditions of Proposition~\ref{prop:ExpoSurrogate},
and thus verifies Assumptions~\ref{assu:a1},
\ref{assu:a3} with $\kappa _{1}=0$, $\kappa _{2}=2/d$,
$\kappa _{3}=6/d$. The result follows by arguing as in the proof of
Theorem~\ref{thm:GLM_guarantees}. We only note that
\begin{align*}
m_{\pi}=p,\qquad \Lambda _{\pi}=p^{2\alpha +1}, \qquad m
\geq c_{1}np^{-6/d}, \qquad \Lambda \leq c_{2}np^{2/d}.
\end{align*}
We proceed by verifying Assumption~\ref{assu:GLM_local_reg}.

\emph{Assumption~\ref{assu:GLM_local_reg}(i).} This follows
from Proposition~\ref{prop:DarcyDeriv}.

\emph{Assumption~\ref{assu:GLM_local_reg}(ii).} Let
$\theta \in \mathcal{B}$ and $v\in \mathbb{R}^{p}$,
$\lVert v\rVert =1$. For $\alpha \geq 21/d$ we find from the Lipschitz
bound in Lemma~\ref{lem:G_Lipschitz} that
\begin{align*}
\bigl\lVert \mathcal{G}(\theta )-\mathcal{G}(\theta _{*,p})\bigr\rVert
_{L^{2}} \lesssim \lVert \theta -\theta _{*,p}\rVert \leq \eta
.
\end{align*}
For $\alpha \geq 7/d$ we find from \eqref{eq:Phi_C1_Rp}
\begin{align}
\lVert \theta \rVert _{7/d} & \leq \lVert \theta -\theta
_{*,p} \rVert _{7/d}+\lVert \theta _{*,p}\rVert
_{7/d}\lesssim 1, \label{eq:theta_gamma} 
\end{align}
which in view of \eqref{eq:Phi_C1}, \eqref{eq:f_theta_C1} and
$\alpha /21\geq 7>5+d/2$ implies for a constant
$C\equiv C(f_{\operatorname{min}},c_{0})$ that
$\lVert \Phi (\theta )\rVert _{C^{5}}\leq C$,
$\lVert \Phi (v)\rVert _{H^{1}}\leq Cp^{1/d}$,
$\lVert \Phi (v)\rVert _{L^{\infty}}\leq Cp^{2/d}$,
$\lVert \Phi (v)\rVert _{C^{1}}\leq Cp^{3/d}$ and
\begin{align*}
\sup_{\theta \in \mathcal{B}}\lVert f_{\theta }\rVert _{H^{5}}
\lesssim \sup_{\theta \in \mathcal{B}}\lVert f_{\theta }\rVert
_{C^{5}} & \leq f_{\operatorname{min}}+\sup_{\theta \in \mathcal{B}}\bigl
\lVert e^{
\Phi (\theta )}\bigr\rVert _{C^{5}}\leq C.
\end{align*}
In particular, by Lemma~\ref{lem:boundedSolutions} for $C'>0$,
$\sup_{\theta \in \mathcal{B}}\lVert u_{f_{\theta }}\rVert _{C^{4}}
\leq C'$.

Next, recall the gradient and Hessian of $\mathcal{G}$ from Proposition~\ref{prop:DarcyDeriv}. For the sup-norm bounds observe first by Lemma~\ref{lem:Linv}(i)
\begin{align*}
\bigl\lVert \mathcal{L}_{f_{\theta }}^{-1}\mathcal{L}_{f_{\theta ,v}}u_{f_{
\theta }}
\bigr\rVert _{H^{2}} & \lesssim \bigl\lVert \nabla \cdot
\bigl(e^{
\Phi (\theta )}\Phi (v)\nabla u_{f_{\theta }} \bigr)\bigr\rVert
_{L^{2}}
\\
& \lesssim \bigl\lVert e^{\Phi (\theta )}\bigr\rVert _{C^{1}}\bigl\lVert
\Phi (v) \bigr\rVert _{H^{1}}\lVert u_{f_{\theta }}\rVert
_{C^{2}}\lesssim p^{1/d},
\end{align*}
and similarly
\begin{align*}
\bigl\lVert \mathcal{L}_{f_{\theta }}^{-1}\mathcal{L}_{f_{\theta ,v}}
\mathcal{L}_{f_{\theta }}^{-1}\mathcal{L}_{f_{\theta ,v}}u_{f_{
\theta }}
\bigr\rVert _{H^{2}} & \lesssim \bigl\lVert \nabla \cdot
\bigl(e^{
\Phi (\theta )}\Phi (v)\nabla \mathcal{L}_{f_{\theta }}^{-1}
\mathcal{L}_{f_{\theta ,v}}u_{f_{\theta }} \bigr)\bigr\rVert _{L^{2}}
\\
& \lesssim \bigl\lVert e^{\Phi (\theta )}\bigr\rVert _{C^{1}}\bigl\lVert
\Phi (v) \bigr\rVert _{C^{1}}\bigl\lVert \mathcal{L}_{f_{\theta }}^{-1}
\mathcal{L}_{f_{
\theta ,v}}u_{f_{\theta }}\bigr\rVert _{H^{2}}\lesssim
p^{4/d},
\\
\bigl\lVert \mathcal{L}_{f_{\theta }}^{-1}\mathcal{L}_{f_{\theta ,v,2}}u_{f_{
\theta }}
\bigr\rVert _{H^{2}} & \lesssim \bigl\lVert \nabla \cdot
\bigl(e^{
\Phi (\theta )}\Phi (v)^{2}\nabla u_{f_{\theta }} \bigr)\bigr
\rVert _{L^{2}}
\\
& \lesssim \bigl\lVert e^{\Phi (\theta )}\bigr\rVert _{C^{1}}\bigl\lVert
\Phi (v) \bigr\rVert _{C^{1}}\bigl\lVert \Phi (v)\bigr\rVert
_{H^{1}}\lVert u_{f_{\theta }} \rVert _{C^{2}}\lesssim
p^{4/d}.
\end{align*}
The Sobolev embedding $H^{2}(\mathcal{X})\subset C(\mathcal{X})$ in
$d\leq 3$ therefore shows
\begin{align*}
\bigl\lVert \mathcal{G}(\theta )\bigr\rVert _{L^{\infty}}\lesssim 1,\qquad
\bigl\lVert v^{\top }\nabla \mathcal{G}(\theta )\bigr\rVert
_{L^{\infty}} \lesssim p^{1/d},\qquad \bigl\lVert v^{\top }
\nabla ^{2}\mathcal{G}(\theta )v \bigr\rVert _{L^{\infty}}\lesssim
p^{4/d}.
\end{align*}
Next, for $\theta '\in \mathcal{B}$, $h,h'\in C^{1}(\mathcal{X})$ write
\begin{align*}
\mathcal{L}_{f_{\theta}}^{-1}-\mathcal{L}_{f_{\theta '}}^{-1}
& = \mathcal{L}_{f_{\theta}}^{-1}\mathcal{L}_{f_{\theta}-f_{\theta '}}
\mathcal{L}_{f_{\theta '}}^{-1},\qquad \mathcal{L}_{h}-
\mathcal{L}_{h'}= \mathcal{L}_{h-h'}=\nabla \cdot
\bigl(h-h'\bigr)\nabla ,
\end{align*}
such that Lemma~\ref{lem:Linv}(i) implies for
$w\in L^{2}(\mathcal{X})$, $w'\in H^{2}(\mathcal{X})$
\begin{align*}
\bigl\lVert \bigl(\mathcal{L}_{f_{\theta }}^{-1}-
\mathcal{L}_{f_{\theta '}}^{-1}\bigr)w \bigr\rVert _{H^{2}} &
\lesssim \bigl\lVert e^{\Phi (\theta )}-e^{\Phi (\theta ')} \bigr\rVert
_{C^{1}}\bigl\lVert \mathcal{L}_{f_{\theta '}}^{-1}w\bigr
\rVert _{H^{2}},
\\
\bigl\lVert (\mathcal{L}_{f_{\theta ,v}}-\mathcal{L}_{f_{\theta ',v}})w'
\bigr\rVert _{L^{2}} & \lesssim \bigl\lVert e^{\Phi (\theta )}-e^{\Phi (\theta ')}
\bigr\rVert _{C^{1}}\bigl\lVert \Phi (v)\bigr\rVert _{C^{1}}\bigl
\lVert w'\bigr\rVert _{H^{2}},
\\
\bigl\lVert (\mathcal{L}_{f_{\theta ,v}}-\mathcal{L}_{f_{\theta ',v}})u_{f_{
\theta }}
\bigr\rVert _{L^{2}} & \lesssim \bigl\lVert e^{\Phi (\theta )}-e^{
\Phi (\theta ')}
\bigr\rVert _{C^{1}}\bigl\lVert \Phi (v)\bigr\rVert _{H^{1}}\lVert
u_{f_{
\theta }}\rVert _{C^{2}},
\\
\bigl\lVert (\mathcal{L}_{f_{\theta ,v,2}}-\mathcal{L}_{f_{\theta ',v,2}})u_{f_{
\theta }}
\bigr\rVert _{L^{2}} & \lesssim \bigl\lVert e^{\Phi (\theta )}-e^{
\Phi (\theta ')}
\bigr\rVert _{C^{1}}\bigl\lVert \Phi (v)\bigr\rVert _{C^{1}}\bigl
\lVert \Phi (v)\bigr\rVert _{H^{1}}\lVert u_{f_{\theta }}\rVert
_{C^{2}}.
\end{align*}
Then,
$\lVert e^{\Phi (\theta )}-e^{\Phi (\theta ')}\rVert _{C^{1}}
\lesssim \lVert \Phi (\theta -\theta ')\rVert _{C^{1}}$, and so the terms
in the last display are upper bounded up to constants by
$p^{3/d}\lVert \theta -\theta '\rVert \lVert w\rVert _{L^{2}}$,
$p^{6/d}\lVert \theta -\theta '\rVert \lVert w\rVert _{H^{2}}$,
$p^{4/d}\lVert \theta -\theta '\rVert $ and
$p^{4/d}\lVert \theta -\theta '\rVert $, respectively. Combining these
estimates with Lemmas \ref{lem:L}(ii), \ref{lem:Linv}(i) and with the Lipschitz
bound from Lemma~\ref{lem:G_Lipschitz} we obtain
\begin{align*}
\bigl\lVert \nabla ^{2}\mathcal{G}(\theta )-\nabla ^{2}
\mathcal{G}\bigl(\theta '\bigr) \bigr\rVert _{L^{\infty}(\mathcal{X},\mathbb{R}^{p\times p})} &
\lesssim p^{7/d} \bigl\lVert \theta -\theta '\bigr\rVert .
\end{align*}
Regarding the wanted $L^{2}$-bounds, use Lemma~\ref{lem:Linv}(iii) to the
extent that
\begin{equation*}
\bigl\lVert \mathcal{L}_{f_{\theta }}^{-1}\mathcal{L}_{f_{\theta ,v}}u_{f_{
\theta }}
\bigr\rVert _{H^{1}}\lesssim \bigl\lVert e^{\Phi (\theta )}\Phi (v) \nabla
u_{f_{\theta }}\bigr\rVert _{L^{2}}\lesssim \bigl\lVert
e^{\Phi (\theta )} \bigr\rVert _{L^{\infty}}\bigl\lVert \Phi (v)\bigr\rVert
_{L^{2}}\lVert u_{f_{
\theta }}\rVert _{C^{1}}\lesssim 1,
\end{equation*}
as well as
\begin{align*}
\bigl\lVert \mathcal{L}_{f_{\theta }}^{-1}\mathcal{L}_{f_{\theta ,v}}
\mathcal{L}_{f_{\theta }}^{-1}\mathcal{L}_{f_{\theta ,v}}u_{f_{
\theta }}
\bigr\rVert _{H^{1}} &\lesssim \bigl\lVert e^{\Phi (\theta )}\Phi (v) \nabla
\bigl(\mathcal{L}_{f_{\theta }}^{-1}\mathcal{L}_{f_{\theta ,v}}u_{f_{
\theta }}
\bigr)\bigr\rVert _{L^{2}}
\\
& \lesssim \bigl\lVert e^{\Phi (\theta )}\bigr\rVert _{L^{\infty}}\bigl
\lVert \Phi (v)\bigr\rVert _{L^{\infty}}\bigl\lVert \mathcal{L}_{f_{\theta }}^{-1}
\bigl(\nabla \cdot \bigl(e^{\Phi (\theta )}\Phi (v)\nabla u_{f_{
\theta }} \bigr)
\bigr)\bigr\rVert _{H^{1}}
\\
& \lesssim p^{2/d}\bigl\lVert e^{\Phi (\theta )}\Phi (v)\nabla
u_{f_{
\theta }}\bigr\rVert _{L^{2}}\lesssim p^{2/d}\bigl
\lVert \Phi (v)\bigr\rVert _{L^{2}} \lVert u_{f_{\theta }}\rVert
_{C^{1}}\lesssim p^{2/d},
\\
\bigl\lVert \mathcal{L}_{f_{\theta }}^{-1}\mathcal{L}_{f_{\theta ,v,2}}u_{f_{
\theta }}
\bigr\rVert _{H^{1}} & \lesssim \bigl\lVert e^{\Phi (\theta )}\Phi
(v)^{2} \nabla u_{f_{\theta }}\bigr\rVert _{L^{2}}
\\
& \lesssim \bigl\lVert e^{\Phi (\theta )}\bigr\rVert _{L^{\infty}}\bigl
\lVert \Phi (v)\bigr\rVert _{L^{\infty}}\bigl\lVert \Phi (v)\bigr\rVert
_{L^{2}}\lVert u_{f_{
\theta }}\rVert _{C^{1}}\lesssim
p^{2/d}.
\end{align*}
In all, since
$\lVert \cdot \rVert _{L^{2}}\leq \lVert \cdot \rVert _{H^{1}}$, Assumption~\ref{assu:GLM_local_reg}(ii) holds with the $k_{i}$, $i=1,\dots ,4$ as
claimed.

\emph{Assumption~\ref{assu:GLM_local_reg}(iii).} An interpolation inequality
for Sobolev spaces (see e.g.,
\citep[Theorems 1.15 and 1.35]{yagi2009}) yields
\begin{equation*}
\lVert w\rVert _{H^{2}}\lesssim \lVert w\rVert _{L^{2}}^{1/2}
\lVert w \rVert _{H^{4}}^{1/2},\quad 0\neq w\in H^{4}(
\mathcal{X}).
\end{equation*}
Applying this to
$w=\mathcal{L}_{f_{\theta}}^{-1}\mathcal{L}_{f_{\theta ,v}}u_{f_{
\theta}}$ and observing the inequalities
$\lVert \mathcal{L}_{f}w\rVert _{L^{2}}\lesssim \lVert w\rVert _{H^{2}}$,
$\lVert w\rVert _{H^{4}}\lesssim \lVert \mathcal{L}_{f}w\rVert _{H^{2}}$
shows
\begin{align*}
\bigl\lVert v^{\top }\nabla \mathcal{G}(\theta )\bigr\rVert
_{L^{2}}=\lVert w \rVert _{L^{2}} & \gtrsim \frac{\lVert \mathcal{L}_{f_{\theta ,v}}u_{f_{\theta }}\rVert _{L^{2}}^{2}}{\lVert \mathcal{L}_{f_{\theta ,v}}u_{f_{\theta }}\rVert _{H^{2}}}.
\end{align*}
Recall the stability estimate from Lemma~\ref{lem:DarcyCondition}(ii),
which yields for
$h=e^{\Phi (\theta )}\Phi (v)\in H_{0}^{1}(\mathcal{X})$ (here we use that
$\Phi (\theta )\in H_{0}^{1}(\mathcal{X})$)
\begin{align*}
\lVert \mathcal{L}_{f_{\theta ,v}}u_{f_{\theta }}\rVert _{L^{2}}^{2}
& \gtrsim \bigl\lVert e^{\Phi (\theta )}\Phi (v)\bigr\rVert _{L^{2}}^{2}
\gtrsim e^{-2
\lVert \Phi (\theta )\rVert _{L^{\infty}}}\bigl\lVert \Phi (v)\bigr\rVert _{L^{2}}^{2}
\gtrsim 1,
\end{align*}
while Lemma~\ref{lem:boundedSolutions} and (\ref{eq:Phi_C1}) provide us
with the upper bound
\begin{align*}
\lVert \mathcal{L}_{f_{\theta ,v}}u_{f_{\theta }}\rVert _{H^{2}} &
\leq \bigl\lVert e^{\Phi (\theta )}\Phi (v)\bigr\rVert _{H^{3}}\lVert
u_{f_{
\theta }}\rVert _{C^{4}}\lesssim \bigl\lVert e^{\Phi (\theta )}
\bigr\rVert _{C^{3}} \bigl\lVert \Phi (v)\bigr\rVert _{H^{3}}
\lesssim p^{3/d}.
\end{align*}
The last three displays yield the wanted lower bound with
$k_{5}=3/d$.

\emph{Assumption~\ref{assu:GLM_local_reg}(iv).} Since
$d\alpha \geq 7>1+d/2$, use Lemmas \ref{lem:Darcy_stability} and
\ref{lem:G_Lipschitz} for the stated $\beta $.
\end{proof}

\bibliographystyle{apalike2}
\bibliography{refs}

\end{document}